\documentclass{amsart}
\usepackage[margin=1in]{geometry}
\usepackage{amsmath, amsfonts, amssymb}
\usepackage{color}
\usepackage{siunitx}
\usepackage{graphicx}
\usepackage{subfig}
\usepackage[normalem]{ulem}
\graphicspath{{images/}}

\newcommand\bg{\boldsymbol{g}}
\newcommand\bh{\boldsymbol{h}}
\newcommand\bvphi{\boldsymbol{\varphi}}
\newcommand\bx{\boldsymbol{x}}
\newcommand\bv{\boldsymbol{v}}
\newcommand\bu{\boldsymbol{u}}
\newcommand\bn{\boldsymbol{n}}

\newcommand\Kn{\mathit{Kn}}
\newcommand{\mM}{\mathcal{M}}
\newcommand{\Qabs}{Q^{\mathrm{abs}}}

\newtheorem{theorem}{Theorem}
\newtheorem{lemma}[theorem]{Lemma}
\newtheorem{corollary}[theorem]{Corollary}

\theoremstyle{remark}

\title{Numerical solver for the Boltzmann equation with self-adaptive collision operators}

\author{Zhenning Cai}
\address[Zhenning Cai]{Department of Mathematics, National University of Singapore,
  Level 4, Block S17, 10 Lower Kent Ridge Road, Singapore 119076}
\email{matcz@nus.edu.sg}

\author{Yanli Wang}
\address[Yanli Wang]{Beijing
    Computational Science Research Center, Beijing, China 100193}
\email{ylwang@csrc.ac.cn}

\thanks{Zhenning Cai's work was supported by the Academic Research Fund of the Ministry of Education of Singapore under Grant No. R-146-000-305-114. The work of Yanli Wang is partially supported by Science Challenge Project (No. TZ2016002) and the National Natural Science Foundation of China (Grant No. U1930402 and 12031013).}


\begin{document}
%%% End-Of-Header
\maketitle

%%% Start-Of-Trailer
\begin{abstract}
We use the Burnett spectral method to solve the Boltzmann equation whose collision term is modeled by separate treatments for the low-frequency part and high-frequency part of the solution. For the low-frequency part representing the sketch of the distribution function, the binary collision is applied, while for the high-frequency part representing the finer details, the BGK approximation is applied.
The parameter controlling the ratio of the high-frequency part and the low-frequency part is selected adaptively on every grid cell at every time step. This self-adaptation is based on an error indicator describing the difference between the model collision term and the original binary collision term. The indicator is derived by controlling the quadratic terms in the modeling error with linear operators. Our numerical experiments show that such an error indicator is effective and computationally affordable.

\vspace{3pt}

\end{abstract}

\section{Introduction} \label{sec:intro}
Due to the extensive applications of rarefied gas dynamics in a number of engineering fields, including the manufacturing of spacecrafts and micro-electro-mechanical systems, the numerical simulation of gas kinetic theory is under active research in recent years. In the kinetic theory, the fluid state is described using the distribution function $f(\bx,\bv,t)$, where $t$ is the time, and $\bx$ and $\bv$ represent the spatial coordinates and the velocity of gas molecules, respectively. The distribution function represents the number density of gas molecules in the joint position-velocity space. In this paper, we consider the Boltzmann equation:
\begin{displaymath}
\frac{\partial f}{\partial t} +
  \bv \cdot \nabla_{\bx} f = Q[f,f],
\end{displaymath}
where $Q[f,f]$ is the binary collision term defined by:
\begin{displaymath}
Q[f,g](\bv) = \frac{1}{2} \int_{\mathbb{R}^3} \int_{\bn \perp \bg} \int_0^{\pi} B(|\bg|, \chi)
  [f(\bv') g(\bv_*') + f(\bv_*') g(\bv') - f(\bv) g(\bv_*) - f(\bv_*) g(\bv)]
  \,\mathrm{d}\chi \,\mathrm{d}\bn \,\mathrm{d}\bv_*.
\end{displaymath}
In the equation above, the relative velocity $\bg$ is defined by $\bg = \bv - \bv_*$, and the post-collisional velocities $\bv'$ and $\bv_*'$ are given by
\begin{align*}
\bv' &= \cos^2(\chi/2) \bv + \sin^2(\chi/2) \bv_* - |\bg| \sin(\chi/2) \cos(\chi/2) \bn, \\
\bv_*' &= \cos^2(\chi/2) \bv_* + \sin^2(\chi/2) \bv + |\bg| \sin(\chi/2) \cos(\chi/2) \bn.
\end{align*}
Note that here $\bn$ is a unit vector in $\mathbb{S}^2$, which implies that the integral with respect to $\bn$ is a one-dimensional integral over a circle perpendicular to $\bg$. The non-negative function $B(\cdot,\cdot)$ is the collision kernel determined by the mutual force between gas molecules.

One of the numerical difficulties in the discretization of the Boltzmann equation lies in the high-dimensional integral form of $Q[f,f]$. To compute the collision term efficiently, the velocity variable in the distribution function is usually discretized by high-order schemes such as the spectral methods \cite{Pareschi1996,Bobylev1997} and discontinuous Galerkin methods \cite{Alekseenko2014}, so that the number of degrees of freedom can be reduced. In the literature, the spectral methods mainly include the Fourier spectral method \cite{Bobylev1997,Filbet2006,Gamba2017} based on the periodization of the velocity variable and the Hermite/Burnett spectral method based on the unbounded velocity domain \cite{Gamba2018, Kitzler2019}. The Fourier spectral method provides a significant improvement  in computational efficiency \cite{Dimarco2018,Jaiswal2019}, and the recent development of the Hermite/Burnett spectral method shows its advantage due to its connection with modelling in the gas kinetic theory. Specifically, the Hermite/Burnett spectral method can be linked to the moment method since the coefficients in the spectral expansion are actually the moments of the distribution functions. Such a property has been applied to derive the regularized 13-moment equations in \cite{Cai2020r}, and inversely, some modelling techniques can therefore also be applied to the spectral methods. In \cite{Cai2015}, by taking the idea of the Shakhov operator \cite{Shakhov1968}, the authors divided all moments of the distribution function into two sets, with one set including low-order moments describing the sketch of the distribution function, and the other including high-order moments providing the details. For the set with low-order moments, the linearized collision operator is applied, while for the set with high-order moments, a simple decay towards the equilibrium is used as an approximation. This hybrid approach is later extended to quadratic collision operators in \cite{Wang2019, Cai2020}. One parameter in this hybrid approach is the critical order $M_0$ that defines the ``low-order'' and ``high-order'' moments. In this paper, we will focus on the selection of this parameter in the spatially inhomogeneous Boltzmann equation.

Since the parameter $M_0$ defines the modeling accuracy, it is expected that the choice of $M_0$ should depend on the ``modeling error'' given by some differences between the current collision model and the exact binary collision model when applied to the current distribution function. Once such an error indicator is obtained, we can change the value of $M_0$ dynamically during our simulation. However, the construction of such an error indicator is far from trivial due to the following reasons:
\begin{enumerate}
\item Unlike the {\it a posteriori} error estimation in the finite element methods, we do not have an equation to define the ``residual'' as an error indicator. 
\item The collision operator is generally unbounded, so that even an {\it a priori} error estimation is non-trivial.
\item Another common technique by comparing the current model and a more accurate model with larger $M_0$ is not applicable here due to the rapid growth of the computational cost with respect to $M_0$ (usually $M_0^8$).
\end{enumerate}
Because of these difficulties, we have to look for non-standard techniques to quantify the error. Since a rigorous and numerically affordable error bound is difficult to find, as an initial study, the goal of this paper is to establish an error indicator with low computational cost compared to the collision term. With this error indicator, we are able to choose this modeling parameter $M_0$ adaptively on each spatial grid cell at each time step with the purpose to reduce the computational time on the collision terms. Due to the high computational complexity with respect to $M_0$, reducing $M_0$ can effectively save the computational cost.

This work contributes to the adaptive methods for the Boltzmann equation. In the literature, the self-adaptive methods have been applied to both spatial discretization and velocity discretization \cite{Kolobov2007,Chen2012,Baranger2014,Abdelmalik2017}, which can effectively reduce the degrees of freedom in the simulation. There have been also many works coupling the kinetic equations and fluid equations, so that the cheaper Navier-Stokes equations or Euler equations can be solved where the fluid is close to its local equilibrium \cite{Degond2005, Degond2012, Rey2015}, and many criteria have been proposed to predict the breakdown of fluid equations \cite{Levermore1998,Wang2003,Filbet2018}. While the method in this work does not change the number of variables, we consider the modeling adaptivity, which changes the complexity of the collision model. We hope that this work can provide a new perspective for the simulation of Boltzmann equations, which might also be applicable in other related areas.

In the rest of this paper, we will first review the Burnett spectral method introduced in \cite{Hu2020b} (Section \ref{sec:Burnett}), and then in Section \ref{sec:indicator}, we will detail the derivation of the error indicator, and the general structure of our numerical algorithm will also be presented. One- and two-dimensional numerical experiments showing the efficiency of the self-adaptive method will be given in Section \ref{sec:num}, and the paper is concluded by a brief summary in Section \ref{sec:conclusion}.

% vim: tw=70:spell

\section{Burnett spectral method for the Boltzmann equation}
\label{sec:Burnett}
The Burnett spectral method is based on the following expansion of the distribution function:
\begin{equation} \label{eq:inf_series}
f(\bx,\bv,t) = \sum_{l=0}^{+\infty} \sum_{m=-l}^l \sum_{n=0}^{+\infty} f_{lmn}(\bx,t) \varphi_{lmn}(\bv),
\end{equation}
where $\varphi_{lmn}(\bv) = \varphi_{lmn}^0(\bv - \overline{\bu})$, and
\begin{multline*}
\varphi_{lmn}^0(\bv) = \sqrt{\frac{2^{1-l}\pi^{3/2}n!}{\Gamma(n+l+3/2)}} L^{(l+1/2)}_n \left( \frac{|\bv|^2}{2\bar{\theta}} \right) \left( \frac{|\bv|}{\sqrt{\bar{\theta}}} \right)^l Y_l^m \left( \frac{\bv}{|\bv|} \right) \cdot \frac{1}{(2\pi \overline{\theta})^{3/2}} \exp \left(-\frac{|\bv|^2}{2\overline{\theta}} \right), \\
l,n = 0,1,\cdots, \qquad m = -l,\cdots,l,
\end{multline*}
with $L_n^{(\alpha)}(\cdot)$ being the Laguerre polynomials and $Y_l^m(\cdot)$ being the spherical harmonics. The parameters $\overline{\bu}$ and $\overline{\theta}$ are chosen to specify the center and the scaling of the basis functions. Note that $\varphi_{lmn}$ is the product of a Gaussian and a polynomial of degree $l+2n$. When we truncate the series \eqref{eq:inf_series} in our numerical method, we select a positive integer $M$ as the upper bound of the polynomial, and preserve only the terms with $l+2n \leqslant M$. Thus, the truncated series reads
\begin{equation} \label{eq:truncated_series}
f_M(\bx,\bv,t) = \sum_{l=0}^M \sum_{m=-l}^l \sum_{n=0}^{\lfloor (M-l)/2 \rfloor} f_{lmn}(\bx,t) \varphi_{lmn}(\bv).
\end{equation}
In this expansion, the basis functions $\varphi_{lmn}$ satisfy the orthogonality
\begin{equation} \label{eq:orth}
\langle \varphi_{lmn}(\bv), \varphi_{l'm'n'}(\bv) \rangle_{\omega} := \int_{\mathbb{R}^3} \varphi_{lmn}^{\dagger}(\bv) \varphi_{l'm'n'}(\bv) \omega(\bv) \,\mathrm{d}\bv = \delta_{ll'} \delta_{mm'} \delta_{nn'},
\end{equation}
where $\dagger$ refers to the complex conjugate, and the weight function is
\begin{displaymath}
\omega(\bv) = \left[ \frac{1}{(2\pi \overline{\theta})^{3/2}} \exp \left(-\frac{|\bv - \overline{\bu}|^2}{2\overline{\theta}} \right) \right]^{-1}.
\end{displaymath}
Note that here $\omega(\bv)$ is the reciprocal of a global Maxwellian, while in what follows, we use the term ``local Maxwellian'' to refer to the Maxwellian $\mathcal{M}(\bv) = \exp(\alpha + \boldsymbol{\beta} \cdot \bv + \gamma |\bv|^2)$ associated with a distribution function $f(\bv)$ by
\begin{equation} \label{eq:local_Maxwellian}
\int_{\mathbb{R}^3} \begin{pmatrix} 1 \\ \bv \\ |\bv|^2 \end{pmatrix} \mathcal{M}(\bv) \,\mathrm{d}\bv =
\int_{\mathbb{R}^3} \begin{pmatrix} 1 \\ \bv \\ |\bv|^2 \end{pmatrix} f(\bv) \,\mathrm{d}\bv.
\end{equation}

Let $\mathbf{f}(\bx,t)$ be the vector including all the coefficients $f_{lmn}(\bx,t)$ with $l + 2n \leqslant M$, and define $\bvphi$ as the vector including all the basis function $\varphi_{lmn}(\bv)$ also with $l + 2n \leqslant M$ arranged in the same order as $\mathbf{f}$. Then
\begin{displaymath}
f_M(\bx,\bv,t) = [\mathbf{f}(\bx,t)]^T \bvphi(\bv).
\end{displaymath}
With the Petrov-Galerkin method \cite{Gamba2018, Hu2020b} based on the orthogonality \eqref{eq:orth}, the semi-discrete Boltzmann equation has the form
\begin{equation} \label{eq:PG}
\frac{\partial \mathbf{f}}{\partial t} +
  \sum_{k=1}^3 \mathbf{A}_k \frac{\partial \mathbf{f}}{\partial x_k} = \mathbf{Q} : (\mathbf{f} \otimes \mathbf{f}).
\end{equation}
Here $\mathbf{A}_k$, $k=1,2,3$ are sparse matrices coming from the discretization of the advection term, and $\mathbf{Q}$ is a 3-tensor representing the discrete collision kernel. Since $\mathbf{f}$ has $O(M^3)$ components, the tensor $\mathbf{Q}$ has $O(M^9)$ elements, where only $O(M^8)$ elements are nonzero due to the rotational invariance of the collision operator \cite{Cai2020}. Despite this sparsity of $\mathbf{Q}$, the computational cost grows quickly as $M$ increases. To reduce the computational cost, in \cite{Cai2020, Hu2020b}, the authors chose $M_0 < M$ and split the discrete distribution function into two parts $f_M = f^{(1)} + f^{(2)}$, where
\begin{equation} \label{eq:f1_f2}
\begin{aligned}
f^{(1)}(\bx,\bv,t) &= \sum_{l=0}^{M_0} \sum_{m=-l}^l \sum_{n=0}^{\lfloor (M_0 - l) / 2 \rfloor} f_{lmn}(\bx,t) \varphi_{lmn}(\bv), \\
f^{(2)}(\bx,\bv,t) &= \sum_{l=0}^M \sum_{m=-l}^l \sum_{n=\max(0, \lfloor (M_0 - l) / 2 \rfloor + 1)}^{\lfloor (M - l) / 2 \rfloor} f_{lmn}(\bx,t) \varphi_{lmn}(\bv).
\end{aligned}
\end{equation}
Due to the high efficiency of the spectral approximation, we expect that by choosing $M_0 < M$, the first part $f^{(1)}$ can capture the sketch of distribution function $f$, while $f^{(2)}$ provides more details of its profile. For simplicity, we write $\mathbf{f}$ as
\begin{displaymath}
\mathbf{f} = \begin{pmatrix} \mathbf{f}^{(1)} \\ \mathbf{f}^{(2)} \end{pmatrix},
\end{displaymath}
where $\mathbf{f}^{(1)}$ includes all the coefficients in the expansion of $f^{(1)}$ and $\mathbf{f}^{(2)}$ includes all the coefficients in the expansion of $f^{(2)}$. Similarly, we let $\mathbf{M}$ denote the coefficients in the truncated expansion of the local Maxwellian defined by \eqref{eq:local_Maxwellian},and use $\mathbf{M}^{(1)}$ and $\mathbf{M}^{(2)}$ to denote its subvectors. Then using the idea of the BGK and Shakhov collision operators, the original collision term $\mathbf{Q} : (\mathbf{f} \otimes \mathbf{f})$ can be approximated by
\begin{equation} \label{eq:approx_coll}
\begin{pmatrix}
\mathbf{Q}_{M_0} : (\mathbf{f}^{(1)} \otimes \mathbf{f}^{(1)}) \\
\nu_{M_0} (\mathbf{M}^{(2)} - \mathbf{f}^{(2)})
\end{pmatrix}.
\end{equation}
Here $\mathbf{Q}_{M_0}$ is the discrete collision kernel for $M = M_0$. The approximation \eqref{eq:approx_coll} applies the accurate binary collision operator to the sketch of the distribution function, while for the part representing the finer details, a simpler BGK-like expression is used instead. Such an idea can be found in \cite{Cercignani2006} as a generalization of the classical BGK model. It was later realized for the linearized Boltzmann collision operator in \cite{Cai2015}, where the authors proved that for the linearized collision term, our BGK-like operator converges to the original operator in the resolvent sense as $M_0 \rightarrow +\infty$. The generalization to quadratic collision operators was first introduced in \cite{Wang2019}, where the choice of the parameter $\nu_{M_0}$ was chosen to be the spectral radius of the truncated linearized collision operator following the approach in \cite{Cai2015}. Here we adopt the same choice of $\nu_{M_0}$, and the details are given in the Appendix \ref{sec:nu}.

To preserve Maxwellian in the homogeneous Boltzmann equation, in \cite{Hu2020b}, the approximate collision term \eqref{eq:approx_coll} is supplemented by adding a close-to-zero term so that the semi-discrete equation reads
\begin{equation} \label{eq:semi_discrete_SSP}
\frac{\partial \mathbf{f}}{\partial t} +
  \sum_{k=1}^3 \mathbf{A}_k \frac{\partial \mathbf{f}}{\partial x_k} = \widetilde{\mathbf{Q}}(M_0; \mathbf{f}) :=
  \begin{pmatrix}
  \mathbf{Q}_{M_0} : (\mathbf{f}^{(1)} \otimes \mathbf{f}^{(1)} - \mathbf{M}^{(1)} \otimes \mathbf{M}^{(1)}) \\
  \nu_{M_0} (\mathbf{M}^{(2)} - \mathbf{f}^{(2)})
  \end{pmatrix}.
\end{equation}
Here $\mathbf{M}^{(1)}$ plays the same role as $\mathbf{f}^{(1)}$ and provides the sketch of the local Maxwellian. Since any Maxwellian is a smooth function, we again expect that $\mathbf{M}^{(1)}$ can well capture the general structure of the Maxwellian with a moderate value of $M_0$. Thus the corresponding collision term $\mathbf{Q}_{M_0} : (\mathbf{M}^{(1)} \otimes \mathbf{M}^{(1)})$ is likely to be close to zero. Such a discrete collision term ensures that it vanishes if $\mathbf{f}^{(1)} = \mathbf{M}^{(1)}$ and $\mathbf{f}^{(2)} = \mathbf{M}^{(2)}$. The idea of this approach comes from the steady-state preserving method introduced in \cite{Filbet2015}, which uses an equivalent form
\begin{displaymath}
\mathbf{Q}_{M_0} : (\mathbf{f}^{(1)} \otimes \mathbf{f}^{(1)} - \mathbf{M}^{(1)} \otimes \mathbf{M}^{(1)}) = \mathbf{Q}_{M_0} : (\mathbf{g}^{(1)} \otimes \mathbf{g}^{(1)} + \mathbf{g}^{(1)} \otimes \mathbf{M}^{(1)} + \mathbf{M}^{(1)} \otimes \mathbf{g}^{(1)}),
\end{displaymath}
where $\mathbf{g}^{(1)} = \mathbf{f}^{(1)} - \mathbf{M}^{(1)}$ denotes the non-equilibrium part of the distribution function. Such splitting also borrows ideas from the method of micro-macro decomposition to develop asymptotic preserving schemes \cite{Bennoune2008}.

The idea of hybridizing expensive and cheap models has been tested in a number of previous works \cite{Cai2015, Cai2018, Wang2019, Hu2020, Cai2020, Hu2020b}. However, the choice of $M_0$ remains to be problem-dependent, and currently its determination can only be based on trial-and-error approaches. In this work, we would like to determine $M_0$ based on an error estimate, and different $M_0$ will be used on different spatial grids.

%%% Local Variables: 
%%% mode: latex
%%% TeX-master: "article"
%%% End: 

% vim: tw=70:spell
\section{Error indicator for the adaptive collision operator} \label{sec:indicator}
% the way to compute the indicator, the approximation of the computational cost 
According to the discussion in the previous section, given $\mathbf{f}$ defined on any spatial grid cell, our purpose is to choose appropriate $M_0$ such that the right-hand side of \eqref{eq:semi_discrete_SSP} is a good approximation of the collision term. Since the collision operator is defined locally in both time and space, we expect that the choice $M_0$ on any cell depends only on the distribution function defined thereon. Hence, we will omit the arguments $\bx$ and $t$ in the following discussion. Also, we assume that the distribution function $f$ has been normalized such that its integral equals $1$. The purpose of this normalization is to provide the bounds for the relative error. To begin with, we will introduce some notations and assumptions for the sake of convenience.

\subsection{Notations and hypotheses}
Let $\mathcal{T}_M$ be truncation operator that cut off the series defined in \eqref{eq:inf_series} by discarding all the terms with polynomials of degree greater than $M$. Therefore
\begin{displaymath}
\mathcal{T}_M f = f_M
\end{displaymath}
with $f_M$ given in \eqref{eq:truncated_series}. Then in the original spectral method \eqref{eq:PG}, the right-hand side represents the coefficients in the expansion of $\mathcal{T}_M Q[f_M, f_M]$. In other words, we have
\begin{equation} \label{eq:original_rhs}
 [\mathbf{Q} : (\mathbf{f} \otimes \mathbf{f})]^T \bvphi = \mathcal{T}_M Q[f_M, f_M].
\end{equation}
Similarly, for the right-hand side of \eqref{eq:semi_discrete_SSP}, we have
\begin{equation} \label{eq:new_rhs}
  \begin{pmatrix}
  \mathbf{Q}_{M_0} : (\mathbf{f}^{(1)} \otimes \mathbf{f}^{(1)} - \mathbf{M}^{(1)} \otimes \mathbf{M}^{(1)}) \\
  \nu_{M_0} (\mathbf{M}^{(2)} - \mathbf{f}^{(2)})
  \end{pmatrix}^T \bvphi = \mathcal{T}_{M_0} \left( Q[f^{(1)}, f^{(1)}] - Q[\mathcal{M}^{(1)}, \mathcal{M}^{(1)}] \right) + \nu_{M_0}(\mathcal{M}^{(2)} - f^{(2)}),
\end{equation}
where we have used $\mathcal{M}$ to denote the local Maxwellian associated with the distribution function $f$, and we will use the notations $\mathcal{M}_M$, $\mathcal{M}^{(1)}$ and $\mathcal{M}^{(2)}$ defined similarly to \eqref{eq:truncated_series} and \eqref{eq:f1_f2}. Likewise, we define the operator
\begin{displaymath}
Q_M = \mathcal{T}_M Q,
\end{displaymath}
so that the right-hand side of \eqref{eq:original_rhs} can be written as $Q_M[f_M, f_M]$.

To choose $M_0$, we are interested in the estimation of the difference between the two right-hand sides in \eqref{eq:original_rhs} and \eqref{eq:new_rhs}. For this aim, we make the following assumptions:
\begin{itemize}
\item The truncated series $f_M$ provides a sufficiently good approximation of the distribution function $f$, so that we can assume
  \begin{displaymath}
  f_M(\bv) \gtrapprox 0, \qquad \forall \bv \in \mathbb{R}^3,
  \end{displaymath}
  where ``$\gtrapprox$'' means the inequality holds approximately.
\item The truncated series $\mathcal{M}_M$ provides a sufficiently good approximation of the local Maxwellian $\mathcal{M}$, so that we can assume $Q[\mathcal{M}_M, \mathcal{M}_M] \approx 0$.
\item The operator $Q_M$ provides a sufficiently good approximation of the collision operator $Q$, so that $Q_M$ and $Q$ are interchangeable in the derivation below.
\end{itemize}
In general, these assumptions mean that $M$ is sufficiently large so that truncated functions and operators can almost preserve the properties of the original functions and operators. The purpose of these conditions is to focus mainly on the modeling error to be described below, and temporarily ignore the error introduced by the spectral method itself. Alternatively, one can regard $M$ as infinity in our following derivations so that the conditions above hold naturally. After the error indicator is derived, to make the computation feasible, the infinities that appear in its expression are replaced by a finite $M$ to approximate our error indicator.

Now we use $\Delta Q$ to denote the modeling error, i.e. the difference between the right-hand sides of \eqref{eq:original_rhs} and \eqref{eq:new_rhs}:
\begin{displaymath}
\Delta Q := Q_M[f_M, f_M] - \Big( Q_{M_0}[f^{(1)}, f^{(1)}] - Q_{M_0}[\mathcal{M}^{(1)}, \mathcal{M}^{(1)}] + \nu_{M_0}(\mathcal{M}^{(2)} - f^{(2)}) \Big).
\end{displaymath}
Our aim is to find a heuristic error indicator characterizing the size of the quantity above. This error indicator must be relatively cheap to compute given the expansion of $f_M$. To this end, we split $\Delta Q$ into three terms, written in the three lines below:
\begin{equation} \label{eq:Delta_Q}
\begin{split}
\Delta Q & = \big( Q_M[f^{(1)}, f^{(1)}] - Q_M[\mathcal{M}^{(1)}, \mathcal{M}^{(1)}] \big) - \big( Q_{M_0}[f^{(1)}, f^{(1)}] - Q_{M_0}[\mathcal{M}^{(1)}, \mathcal{M}^{(1)}] \big) \\
& + Q_M[f_M, f_M] - \big( Q_M[f^{(1)}, f^{(1)}] - Q_M[\mathcal{M}^{(1)}, \mathcal{M}^{(1)}] \big) \\
& - \nu_{M_0} (\mathcal{M}^{(2)} - f^{(2)}).
\end{split}
\end{equation}
The first line in this equation is the truncation error of the function $Q_M[f^{(1)}, f^{(1)}] - Q_M[\mathcal{M}^{(1)}, \mathcal{M}^{(1)}]$. The estimation of the truncation error usually requires the information of the function, which is expensive to retrieve. As a workaround, we choose to ignore this term in our error indicator. In fact, this term may be relatively small due to the following two reasons:
\begin{enumerate}
\item Both $f^{(1)}$ and $\mathcal{M}^{(1)}$ are early truncations of the series (only include polynomials of degree  up to $M_0$), which usually appear to be very smooth. Therefore both $Q_M[f^{(1)}, f^{(1)}]$ and $Q_M[\mathcal{M}^{(1)}, \mathcal{M}^{(1)}]$ are sufficiently smooth functions which can be well approximated by an early truncation of their expansions.
\item According to the observations in \cite{Cai2015}, the dependence of higher moments on the lower moments is relatively weak, meaning that $f^{(1)}$ does not produce large numbers in the higher coefficients in the expansion of $Q_M[f^{(1)}, f^{(1)}]$, which is similar for $\mathcal{M}^{(1)}$.
\end{enumerate}
These statements are yet to be verified rigorously. We would like to leave it to  future work. Below we will mainly focus on the quantification of the second line in \eqref{eq:Delta_Q}.

\subsection{Building indicator}

According to our working hypotheses, we rewrite the second line of \eqref{eq:Delta_Q} by replacing $Q_M$ with $Q$ and subtracting an approximately zero term $Q[\mathcal{M}_M, \mathcal{M}_M]$. Thereby we can rewrite this term as
\begin{equation}
    \label{eq:DQ}
    \delta Q := \big(Q[f_M, f_M] - Q[\mM_M, \mM_M]\big) - \big(Q [f^{(1)}, f^{(1)}] - Q[\mM^{(1)}, \mM^{(1)}]\big).
\end{equation}
Using $f_M = f^{(1)} + f^{(2)}$ and the fact that $Q[\cdot,\cdot]$ is quadratic and symmetric, we have 
\begin{equation}
    \label{eq:sim_DQ1}
    \begin{aligned}
    \delta Q & =\big(Q[\mM_M +( f_M - \mM_M),\mM_M +( f_M - \mM_M) ]  - Q[\mM_M, \mM_M]\big) \\
    & \qquad  - \big(Q[\mM^{(1)} +(f^{(1)} - \mM^{(1)}),\mM^{(1)} +(f^{(1)} - \mM^{(1)}) ]  - Q[\mM^{(1)}, \mM^{(1)}] \big)\\
%    & = \big(2Q[f_M - \mM_M, \mM_M] - Q[f_M - \mM_M, f_M - \mM_M] \big)  - \big(2 Q[f^{(1)} - \mM^{(1)}, \mM^{(1)}] - Q[f^{(1)} - \mM^{(1)}, f^{(1)} - \mM^{(1)}] \big) \\
    & =  \underbrace{2\big(Q[f_M - \mM_M, \mM_M] - Q[f^{(1)} - \mM^{(1)}, \mM^{(1)}] \big)}_{\delta Q_a} + \underbrace{\big( Q[f_M - \mM_M, f_M - \mM_M] - Q[f^{(1)} - \mM^{(1)}, f^{(1)} - \mM^{(1)}]\big)}_{\delta Q_b},
    \end{aligned}
\end{equation}
where $\delta Q_a$ and $\delta Q_b$ can be further simplified as 
\begin{equation}
    \label{eq:Q_ab} 
    \begin{aligned}
    \delta Q_a & =2 Q[(f^{(1)} - \mM^{(1)} )+ (f^{(2)} - \mM^{(2)}), \mM^{(1)} + \mM^{(2)}]  - 2Q[f^{(1)} - \mM^{(1)}, \mM^{(1)}] \\
    & = 2 Q[f^{(2)} - \mM^{(2)},  \mM_M]  + 2 Q[f^{(1)} - \mM^{(1)},  \mM^{(2)}], \\
    \delta Q_b & =Q[f_M - \mM_M, f_M - \mM_M] - Q[(f_M - \mM_M) - (f^{(2)} -  \mM^{(2)}), (f_M - \mM_M) - ( f^{(2)} - \mM^{(2)})] \\
     & =  Q[f^{(2)} - \mM^{(2)}, f^{(2)}- \mM^{(2)}]  + 2 Q[f_M - \mM_M , f^{(2)}- \mM^{(2)}].  \end{aligned}
\end{equation}
Inserting \eqref{eq:Q_ab} into \eqref{eq:sim_DQ1} yields
\begin{equation}
     \label{eq:sim_DQ}
    \begin{aligned}   
    \delta Q  & = 2\underbrace{Q[ f_M, f^{(2)} - \mM^{(2)}]}_{\delta Q_1} + 2\underbrace{Q[f^{(1)} - \mM^{(1)}, \mM^{(2)} ]}_{\delta Q_2} 
     +\underbrace{Q[f^{(2)} - \mM^{(2)}, f^{(2)} - \mM^{(2)} ]}_{\delta Q_3}.
    \end{aligned}
\end{equation}
This expression implies that $\delta Q$ is small in either of the following two scenarios:
\begin{enumerate}
\item The functions $f^{(1)}$ and $\mM^{(1)}$ can accurately describe $f_M$ and $\mM_M$, respectively, which means $f^{(2)}$ and $\mM^{(2)}$ are small. 
\item The function $f_M$ is close to the equilibrium $\mM_M$, which means $f^{(1)} - \mathcal{M}^{(1)}$ and $f^{(2)} - \mathcal{M}^{(2)}$ are both small.
\end{enumerate}
In either case, the term $\delta Q_3$ appears to be a quadratic term, and is expected to be smaller than the previous two terms. Hence, we ignore this term in our error indicator and mainly discuss the estimation of $\delta Q_1$ and $\delta Q_2$. The analysis above indicates that $\delta Q_1$ and $\delta Q_2$ present two different sources of the error: $\delta Q_1$ mainly captures the error due to the BGK approximation of the high-frequency part, and $\delta Q_2$ mainly captures the error due to the missing interaction between the low-frequency part and the high-frequency part.

Before proceeding, we would like to first discuss how much computational cost we can afford to estimate  \eqref{eq:sim_DQ}. Recall that the time complexity for evaluating all the coefficients in the expansion of $Q_M[f_M, f_M]$ is $O(M^8)$. Therefore, the computational cost of the indicator should be essentially smaller than $O(M^8)$ to achieve savings. Besides, the computational cost for the right-hand side of \eqref{eq:semi_discrete_SSP} is $O(M_0^8 + M^3)$ according to \cite{Hu2020b}, and our computational cost for the indicators should not be significantly larger than this. Thus, it is unrealistic to compute $\delta Q_1$ directly since its computation requires already $O(M^8)$ operations. Meanwhile, we would also like to avoid direct computation of $\delta Q_2$ since it requires $O(M^6 M_0^2)$ operations, which would take the most time of the simulation if computed. As a result, we need to estimate these terms using a computationally cheaper expression. The most naive approach is to consider the following type of estimation:
\begin{displaymath}
\|Q[f,g]\| \leqslant C \|f\| \cdot \|g\|.
\end{displaymath}
Unfortunately, the collision operator $Q$ is unbounded for most collision kernels. Note that the gain term may be bounded when the assumption of Grad's angular cut-off holds (see e.g. \cite{Mouhot2004}), while our approach to be proposed in the rest part of this section does not rely on this assumption.

The basic idea is to bound $Q[f,g]$ by splitting it into the gain and loss operators:
\begin{align*}
Q^+[f,g] &= \frac{1}{2} \int_{\mathbb{R}^3} \int_{\bn \perp \bg} \int_0^{\pi} B(|\bg|, \chi) [f(\bv') g(\bv_*') + f(\bv_*') g(\bv')] \,\mathrm{d}\chi \,\mathrm{d}\bn \,\mathrm{d}\bv_*, \\
Q^-[f,g] &= \frac{1}{2} \int_{\mathbb{R}^3} \int_{\bn \perp \bg} \int_0^{\pi} B(|\bg|, \chi) [f(\bv) g(\bv_*) + f(\bv_*) g(\bv)] \,\mathrm{d}\chi \,\mathrm{d}\bn \,\mathrm{d}\bv_*.
\end{align*}
Then for any functions $f$ and $g$ satisfying $f \geqslant 0$, if we can find another distribution function $h$ satisfying $|g(\bv)| \leqslant h(\bv)$ for every $\bv \in \mathbb{R}^3$, then it holds that
\begin{equation} \label{eq:est}
|Q[f,g]| = |Q^+[f,g] - Q^-[f,g]| \leqslant Q^+[f,|g|] + Q^-[f,|g|] \leqslant Q^+[f,h] + Q^-[f,h],
\end{equation}
where we have used the positivity of the collision kernel $B(|\bg|,\chi)$ to get the last inequality.
Thus, the right-hand side of \eqref{eq:est} can be used as an upper bound of $|Q[f,g]|$. In general, the computational cost of this upper bound is as high as a full collision operator. Therefore, the function $h$ must be chosen carefully so that $Q^+[f,h] + Q^-[f,h]$ can be efficiently computed. For simplicity, we define
\begin{displaymath}
\Qabs[f,h] = Q^+[f,h] + Q^-[f,h],
\end{displaymath}
which is different from $Q[f,h]$ since $Q[f,h]$ is the difference of these two terms. Based on this idea, we need to answer the following two questions:
\begin{enumerate}
\item What should the general form of $h$ be such that the computation of $\Qabs[f,h]$ is efficient?
\item Given $f_M$, how to find the function $h$ as the bound of $f_M$?
\end{enumerate}
These two questions will be addressed in the following two subsections.

\subsubsection{Space of the bounding function}
Our choice of $h$ is established on the rotational invariance of the collision operator:
\begin{lemma}
For any orthogonal matrix $\boldsymbol{R}$, let
\begin{displaymath}
f^R(\bv) = f(\boldsymbol{R} \bv), \quad g^R(\bv) = g(\boldsymbol{R} \bv).
\end{displaymath}
Then
\begin{displaymath}
Q[f,g](\boldsymbol{R}\bv) = Q[f^R, g^R](\bv), \qquad
\Qabs[f,g](\boldsymbol{R}\bv) = \Qabs[f^R, g^R](\bv).
\end{displaymath}
\end{lemma}
Here the rotational invariance of the collision operator is a classical result and can be found in \cite[p. 45]{Degond2004}. The rotational invariance of $\Qabs$ can be derived from the rotational invariance of both $Q^+$ and $Q^-$. A natural consequence of this result is
\begin{corollary}
Let $h(\bv)$ be a function depending only on $|\bv|$. Then the linear operator $\mathcal{L}_h[\cdot] := \Qabs[\cdot,h]$ is a rotational invariant operator, \textit{i.e.}, for any orthogonal matrix $\boldsymbol{R}$, we have
\begin{displaymath}
\mathcal{L}_h[f](\boldsymbol{R}\bv) = \mathcal{L}_h[f^R](\bv),
\end{displaymath}
for $f^R(\bv) = f(\boldsymbol{R}\bv)$.
\end{corollary}
For linear rotationally invariant operators, we have the following result:
\begin{theorem}
Suppose $\mathcal{L}[\cdot]$ is a linear rationally invariant operator. For any non-negative integers $l$ and $n$, there exists an isotropic function $c_{ln}(|\bv|)$ such that
\begin{displaymath}
\mathcal{L}[\varphi_{lmn}](\bv) = c_{ln}(|\bv|) \varphi_{lmn}(\bv), \qquad \forall m = -l, \cdots, l.
\end{displaymath}
\end{theorem}
This result is already given in \cite[Theorem 1]{Cai2015}, where the statement is written for the linearized collision operator but the proof only requires the rotational invariance of $\mathcal{L}$. This theorem indicates that the series form of $\mathcal{L}_h[\varphi_{lmn}]$ should be
\begin{equation} \label{eq:L_h}
\mathcal{L}_h[\varphi_{lmn}] = \sum_{n_1=0}^{+\infty} a_{lnn_1}^{(h)} \varphi_{lmn_1}.
\end{equation}
Consequently, if we choose $h(\bv)$ to be an isotropic function that depends only on $|\bv|$, we have
\begin{displaymath}
\Qabs[f_M,h] = \mathcal{L}_h[f_M] = \sum_{l=0}^M \sum_{m=-l}^l \sum_{n=0}^{\lfloor (M-l)/2 \rfloor} f_{lmn} \mathcal{L}_h[\varphi_{lmn}] = \sum_{l=0}^M \sum_{m=-l}^l \sum_{n=0}^{\lfloor (M-l)/2 \rfloor} \sum_{n_1=0}^{+\infty} f_{lmn} a_{lnn_1}^{(h)} \varphi_{lmn_1}.
\end{displaymath}
Numerically, we truncate the series above by replacing $+\infty$ with $\lfloor (M-l)/2 \rfloor$. Thus, when all the coefficients $a_{lnn'}^{(h)}$ are given, the computational cost for the expansion of $\Qabs[f_M,h]$ is $O(M^4)$. Technically, this can be done by carrying out the matrix-vector multiplication
\begin{equation} \label{eq:Qabs_coef}
\Qabs_{lmn_1} = \sum_{n=0}^{\lfloor (M-l)/2 \rfloor} a_{lnn_1}^{(h)} f_{lmn}, \qquad l = 0,1\cdots,M, \quad m = -l, \cdots, l, \quad n_1 = 0,1,\cdots,\lfloor (M-l)/2 \rfloor,
\end{equation}
as allows efficient libraries of linear algebra to be used.

To find the coefficients $a_{lnn_1}^{(h)}$, we need the expansion of $h$ in terms of the basis functions $\varphi_{lmn}$. Since we choose $h$ to be isotropic, the expansion holds the form
\begin{displaymath}
h(\bv) = \sum_{n'=0}^{+\infty} h_{n'} \varphi_{00n'}(\bv).
\end{displaymath}
For each $n'$, the function $\varphi_{00n'}$ is also isotropic, implying that $\mathcal{L}_{\varphi_{00n'}}[\cdot] = \Qabs[\varphi_{00n'}, \cdot]$ is rotationally invariant. Thus we can assume
\begin{displaymath}
\mathcal{L}_{\varphi_{00n'}}[\varphi_{lmn}] = \sum_{n_1 = 0}^{+\infty} a_{lnn_1}^{n'} \varphi_{lmn_1},
\end{displaymath}
so that
\begin{displaymath}
\mathcal{L}_{h}[\varphi_{lmn}] = \sum_{n'=0}^{+\infty} h_{n'} \mathcal{L}_{\varphi_{00n'}}[\varphi_{lmn}] = \sum_{n_1=0}^{+\infty} \sum_{n'=0}^{+\infty} a_{lnn_1}^{n'} h_{n'} \varphi_{lmn_1}.
\end{displaymath}
By comparing this equation with \eqref{eq:L_h}, one can find that
\begin{equation} \label{eq:a_lnn1}
a_{lnn_1}^{(h)} = \sum_{n' = 0}^{+\infty} a_{lnn_1}^{n'} h_{n'}, \qquad l = 0,1,\cdots,M, \quad n, n_1 = 0,1,\cdots, \lfloor (M-l)/2 \rfloor.
\end{equation}
In practice, we again truncate the above series by replacing $+\infty$ with some $N_0 = O(M)$. Then the equation \eqref{eq:a_lnn1} shows that the computation of all required coefficients again needs $O(M^4)$ operations.

By now, we have concluded that choosing $h(\bv)$ to be an isotropic function can reduce the computational cost of $\Qabs[f,h]$ to $O(M^4)$ (including \eqref{eq:a_lnn1} and \eqref{eq:Qabs_coef}). To complete the computation, we still need to obtain $a_{lnn_1}^{n'}$ for a given collision operator. These coefficients can be precomputed before the simulation, which will be discussed in detail in Section \ref{sec:a_lnn1}. Now we will first discuss the construction of $h$ such that $|g| \leqslant h$ for some given function $g$ in order that the estimation \eqref{eq:est} holds.

\subsubsection{The approximation of the bounding function and the error indicator}
In our implementation, instead of looking for $h$ that bounds $|g|$ pointwisely, we choose to find an approximate upper bound with the form
\begin{equation} \label{eq:hv}
h(\bv) = \sum_{n'=0}^{N_0} h_{n'} \varphi_{00n'}(\bv), \qquad N_0 = \left\lceil \frac{M}{2} \right\rceil.
\end{equation}
The general idea to find $h$ is to bound the radial part and the angular part separately. To this end, we write the basis functions as
\begin{equation}
  \varphi_{lmn}(\bv) = \varphi^{1}_{ln}(\bv-\overline{\bu}) Y_l^m \left( \frac{\bv-\overline{\bu}}{|\bv-\overline{\bu}|} \right),
\end{equation}
where 
\begin{equation}
\label{eq:split_bas_psh}
\varphi^{1}_{ln}(\bv) = 
\sqrt{\frac{2^{1-l}\pi^{3/2}n!}{\Gamma(n+l+3/2)}} L^{(l+1/2)}_n \left( \frac{|\bv|^2}{2\overline{\theta}} \right) \left( \frac{|\bv|}{\sqrt{\overline{\theta}}}\right)^l\cdot \frac{1}{(2\pi \overline{\theta})^{3/2}} \exp \left(-\frac{|\bv|^2}{2\overline{\theta}} \right)
\end{equation}
represents the radial part of the basis function. Without loss of generality, we set $\bar{\bu} = 0$ and $\bar{\theta} = 1$ in the analysis below. Since $\varphi_{ln}^1(\bv)$ depends only on $|\bv|$, here we approximate its absolute value by a linear combination of $\varphi_{00n}$:
\begin{equation}
\label{eq:appro_phi1}
|\varphi^{1}_{ln}(\bv)| \approx \sum_{n' = 0}^{N_0 } s_{ln}^{n'} \varphi_{00n'}(\bv),
\end{equation}
and we choose to find the approximation by orthogonal projection:
 \begin{equation}
     \label{eq:sn}
     s_{ln}^{n'} = \int_{\mathbb{R}^3}|\varphi^1_{ln}(\bv)| \varphi_{00n'}(\bv) \omega(\bv) \,\mathrm{d}\bv.
 \end{equation}
 Since $\varphi_{00n}$ contains a polynomial with even order $2n$, by choosing $N_0 = \lceil{\frac{M}{2}}\rceil$, we can guarantee that the degree of the polynomial in the right-hand side of \eqref{eq:appro_phi1} is no less than that in the radial function $\varphi^1_{ln}$. Thus, it holds for any distribution function $g$ that 
\begin{equation}
    \label{eq:bound_f}
    \begin{aligned}
    |g(\bv)| & 
    = \left| \sum_{l=0}^M \sum_{n=0}^{\lfloor (M-l)/2 \rfloor} \sum_{m=-l}^l  g_{lmn} Y_l^m\left(\frac{\bv}{|\bv|}\right) \varphi^1_{ln}(\bv)\right|\leqslant \sum_{l=0}^M \sum_{n=0}^{\lfloor (M-l)/2 \rfloor} \left|\sum_{m=-l}^l g_{lmn}Y_{l}^m\left(\frac{\bv}{|\bv|}\right)\right| |\varphi^1_{ln}(\bv)| \\
    & \leqslant \sum_{l=0}^M \sum_{n=0}^{\lfloor (M-l)/2 \rfloor} g_{ln} |\varphi^1_{ln}(\bv)|
    \approx \sum_{l=0}^M \sum_{n=0}^{\lfloor (M-l)/2 \rfloor} g_{ln} \sum_{n' = 0}^{N_0} s_{ln}^{n'}\varphi_{00n'}(\bv),
    \end{aligned}
\end{equation}
where 
\begin{equation}
    \label{eq:coeG}
    g_{ln} = \max_{|\bn| = 1}\left| \sum_{m=-l}^l g_{lmn}Y^m_{l}(\bn)\right|.
\end{equation}
Equation \eqref{eq:bound_f} shows that we can choose
\begin{equation} \label{eq:h}
h_{n'} = \sum_{l=0}^M \sum_{n=0}^{\lfloor (M-l)/2 \rfloor} g_{ln} s_{ln}^{n'}, \qquad n' = 0,1,\cdots,N_0
\end{equation}
such that $h(\bv)$ is an approximate upper bound of $g(\bv)$.

In the calculation of $h_{n'}$, the coefficients $s_{ln}^{n'}$ can be precomputed by numerical integration before the simulation. Thus once $g_{ln}$ are obtained, the computational cost is $O(M^3)$. As for $g_{ln}$, we again choose to approximate them instead of computing them exactly. We pick a finite set of points $\Omega \in \mathbb{S}^2$ and approximate $g_{ln}$ by
\begin{equation}
    \label{eq:coeG1}
    g_{ln} \approx \max_{\bn \in \Omega}\left| \sum_{m=-l}^l g_{lmn}Y^m_{l}(\bn)\right|.
\end{equation}
In our implementation, the fifty-point Lebedev-Gauss integral points \cite{Lebedev1975} are chosen to form the set $\Omega$. Thus the computational cost to find all $g_{ln}$ is also $O(M^3)$.

By such means, we can find $h(\bv)$ with the form \eqref{eq:hv} such that $|f^{(2)}(\bv) - \mM^{(2)}(\bv)| \lessapprox h(\bv)$, meaning that $h(\bv)$ is an approximate bound of $f^{(2)} - \mM^{(2)}$. Thus $\delta Q_1$ defined in \eqref{eq:sim_DQ} can be bounded by
  \begin{equation}
     \label{eq:est_Q1}
     \begin{aligned}
     |\delta Q_1| &\leqslant Q^{\rm abs}[f_M, |f^{(2)} - \mM^{(2)}|] \lessapprox \Qabs[f_M, h],
     \end{aligned}
 \end{equation}
which can be computed via \eqref{eq:Qabs_coef} and \eqref{eq:a_lnn1} with computational cost $O(M^4)$. To bound $\delta Q_2$, we adopt the similar approach: by constructing $h^{(1)}(\bv)$ and $h^{(2)}(\bv)$ such that $|f^{(1)} - \mM^{(1)}| \lessapprox h^{(1)}(\bv)$, $|\mM^{(2)}| \lessapprox h^{(2)}(\bv)$ and
\begin{equation}
  h^{(1)}(\bv) = \sum_{n' = 0}^{N_0} h^{(1)}_{n'} \varphi_{00n'}(\bv),
  \qquad h^{(2)}(\bv) = \sum_{n' = 0}^{N_0} h^{(2)}_{n'} \varphi_{00n'}(\bv),
\end{equation}
we have
 \begin{equation}
     \label{eq:eta_Q2}
     \begin{aligned}
      |\delta Q_2| &\leqslant \Qabs[|f^{(1)} - \mM^{(1)}|, |\mM^{(2)}|] \lessapprox \Qabs[h^{(1)}, h^{(2)}] \\
      & =  \sum_{n' = 0}^{N_0} \sum_{n=0}^{N_0} h^{(1)}_{n} h^{(2)}_{n'} \mathcal{L}_{\varphi_{00n'}} [\varphi_{00n}] = \sum_{n' = 0}^{N_0} \sum_{n=0}^{N_0} h_n^{(1)}h^{(2)}_{n'} \sum_{n_1 = 0}^{N_0} a_{0n n_1}^{n'}\varphi_{00n_1} \\
      & = \sum_{n_1 = 0}^{N_0} \left( \sum_{n'=0}^{N_0} \sum_{n=0}^{N_0} h_n^{(1)} h_{n'}^{(2)} a_{0nn_1}^{n'} \right) \varphi_{00n_1},
     \end{aligned}
 \end{equation}
 where the coefficients
\begin{equation} \label{eq:tQ}
\widetilde{Q}^{\mathrm{abs}}_{n_1} = \sum_{n'=0}^{N_0} \sum_{n=0}^{N_0} h_n^{(1)} h_{n'}^{(2)} a_{0nn_1}^{n'}, \qquad n = 0,1,\cdots,N_0
\end{equation}
can be computed with time complexity $O(M^3)$. Finally, we choose our error indicator to be the sum of the bounds of both $\delta Q_1$ and $\delta Q_2$:
\begin{equation} \label{eq:indicator}
\mathrm{Indicator} = \sqrt{\sum_{l=0}^M \sum_{m=-l}^l \sum_{n_1=0}^{N_0} \left|\Qabs_{lmn_1}\right|^2} +\sqrt{ \sum_{n_1=0}^{N_0} \left|\widetilde{Q}^{\mathrm{abs}}_{n_1}\right|^2}.
\end{equation}
Here we have used the weighted $L^2$-norm with weight function $\omega(\bv)$ so that the norm is simply the sum of squares of the coefficients. This error indicator can be considered as an \textit{a posteriori} estimation of the truncation error, since it depends on the numerical solution but only estimates the error of the collision term instead of the solution itself.

In the indicator above, we did not consider the third line of \eqref{eq:Delta_Q} since the contribution of this term is similar to $h$ times a constant (since $h(\bv)$ is the approximate bound of $|f^{(2)}(\bv) - \mM^{(2)}(\bv)|$). Such contribution has been covered by the term $\Qabs[f_M,h]$ and it is less meaningful to duplicate it in the error indicator.
 
\subsubsection{Computation of the coefficients $a_{lnn_1}^{n'}$} \label{sec:a_lnn1}
By the orthogonality of the basis functions \eqref{eq:orth}, we have
\begin{equation} \label{eq:a_lnn1^n'}
a_{lnn_1}^{n'} = \int_{\mathbb{R}^3} p_{lmn}^{\dagger}(\bv) \Qabs[\varphi_{lmn_1},\varphi_{00n'}](\bv) \,\mathrm{d}\bv
\end{equation}
for any $m = -l,\cdots,l$. For simplicity, here we have used
\begin{displaymath}
p_{lmn}(\bv) = \varphi_{lmn}(\bv) \omega(\bv),
\end{displaymath}
and $p_{lmn}$ is a polynomial of degree $l+2n$. Below we will also use $p_{lmn}^0$ to denote the polynomial $p_{lmn}$ with $\overline{\bu}$ set to be zero. In our calculation, we choose $m = 0$ so that all the functions are real and we can remove ``$\dagger$'' in \eqref{eq:a_lnn1^n'}. A well-known property of the collision integral is
\begin{equation} \label{eq:pq_int}
\begin{split}
& \int_{\mathbb{R}^3} p_{l0n}(\bv) \Qabs[\varphi_{l0n_1},\varphi_{00n'}](\bv) \,\mathrm{d}\bv \\
={} & \frac{1}{2} \int_{\mathbb{R}^3} \int_{\mathbb{R}^3} \int_{\bn \perp \bg} \int_0^{\pi} [p_{l0n}(\bv') + p_{l0n}(\bv_*') + p_{l0n}(\bv) + p_{l0n}(\bv_*)] B(|\bg|, \chi) \varphi_{l0n_1} (\bv) \varphi_{00n'}(\bv_*) \,\mathrm{d}\chi \,\mathrm{d}\bn \,\mathrm{d}\bv_* \,\mathrm{d}\bv.
\end{split}
\end{equation}
A classical approach \cite{Grad1949, kumar1966} to computing this integral is to define
\begin{displaymath}
\bg' = \bv' - \bv_*', \qquad \bh = \frac{\bv + \bv_*}{2},
\end{displaymath}
which yields
\begin{displaymath}
|\bg| = |\bg'|, \quad \bv = \bh + \frac{1}{2} \bg, \quad \bv_* = \bh - \frac{1}{2} \bg, \quad \bv' = \bh + \frac{1}{2} \bg', \quad \bv'_* = \bh - \frac{1}{2} \bg', \quad \mathrm{d}\bv \,\mathrm{d}\bv_* = \mathrm{d}\bg \,\mathrm{d}\bh.
\end{displaymath}
The purpose of these changes of variables is to convert the integral with respect to $\bv_*$ and $\bv$ to the integral with respect to $\bg$ and $\bh$. This requires us to express the polynomials in \eqref{eq:pq_int} by linear combinations of the basis polynomials of $\bg$ and $\bh$. This requires the following result:
\begin{theorem} \label{thm:pp}
For any non-negative integers $l$, $n$ and $n'$, it holds that
\begin{equation} \label{eq:vv_to_gh}
\begin{split}
& p_{l0n}\left( \bh + \frac{1}{2} \bg \right) p_{00n'} \left(\bh - \frac{1}{2} \bg\right) \\
={} & \sum_{\substack{l_1, l_2, n_1, n_2 \geqslant 0\\ l_1 + l_2 + 2(n_1 + n_2) = l + 2(n+n')}} \sum_{\substack{m_1 = -l_1, \cdots, l_1 \\ m_2 = -l_2, \cdots, l_2 \\ m_1 + m_2 = 0}} A_{lnn'}^{l_1 l_2 m_2 n_2} p_{l_1 m_1 n_1}(\sqrt{2}\bh) p_{l_2 m_2 n_2}^0\left( \frac{\bg}{\sqrt{2}} \right).
\end{split}
\end{equation}
The coefficients $A_{lnn'}^{l_1 l_2 m_2 n_2}$ are constants satisfying the recurrence relation:
\begin{equation} \label{eq:A_recur}
\begin{split}
& A_{ln,n'+1}^{l_1 l_2 m_2 n_2} =
  \frac{1}{\sqrt{(n'+1)(n'+3/2)}} \Bigg[
    \sum_{k=1}^2 \frac{\sqrt{n_k(n_k+l_k+1/2)}}{2}
      A_{lnn'}^{l_1l_2m_2,n_2-\delta_{2k}} \\
& \qquad +\sum_{\mu = -1}^1 (-1)^{\mu} \bigg(
      \sqrt{(l_1 + n_1 + 1/2)(l_2 + n_2 + 1/2)}
      \gamma_{l_1,m_2+\mu}^{-\mu} \gamma_{l_2,m_2+\mu}^{-\mu}
        A_{lnn'}^{l_1-1,l_2-1,m_2+\mu,n_2} \\
& \hspace{90pt} -(-1)^{\mu} \sqrt{n_1(l_2 + n_2 + 1/2)}
      \gamma_{-l_1-1,m_2+\mu}^{-\mu} \gamma_{l_2,m_2+\mu}^{-\mu}
        A_{lnn'}^{l_1+1,l_2-1,m_2+\mu,n_2} \\
& \hspace{90pt} -(-1)^{\mu} \sqrt{n_2(l_1 + n_1 + 1/2)}
      \gamma_{l_1,m_2+\mu}^{-\mu} \gamma_{-l_2-1,m_2+\mu}^{-\mu}
        A_{lnn'}^{l_1-1,l_2+1,m_2+\mu,n_2-1} \\
& \hspace{90pt} +\sqrt{n_1 n_2}
      \gamma_{-l_1-1,m_2+\mu}^{-\mu} \gamma_{-l_2-1,m_2+\mu}^{-\mu}
        A_{lnn'}^{l_1+1,l_2+1,m_2+\mu,n_2-1}
    \bigg)
  \Bigg],
\end{split}
\end{equation}
where $n_1 = (l - l_1 - l_2) / 2 + (n + n' - n_2)$, and $\gamma_{lm}^{\mu}$ is defined by
\begin{equation} \label{eq:gamma_lm}
\gamma_{lm}^{\mu} =
  \sqrt{\frac{[l + (2\delta_{1,\mu} - 1) m + \delta_{1,\mu}]
      [l - (2\delta_{-1,\mu} - 1) m + \delta_{-1,\mu}]}
    {2^{|\mu|} (2l-1)(2l+1)}},
\end{equation}
and $A_{lnn'}^{l_1 m_1 n_1, l_2 m_2 n_2}$ is regarded as zero if any of the following conditions are violated:
\begin{displaymath}
\text{(1) } l_1, l_2, n_2 \text{ are positive integers;} \qquad \text{(2) } l - l_1 - l_2 + 2(n + n' - n_2) \geqslant 0; \qquad \text{(3) } |m_2| \leqslant \min(l_1, l_2).
\end{displaymath}
\end{theorem}
The proof of this theorem is similar to the proof of \cite[Proposition 3]{Cai2015}, and the details can be found in Appendix \ref{sec:proof}. The recurrence relation \eqref{eq:A_recur} helps compute the coefficients in the expansion \eqref{eq:vv_to_gh}. The initial condition corresponds to the case $n' = 0$, for which $p_{00n'}(\cdot) \equiv 1$ and the corresponding coefficients $A_{lnn'}^{l_1 l_2 m_2 n_2}$ have been derived in \cite[Proposition 3]{Cai2015}.

By \eqref{eq:vv_to_gh}, we have
\begin{displaymath}
\begin{split}
& p_{l0n}(\bv') + p_{l0n}(\bv_*') +  p_{l0n}(\bv) +  p_{l0n}(\bv_*) \\
={} & \sum_{\substack{l_1, l_2, n_1, n_2 \geqslant 0\\ l_1 + l_2 + 2(n_1 + n_2) = l + 2n}} \sum_{\substack{m_1 = -l_1, \cdots, l_1 \\ m_2 = -l_2, \cdots, l_2 \\ m_1 + m_2 = 0}} A_{ln0}^{l_1 l_2 m_2 n_2} p_{l_1 m_1 n_1}(\sqrt{2}\bh) \times {} \\
& \hspace{120pt} \left[ p_{l_2 m_2 n_2}^0\left( \frac{\bg'}{\sqrt{2}} \right) + p_{l_2 m_2 n_2}^0\left( -\frac{\bg'}{\sqrt{2}} \right) + p_{l_2 m_2 n_2}^0\left( \frac{\bg}{\sqrt{2}} \right) + p_{l_2 m_2 n_2}^0\left( -\frac{\bg}{\sqrt{2}} \right)\right] \\
={} & \sum_{\substack{l_1, l_2, n_1, n_2 \geqslant 0\\ l_1 + l_2 + 2(n_1 + n_2) = l + 2n}} \sum_{\substack{m_1 = -l_1, \cdots, l_1 \\ m_2 = -l_2, \cdots, l_2 \\ m_1 + m_2 = 0}} [1 + (-1)^{l_2}] A_{ln0}^{l_1 l_2 m_2 n_2} p_{l_1 m_1 n_1}(\sqrt{2}\bh)
\left[ p_{l_2 m_2 n_2}^0 \left( \frac{\bg'}{\sqrt{2}} \right) + p_{l_2 m_2 n_2}^0\left( \frac{\bg}{\sqrt{2}} \right)\right],
\end{split}
\end{displaymath}
where we have used the property that $p_{lmn}^0$ is odd/even if $l$ is odd/even. By the equality \cite[Eq. (7.1)]{Ikenberry1956}
\begin{displaymath}
\int_{\bn \perp \bg} p_{lmn}^0\left( \frac{\bg'}{\sqrt{2}} \right) \,\mathrm{d}\bn = 2\pi p_{lmn}^0\left(\frac{\bg}{\sqrt{2}} \right) P_l(\cos \chi),
\end{displaymath}
we conclude that
\begin{equation} \label{eq:n_int}
\begin{split}
& \int_{\bn \perp \bg} [p_{l0n}(\bv') + p_{l0n}(\bv_*') +  p_{l0n}(\bv) +  p_{l0n}(\bv_*)] \,\mathrm{d}\bn \\
={} & \sum_{\substack{l_1, l_2, n_1, n_2 \geqslant 0\\ l_1 + l_2 + 2(n_1 + n_2) = l + 2n}} \sum_{\substack{m_1 = -l_1, \cdots, l_1 \\ m_2 = -l_2, \cdots, l_2 \\ m_1 + m_2 = 0}} 2\pi [1 + (-1)^{l_2}][P_{l_2}(\cos \chi) + 1] A_{ln0}^{l_1 l_2 m_2 n_2} p_{l_1 m_1 n_1}(\sqrt{2}\bh) p_{l_2 m_2 n_2}^0\left( \frac{\bg}{\sqrt{2}} \right).
\end{split}
\end{equation}
Next, we perform the following operations:

\begin{enumerate}
\item Insert \eqref{eq:n_int} into \eqref{eq:pq_int};
\item Use \eqref{eq:vv_to_gh} to expand $\varphi_{l0n_1}(\bv) \varphi_{00n'}(\bv_*)$ in \eqref{eq:pq_int};
\item Use the orthogonality \eqref{eq:orth} to integrate with respect to $\bh$.
\item Represent $\bg$ using spherical coordinates as $g \boldsymbol{\omega}$ and use the orthogonality of spherical harmonics to integrate with respect to $\boldsymbol{\omega}$.
\end{enumerate}
We omit the details of these steps since they are standard procedures to compute the moments of the collision operators. Afterward, we obtain
\begin{equation} \label{eq:a}
\begin{split}
a_{lnn_1}^{n'} = & \hspace{-8pt} \sum_{\substack{l_1, l_2 \geqslant 0 \\ l - (l_1 + l_2) \text{ is even} \\ l_1 + l_2 \leqslant l + 2 \min(n,n_1 + n')}} \sum_{m_2 = -\min(l_1,l_2)}^{\min(l_1,l_2)}\sum_{\substack{n_2, n_2' \geqslant 0 \\ n - n_2 = n_1 + n' - n_2'\\n_2 \leqslant (l-l_1-l_2)/2+n}} \sqrt{\frac{n_2! n_2'!}{\Gamma(l_2 + n_2 + 3/2) \Gamma(l_2 + n_2' + 3/2)}} A_{ln0}^{l_1 l_2 m_2 n_2} A_{ln_1 n'}^{l_1 l_2 m_2 n_2'} \\
& \times \pi [1 + (-1)^{l_2}] \underline{\int_0^{+\infty} \int_0^{\pi} B(\sqrt{4\bar{\theta} r},\chi) [P_{l_2}(\cos\chi) + 1] L_{n_2}^{(l_2 + 1/2)}(r) L_{n_2'}^{(l_2 + 1/2)}(r) r^{l_2 + 1/2} \exp(-r) \,\mathrm{d}\chi \,\mathrm{d}r},
\end{split}
\end{equation}
where $r$ comes from the change of variables $r = g^2/(4\bar{\theta})$.

The integrals with respect to $\chi$ and $r$ depend on the collision kernel. For the variable-hard-sphere (VHS) model, the collision kernel is
\begin{equation}\label{eq:vhs}
B(g,\chi) = g^{\nu} \sin \chi \quad \text{for } \nu \in [0,1].
\end{equation}
In this case, the underlined integral in the equation above can be computed explicitly as
\begin{equation} \label{eq:underlined}
\begin{split}
& 2^{\nu+1} (\delta_{0,l_2} + 1) \bar{\theta}^{\nu/2} \int_0^{+\infty} L_{n_2}^{(l_2 + 1/2)}(r) L_{n_2'}^{(l_2 + 1/2)}(r) r^{l_2 + (\nu+1)/2} \exp(-r) \,\mathrm{d}r \\
={} & (-1)^{n_2 + n_2'} 2^{\nu+1} (\delta_{0l_2} + 1) \bar{\theta}^{\nu/2} \Gamma\left(l_2 + 1 + \frac{\nu+1}{2} \right) \sum_{k=0}^{\min(n_2,n_2')} 
\begin{pmatrix} \nu/2 \\ n_2 - k \end{pmatrix}
\begin{pmatrix} \nu/2 \\ n_2' - k \end{pmatrix}
\begin{pmatrix} k + l_2 + (\nu+1)/2 \\ k \end{pmatrix}.
\end{split}
\end{equation}
Note that the coefficients $a_{lnn_1}^{n'}$ can all be precomputed before the simulation. In our implementation, we ignore the coefficient $\bar{\theta}^{\nu/2}$ since it only introduces a universal constant to the indicator.

\subsection{Adaptive strategy} \label{sec:adaptive}
With the error indicator defined by \eqref{eq:indicator}, we compute this quantity for the distribution function on each spatial grid cell after every time step. For distribution functions with a large indicator, we increase the value of $M_0$ at the next time step, and vice versa. In our implementation, in order that the numerical solution does not oscillate due to the self-adaptation, we would like to maintain the stability of the collision model by avoiding the drastic change of the value of $M_0$. To this end, we adopt the following two strategies:
\begin{itemize}
\item Instead of a single threshold for the indicator like in most adaptive methods, we introduce two thresholds $\epsilon_1$ and $\epsilon_2$. If the error indicator of a certain distribution function lies between $(\epsilon_1, \epsilon_2)$, we keep the value of $M_0$ unchanged.
\item The value of $M_0$ changes only by $1$ at each time step. More precisely, if the error indicator exceeds $\epsilon_2$, we increase $M_0$ by $1$; if the error indicator falls below $\epsilon_1$, we reduce $M_0$ by $1$.
\end{itemize}
In general, if the range of the interval $(\epsilon_1, \epsilon_2)$ is wider, $M_0$ is more stable and the algorithm is less adaptive. A larger lower bound $\epsilon_1$ makes it easier for $M_0$ to drop; a larger upper bound $\epsilon_2$ makes it harder for $M_0$ to increase. In many applications, the bounds do not need to be too tight since the first few moments (e.g., density, velocity, and temperature) are usually not very sensitive about small changes of the collision models. This is why some simpler collision models such as the ES-BGK and the Shakhov models can still provide decent numerical results for these macroscopic variables.

To select the proper values of $\epsilon_1$ and $\epsilon_2$, we adopt the following strategy:
\begin{enumerate}
\item Do a test run with a small $M_0$ (e.g. $M_0 = 3$) being fixed without self-adaptation. Calculate the indicators for all time steps on all grid cells, and find its maximum value $\epsilon_{\max}$.
\item Use $\epsilon_{\max}$ as a reference value and choose the initial guesses of $\epsilon_1$ and $\epsilon_2$. They should be less than but not too far away from $\epsilon_{\max}$, e.g. $\epsilon_1 = \epsilon_{\max}/4$ and $\epsilon_2 = \epsilon_{\max}/2$.
\item Do a test run for the self-adaptive algorithm with the chosen $\epsilon_1$ and $\epsilon_2$, and then reduce $\epsilon_1$ (e.g. set $\epsilon_1$ to be $\epsilon_1 / 2$).
\item Do another test run for the current $\epsilon_1$ and $\epsilon_2$.
\item Compare the results of the two most recent runs. Go to the next step if the two results are sufficiently close to each other; otherwise, reduce $\epsilon_1$ again and return to the previous step.
\item Keep $\epsilon_1$ fixed and reduce the value of $\epsilon_2$ (e.g. set $\epsilon_2$ to be $\epsilon_2 / 2$).
\item Compare the results of the two most recent runs. Stop if the two results are sufficiently close to each other; otherwise, reduce the value of $\epsilon_2$ and check if $\epsilon_1 < \epsilon_2$. If so, return to Step 6; otherwise, reduce 
$\epsilon_1$ and return to Step 3.
\end{enumerate}
Here, the purpose of the first step is to provide a general range of the acceptable error indicator. Then, based on the initial guess of $\epsilon_1$ and $\epsilon_2$ in the second step, we first determine the lower bound by reducing $\epsilon_1$ until the solution looks stable, and then apply the same approach to $\epsilon_2$ to find a suitable upper bound. All test runs can be carried out on a coarse grid to save computational time. An example will be given in Section \ref{sec:colliding_flow}.

\subsection{Outline of the algorithm}
As a summary, below we list out the general steps of our algorithm:
\begin{enumerate}
\item Precompute the coefficients $s_{ln}^{n'}$ according to \eqref{eq:sn}.
\item Precompute the coefficients $A_{lnn'}^{l_1 l_2 m_2 n_2}$ according to \eqref{eq:A_recur}.
\item Precompute the coefficients $a_{lnn_1}^{n'}$ according to \eqref{eq:a} with the underlined term replaced by \eqref{eq:underlined}.
\item \label{itm:solve} Solve the Boltzmann equation by one time step. Terminate if the final time is reached.
\item \label{itm:bounds} For each distribution function, use \eqref{eq:h} and \eqref{eq:coeG1} to find the bounding functions $h$, $h^{(1)}$ and $h^{(2)}$, which bound $f^{(2)} - \mM^{(2)}$, $f^{(1)} - \mM^{(1)}$ and $\mM^{(2)}$, respectively.
\item \label{itm:indicator} Use \eqref{eq:Qabs_coef}\eqref{eq:a_lnn1}\eqref{eq:tQ}\eqref{eq:indicator} to compute the error indicator for each collision term.
\item Perform self-adaptation on each grid cell: if the error indicator is greater than the threshold $\epsilon_1$, we increase $M_0$ by $1$; if the error indicator is less than the threshold $\epsilon_2$, we decrease $M_0$ by $1$ if $M_0 > 3$.
\item Return to Step \ref{itm:solve}.
\end{enumerate}
In the algorithm, we have required that $M_0$ is no less than $3$, which is the smallest $M_0$ that includes the heat flux in the quadratic part of the collision operator. This ensures that the Navier-Stokes limit can always be correctly captured. The computation of the indicator lies in Steps \ref{itm:bounds} and \ref{itm:indicator}. In these two steps, the computation of \eqref{eq:Qabs_coef} and \eqref{eq:a_lnn1} has complexity $O(M^4)$, and the computation \eqref{eq:coeG1} requires $O(|\Omega|M^3)$ operations, where $|\Omega|$ denotes the number of quadrature points on the sphere. These parts take up most of the computational time in Steps \ref{itm:bounds} and \ref{itm:indicator}. Note that the computational cost of the indicator depends on $M$ instead of $M_0$, since the computation of $\Qabs(\cdot, \cdot)$ involves the complete distribution function instead of only the low-frequency part, so that the error due to the inaccuracy of the high-frequency part can be captured. Nevertheless, as we will see in the next section, such a cost is quite small compared to the evaluation of the collision operator for a large $M_0$.
%%% Local Variables: 
%%% mode: latex
%%% TeX-master: "article"
%%% End: 

\section{Numerical scheme and experiments}
\label{sec:num}
We are now ready to integrate the adaptation technique into the Boltzmann solver and carry out numerical experiments. In what follows, we will first brief our numerical algorithm to solve the system \eqref{eq:semi_discrete_SSP}, and then present several numerical experiments to demonstrate the effectiveness of the proposed indicator. 

\subsection{Numerical scheme}
For convenience, we will only provide the numerical algorithm for the
spatially one-dimensional case (the velocity space is still three-dimensional), where we assume that
\begin{equation}
    \label{eq:1D}
    \frac{\partial \mathbf{f}}{\partial x_2}  =  \frac{\partial \mathbf{f}}{\partial x_3} = 0.
\end{equation}
The algorithm can be naturally generalized to the multi-dimensional
case with uniform grids. Suppose the spatial domain
$\Omega \subset \mathbb{R}$ is discretized by a uniform grid with cell
size $\Delta x$. Using $\mathbf{f}_j^n$ to
approximate the average of $\mathbf{f}$ over the
$j$th grid cell $[x_{j-1/2}, x_{j+1/2}]$ at time $t^n$, we can solve
the system \eqref{eq:semi_discrete_SSP} by the following finite volume
method with time step size $\Delta t$:
 \begin{equation}
     \label{eq:numerical_sch}
     \mathbf{f}_j^* = \mathbf{f}_j^n - \frac{\Delta t}{\Delta x} [\mathbf{F}_{j+1/2}^n -\mathbf{F}_{j-1/2}^n], \qquad \mathbf{f}_j^{n+1} = \mathbf{f}_j^* + \Delta t \widetilde{\mathbf{Q}}(M_0; \mathbf{f}_j^n),
 \end{equation}
 where $\widetilde{\mathbf{Q}}(M_0; \mathbf{f}_j^n)$ is the modified
 collision operator defined in \eqref{eq:semi_discrete_SSP}. The
 numerical fluxes $\mathbf{F}_{j\pm 1/2}^n$ are chosen according to
 the HLL scheme \cite{HLL}:
\begin{equation}
    \label{eq:HLL}
    \mathbf{F}_{j+1/2}^n = 
    \frac{\lambda^R\mathbf{A}_1 \mathbf{f}_j^n - \lambda^L\mathbf{A}_1 \mathbf{f}_{j+1}^n + \lambda^R \lambda^L \Big(\mathbf{f}_{j+1}^n - \mathbf{f}_j^n\Big)}{\lambda^R - \lambda^L},
\end{equation}
where $\lambda^R$ and $\lambda^L$ are the minimum and maximum
eigenvalues of $\mathbf{A}_1$, respectively. Precisely, we have

\begin{equation}
\lambda^L = \bar{u}_1 - C_{M+1} \sqrt{\bar{\theta}} , \qquad \lambda^R = \bar{u}_1 + C_{M+1} \sqrt{\bar{\theta}},
\end{equation}
and $C_{M+1}$ is the largest zero of the Hermite polynomial of degree
$M+1$. Here the parameters are always chosen such that $\lambda^L < 0$
and $\lambda^R > 0$ to avoid advection only in one direction. Besides,
the time step size is determined by the CFL condition
\begin{equation}
    \label{eq:CFL}
    \Delta t \frac{|\bar{u}_1| + C_{M+1} \sqrt{\bar{\theta}}}{\Delta x} = {\rm CFL} < 1.
\end{equation}
In our actual implementation, we have upgraded this scheme to the second
order by linear reconstruction with minmod limiter, Heun's time
integrator, and the Strang splitting. Such strategies are standard
techniques and can be found in many textbooks
(e.g. \cite{Leveque2002}).

\subsection{One-dimensional numerical examples}
In this section, we present two numerical examples, both of which use
the variable hard sphere model with $\nu = 5/9$ (see
\eqref{eq:vhs}). In our simulation, in order to prevent the
computational cost from getting out of control, we set a cap for the
value of $M_0$ to be $15$, and we use the non-adaptive results with $M_0 = 15$ being fixed
as our reference solution. All the numerical tests in this section
are carried out on a desktop with CPU model
Intel\textsuperscript{\textregistered} Core\texttrademark{} i7-7600U.

\subsubsection{Colliding flow} \label{sec:colliding_flow}
We consider the colliding flow with the initial condition 
\begin{equation}
    \label{eq:ini_ex1_1}
    f(x, \boldsymbol{v}, 0) = \frac{\rho(x)}{(2 \pi \theta(x))^{3/2}} \exp\left(-\frac{|\bv - \boldsymbol{u}(x)|^2}{2 \theta(x)}\right)
\end{equation}
with 
\begin{equation}
    \label{eq:ini_ex1_2}
    \rho(x) = 1, \qquad \bu(x) = \left\{\begin{array}{cc}
    (1, 0, 0)^T, & {\rm if}~ x < 0,  \\
    (-1, 0, 0)^T, & {\rm if}~ x > 0, 
    \end{array}  \right. \qquad \theta(x) = 1/3.
\end{equation}
We scale the collision term such that the Knudsen number equals
$0.5$. The initial condition consists of two equilibrium flows with
the same temperature moving in opposite directions, and it is expected
that the collision of the two Maxwellians will create some
non-equilibrium effects, which require a relatively large $M_0$ to
accurately capture the flow states.

The computational domain is set as $[-20, 20]$. In \eqref{eq:truncated_series}, we choose
$M$ to be $30$, and $\bar{\bu}$ and $\bar{\theta}$ are set to be
$\mathbf{0}$ and $1$, respectively. To determine the values of the thresholds $\epsilon_1$ and $\epsilon_2$, we follow the strategy in Section \ref{sec:adaptive} and carry out a test run for $M_0$ fixed to be $3$ up to $t = 15$. The cell size is chosen to be $\Delta x = 0.4$ in all the test runs below. The numerical solutions of this test run at different times are given in Figure \ref{fig:ex1_M03}, which plots the density $\rho$ and heat flux $\boldsymbol{q}$ of the gas, which are defined by
\begin{equation}
\label{eq:moments1}
\rho = \int_{\mathbb{R}^3} f(\bv) \,\mathrm{d}\bv, \qquad
\boldsymbol{q} = \frac{1}{2} \int_{\mathbb{R}^3} |\bv-\bu|^2 (\bv - \bu) f(\bv) \,\mathrm{d}\bv,
\end{equation}
where $\bu$ is the average velocity of gas molecules:
\begin{equation} \label{eq:moments_u}
\bu = \frac{1}{\rho} \int_{\mathbb{R}^3} \bv f(\bv) \,\mathrm{d}\bv.
\end{equation}
For this one-dimensional flow, only the first component of $\boldsymbol{q}$ is plotted. Due to the insufficient resolution of the solution and the small value of $M_0$, the peak values of both density and heat flux are not well captured, but the general behavior of the flow is still qualitatively correct: the collision of the flow generates two shock waves moving in opposite directions, and the heat flux is nonzero inside these shock waves.
During the test run, we record the maximum value of the indicator, which turns out to be
\begin{displaymath}
\epsilon_{\max} = 15.6.
\end{displaymath}
This leads us to the following initial guess of $\epsilon_1$ and $\epsilon_2$:
\begin{displaymath}
(\epsilon_1, \epsilon_2) = (4,8) \approx (\epsilon_{\max}/4, \epsilon_{\max}/2).
\end{displaymath}

\begin{figure}[!htb]
% \colorbox{black}
  \centering
  \subfloat[$\rho$ ]{\includegraphics[width=0.49\textwidth,clip]{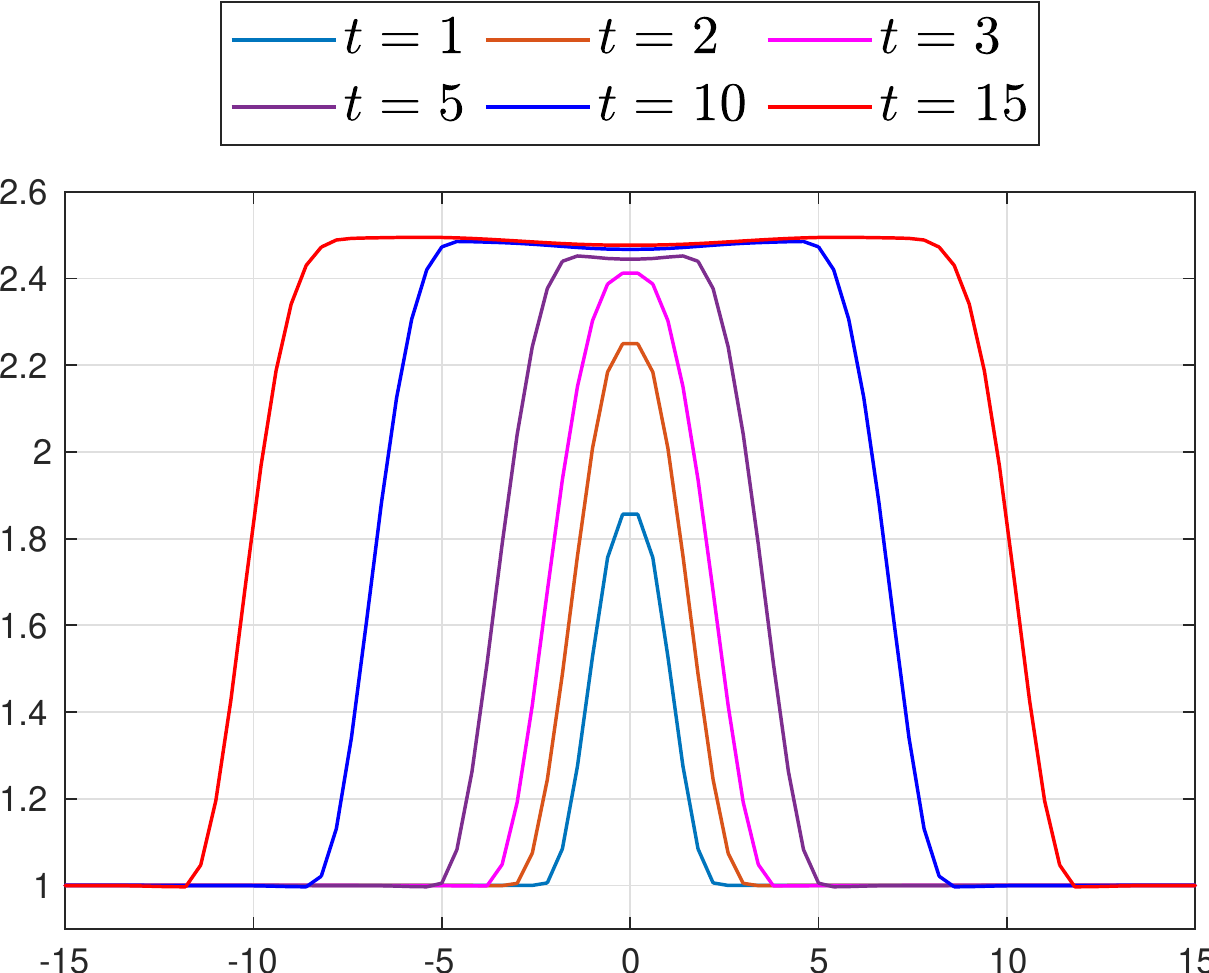}}\hfill
    \subfloat[$q_1$]{\includegraphics[width=0.49\textwidth,clip]{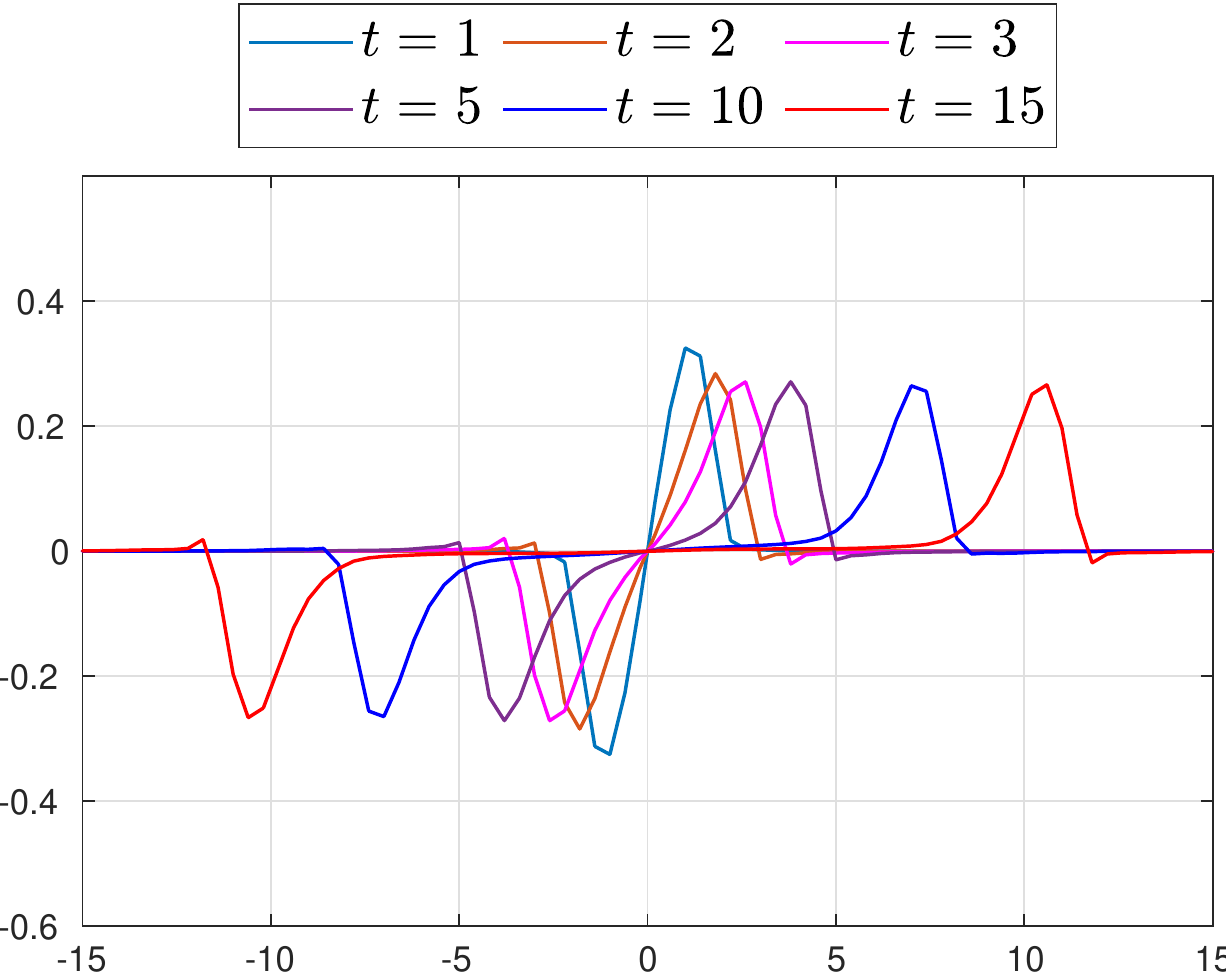}} 
  \caption{Plots of density and heat flux for the colliding flow with $M_0 = 3$ (fixed) at different times.}
  \label{fig:ex1_M03}
\end{figure}

We now carry out a test with adaptive collision operators using this pair of parameters, and then reduce the value of $\epsilon_1$ to $2$ and run another test. The comparison of these two tests is given in Figure \ref{fig:ex1_ep_1}, where we only present the non-equilibrium variable $q_1$ that shows a more significant difference than the equilibrium variables. It can be seen from Figure \ref{fig:ex1_ep_1_M0} that the smaller value of $\epsilon_1$ leads to slightly greater $M_0$ inside the shock wave. Since the graph of $q_1$ in Figure \ref{fig:ex1_ep_1_q1} still shows quite some difference between the two solutions, we further reduce $\epsilon_1$ by a half and carry out another test run. The comparison of the solutions is given in Figure \ref{fig:ex1_ep_2}. We are now satisfied with the small difference and will fix the value of $\epsilon_1$ to be $1$.

\begin{figure}[!htb]
% \colorbox{black}
  \centering
  \subfloat[$q_1$. Solid lines: $(\epsilon_1,\epsilon_2) = (4, 8)$. Dashed lines: $(\epsilon_1,\epsilon_2) = (2,8)$.\label{fig:ex1_ep_1_q1}]{\includegraphics[width=0.49\textwidth,height=0.35\textwidth, clip]{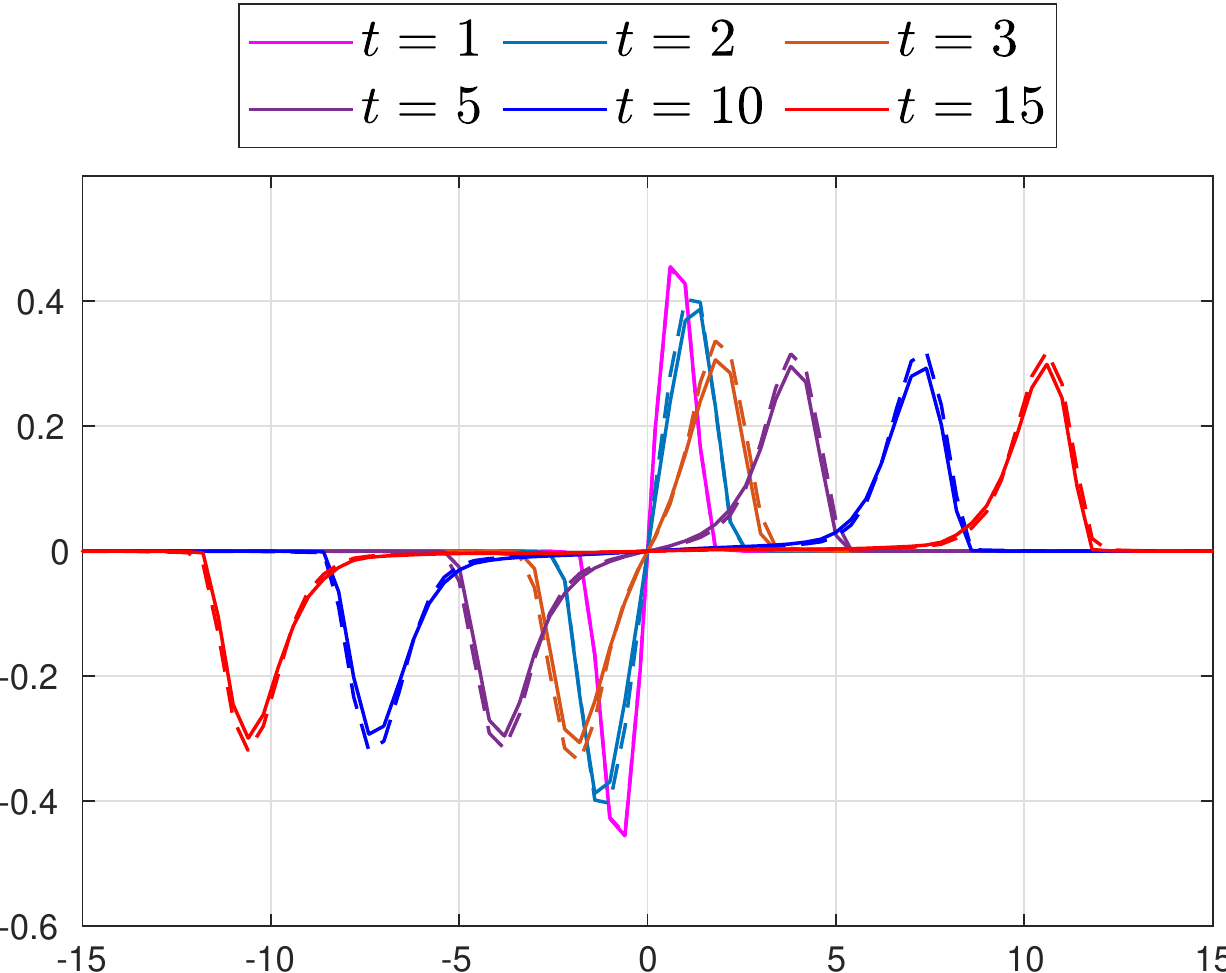}}\hfill
  \subfloat[$M_0$. Solid lines: $(\epsilon_1,\epsilon_2) = (4, 8)$. Dashed lines: $(\epsilon_1,\epsilon_2) = (2,8)$.\label{fig:ex1_ep_1_M0}]{\includegraphics[width=0.49\textwidth,height=0.35\textwidth, clip]{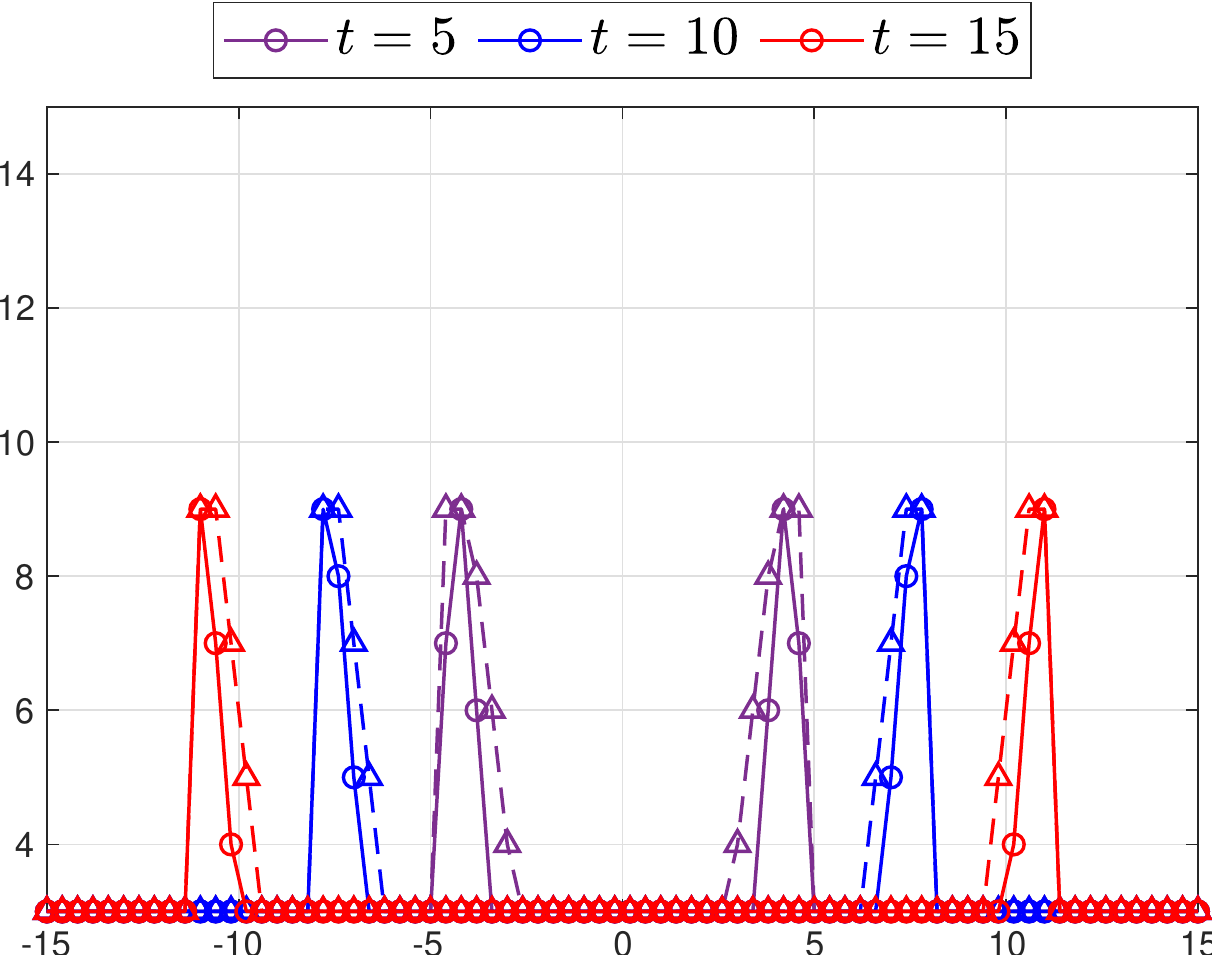}} 
   \caption{Profiles of heat flux and distributions of $M_0$ for the colliding flow with two different pairs of indicator thresholds $(\epsilon_1,\epsilon_2) = (4, 8)$ and $(\epsilon_1,\epsilon_2) = (2,8)$.}
  \label{fig:ex1_ep_1}
\end{figure}

\begin{figure}[!htb]
% \colorbox{black}
  \centering
  \subfloat[$q_1$. Solid lines: $(\epsilon_1,\epsilon_2) = (2, 8)$. Dashed lines: $(\epsilon_1,\epsilon_2) = (1,8)$.]{\includegraphics[width=0.49\textwidth,height=0.35\textwidth, clip]{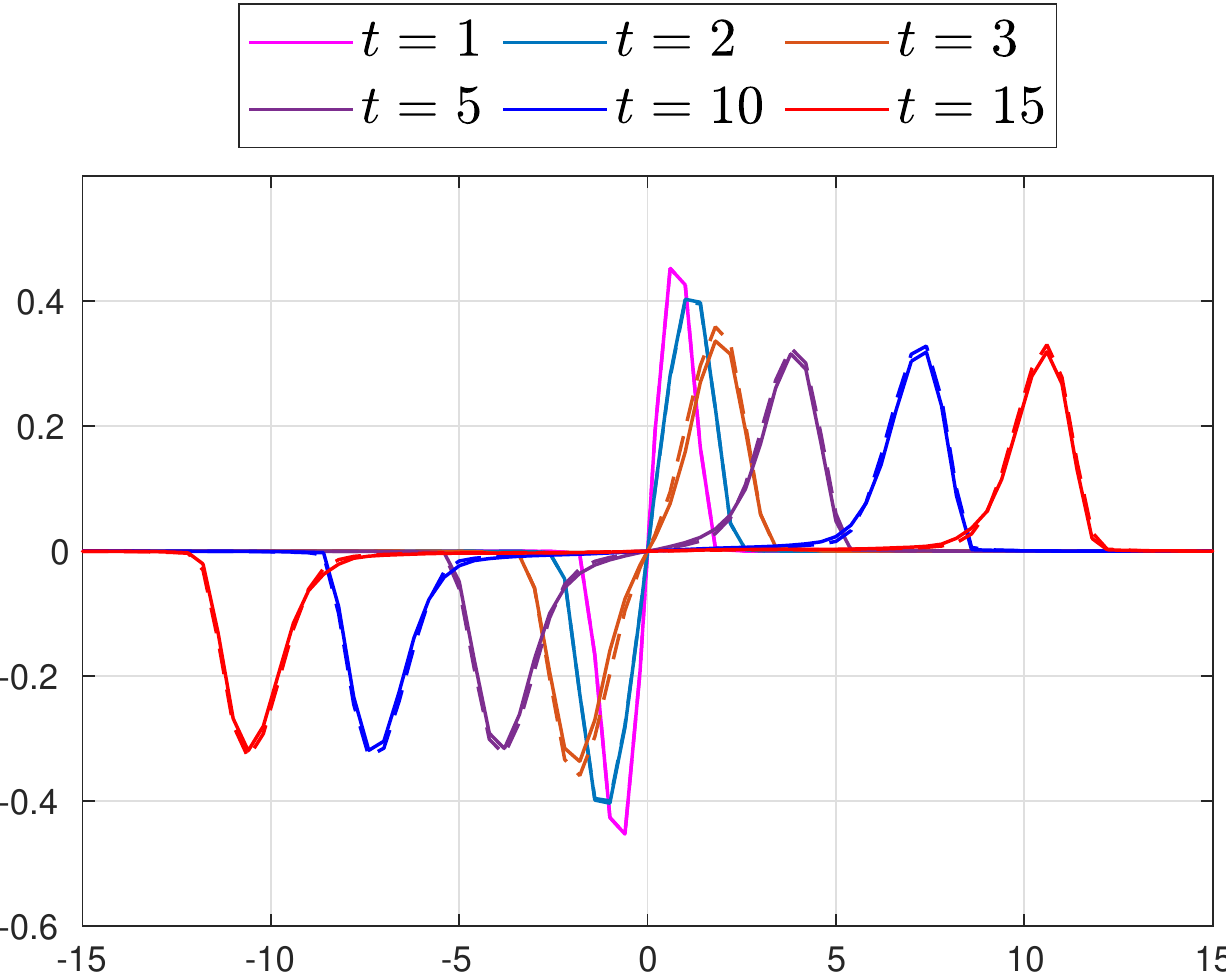}}\hfill
  \subfloat[$M_0$. Solid lines: $(\epsilon_1,\epsilon_2) = (2, 8)$. Dashed lines: $(\epsilon_1,\epsilon_2) = (1,8)$.]{\includegraphics[width=0.49\textwidth,height=0.35\textwidth, clip]{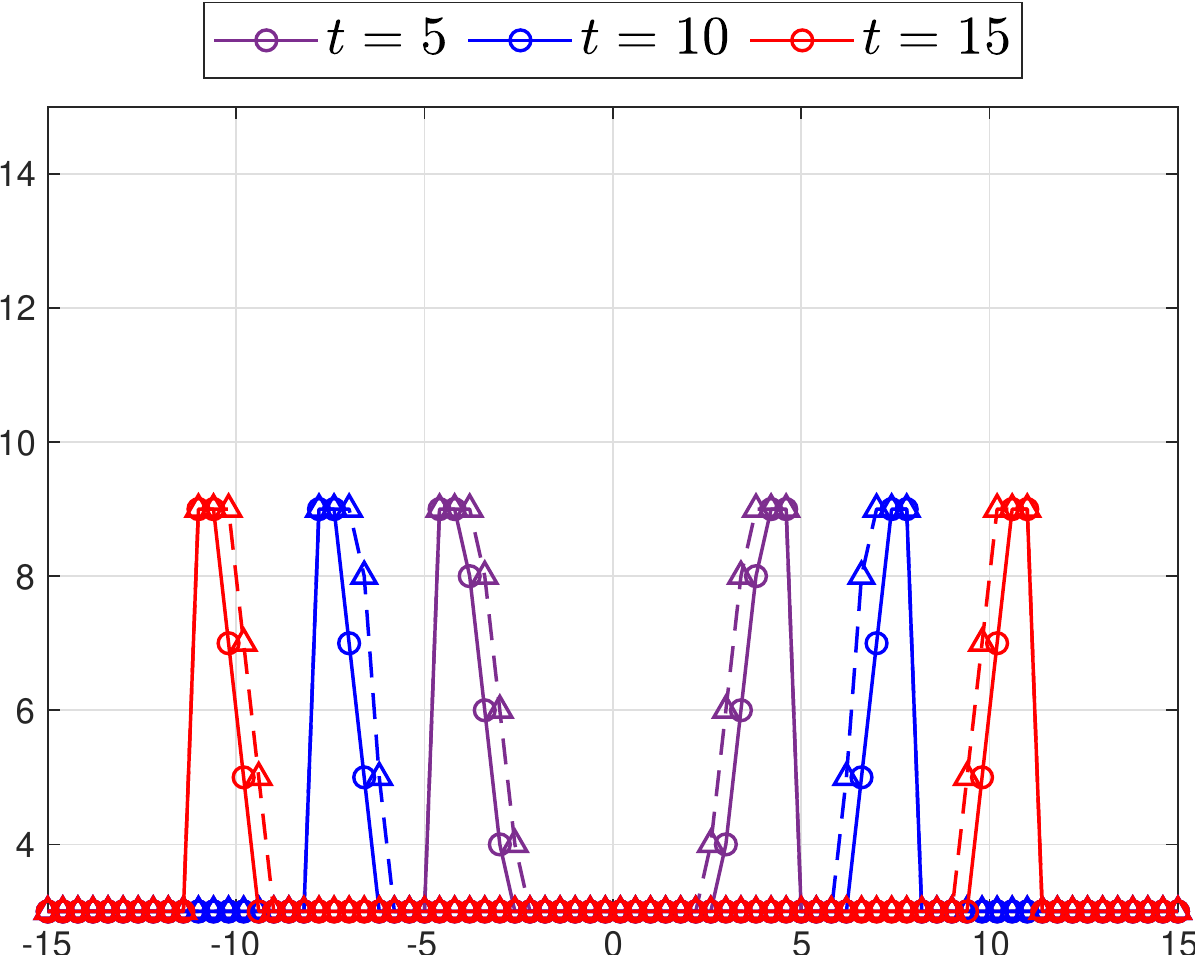}}
   \caption{Profiles of heat flux and distributions of $M_0$ for the colliding flow with two different pairs of indicator thresholds $(\epsilon_1,\epsilon_2) = (2, 8)$ and $(\epsilon_1,\epsilon_2) = (1,8)$.}
  \label{fig:ex1_ep_2}
\end{figure}

The selection of $\epsilon_2$ is done in a similar way. We reduce $\epsilon_2$ from $8$ to $4$ and compare the results with the parameters $(\epsilon_1, \epsilon_2) = (1,8)$ and $(\epsilon_1, \epsilon_2) = (1,4)$. As shown in Figure \ref{fig:ex1_ep_3}, in a few grid cells inside the shock wave, the values of $M_0$ have been significantly increased, while the solution of the heat flux does not change too much. We therefore fix the value of $\epsilon_2$ to be $4$. The computational times for these test runs are given in Table \ref{tab:ex1_ep}, which look affordable due to the coarse grid size.

\begin{figure}[!htb]
% \colorbox{black}
  \centering
  \subfloat[$q_1$. Solid lines: $(\epsilon_1,\epsilon_2) = (1, 8)$. Dashed lines: $(\epsilon_1,\epsilon_2) = (1,4)$.]{\includegraphics[width=0.49\textwidth,height=0.35\textwidth, clip]{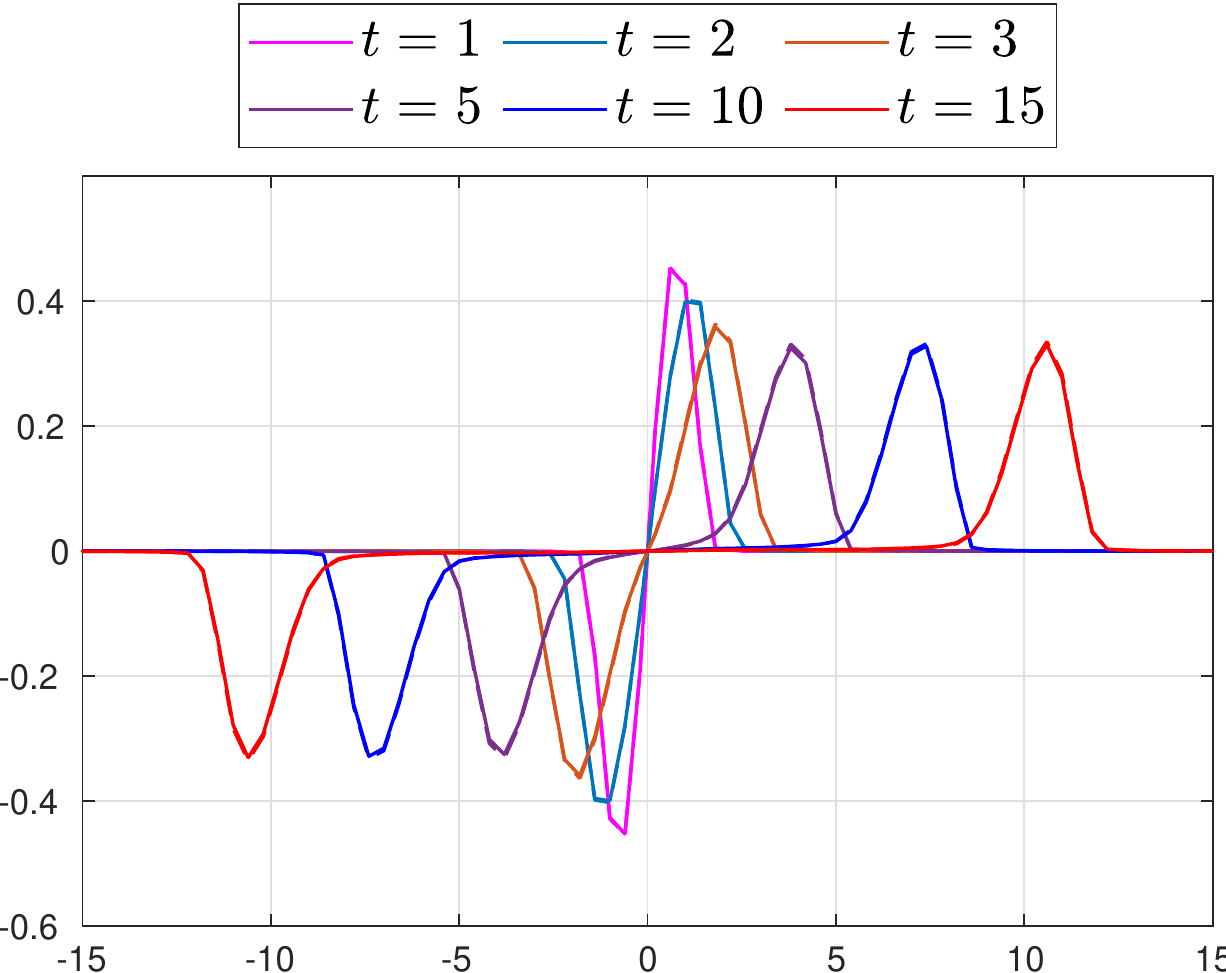}}\hfill
  \subfloat[$M_0$. Solid lines: $(\epsilon_1,\epsilon_2) = (1, 8)$. Dashed lines: $(\epsilon_1,\epsilon_2) = (1,4)$.]{\includegraphics[width=0.49\textwidth,height=0.35\textwidth, clip]{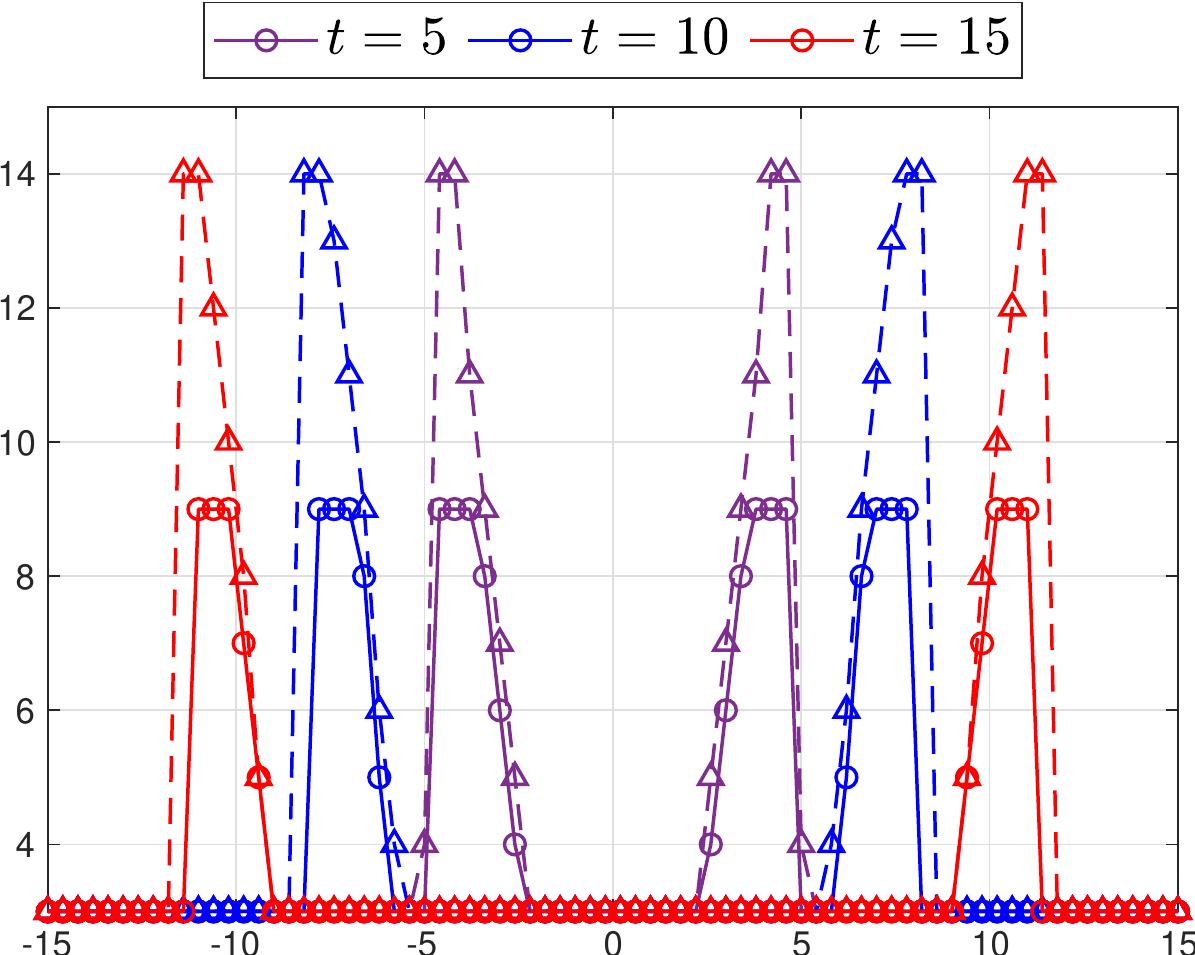}} 
   \caption{Profiles of heat flux and distributions of $M_0$ for the colliding flow with two different pairs of indicator thresholds $(\epsilon_1,\epsilon_2) = (1,8)$ and $(\epsilon_1,\epsilon_2) = (1,4)$.}
  \label{fig:ex1_ep_3}
\end{figure}

\begin{table}[!ht]
\centering
\def\arraystretch{1.5}
\caption{CPU times of the test run for the collision flow. The parameters $(0,+\infty)$ refers to the non-adaptive run with $M_0$ fixed to be $3$.}
\label{tab:ex1_ep}
\begin{tabular}{c|c@{\qquad}c@{\qquad}c@{\qquad}c@{\qquad}c@{\qquad}}
$(\epsilon_1, \epsilon_2)$ & $(0,+\infty)$ & $(4, 8)$ & $(2, 8)$ & $(1, 8)$ &$(1, 4)$\\
\hline
Total CPU time (s) & $240.62$ & $479.27$ & $510.55$ & $525.35$ & $724.82$
\end{tabular}
\end{table}

Next, we refine the grid and set the cell size to be $\Delta x = 0.1$. With $(\epsilon_1, \epsilon_2)$ chosen to be $(1,4)$, we rerun the simulation up to $t = 15$. The numerical solutions at different times are given in Figure
\ref{fig:ex1_sol}, including three equilibrium quantities (density
$\rho$, velocity $\bu$ and temperature $\theta$) and one
non-equilibrium moment (heat flux $\boldsymbol{q}$). The temperature $\theta$
is related to the distribution function by
\begin{equation} \label{eq:moments}
\theta = \frac{1}{3\rho} \int_{\mathbb{R}^3} |\bv-\bu|^2 f(\bv) \,\mathrm{d}\bv.
\end{equation}
For the velocity and heat flux, only their first components are
plotted due to the one-dimensional nature of the flow. Similar to our test runs,
the flow structure emerges from the middle of the domain due to the
interaction of the two Maxwellians, producing higher density and
temperature. Then two shock waves are formed and move in opposite
directions. After the two shock waves are separated, the center of the
domain returns to local equilibrium state. We would like to emphasize
that Figure \ref{fig:ex1_sol} includes two sets of solutions (the
reference solution and the self-adaptive solution) which almost coincide.

\begin{figure}[!htb]
% \colorbox{black}
  \centering
  \subfloat[$\rho$ ]{\includegraphics[width=0.49\textwidth,clip]{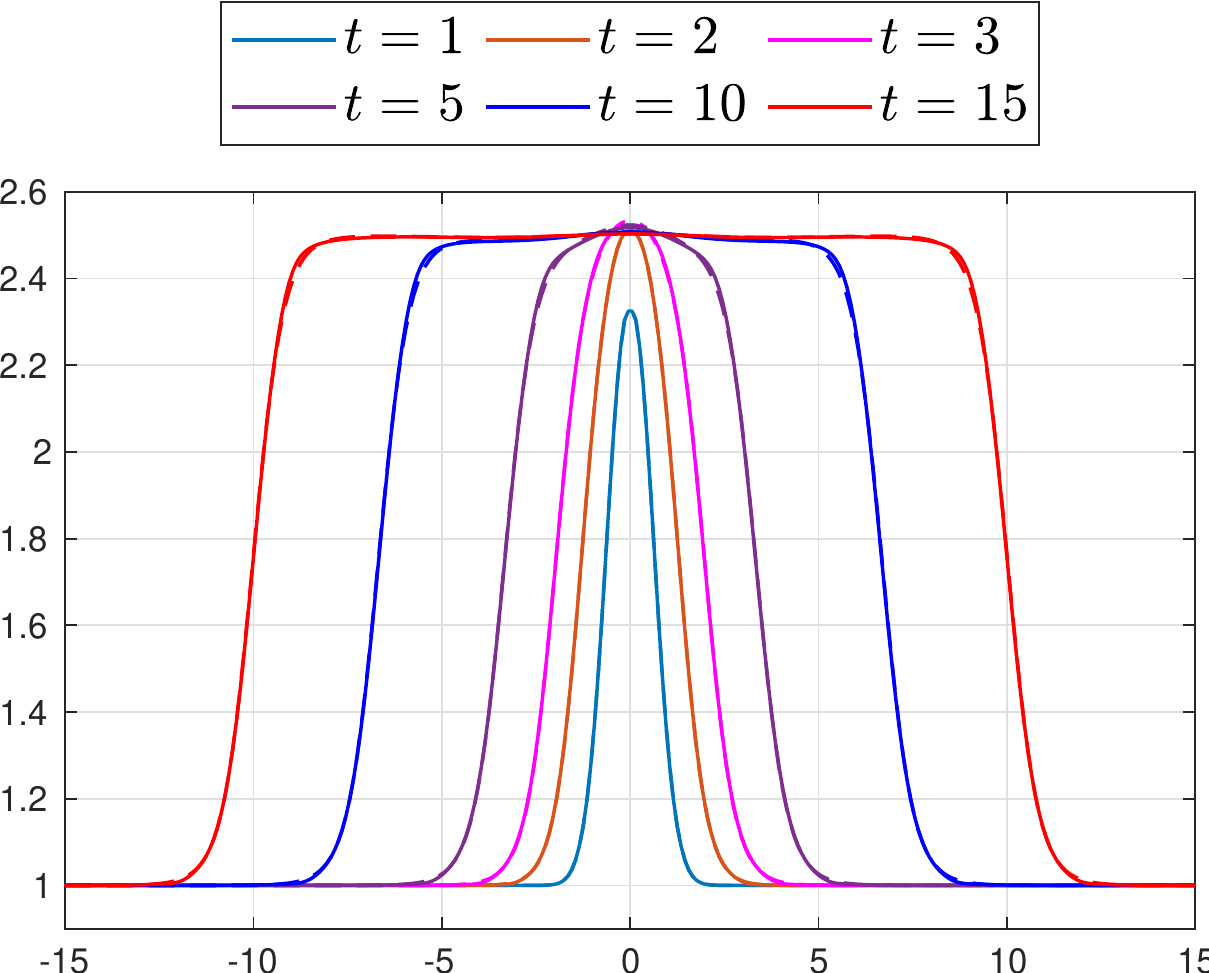}}\hfill
  \subfloat[$u_1$]{\includegraphics[width=0.49\textwidth,clip]{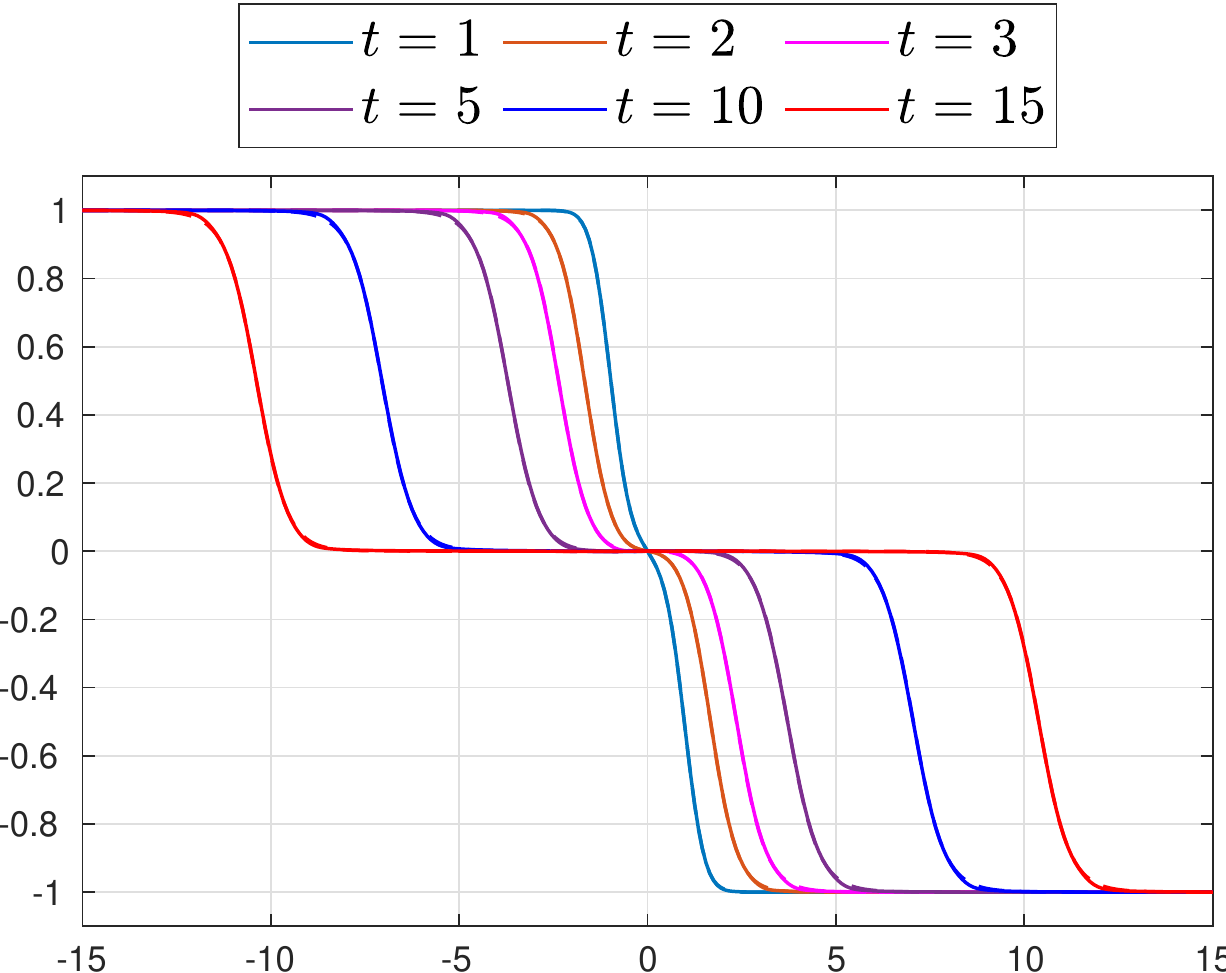}} \\
  \subfloat[$\theta$]{\includegraphics[width=0.49\textwidth,clip]{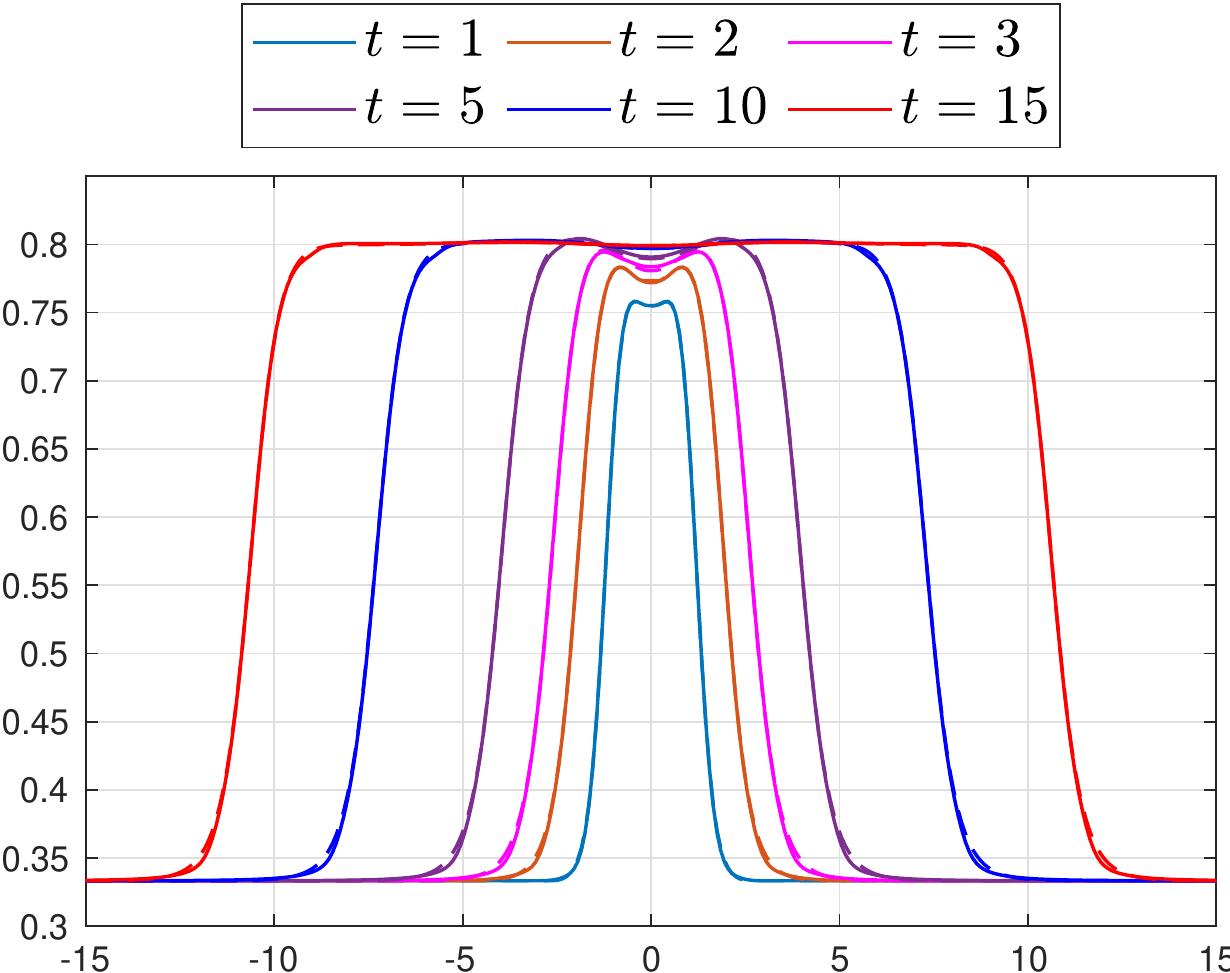}}\hfill
  \subfloat[$q_1$]{\includegraphics[width=0.49\textwidth,clip]{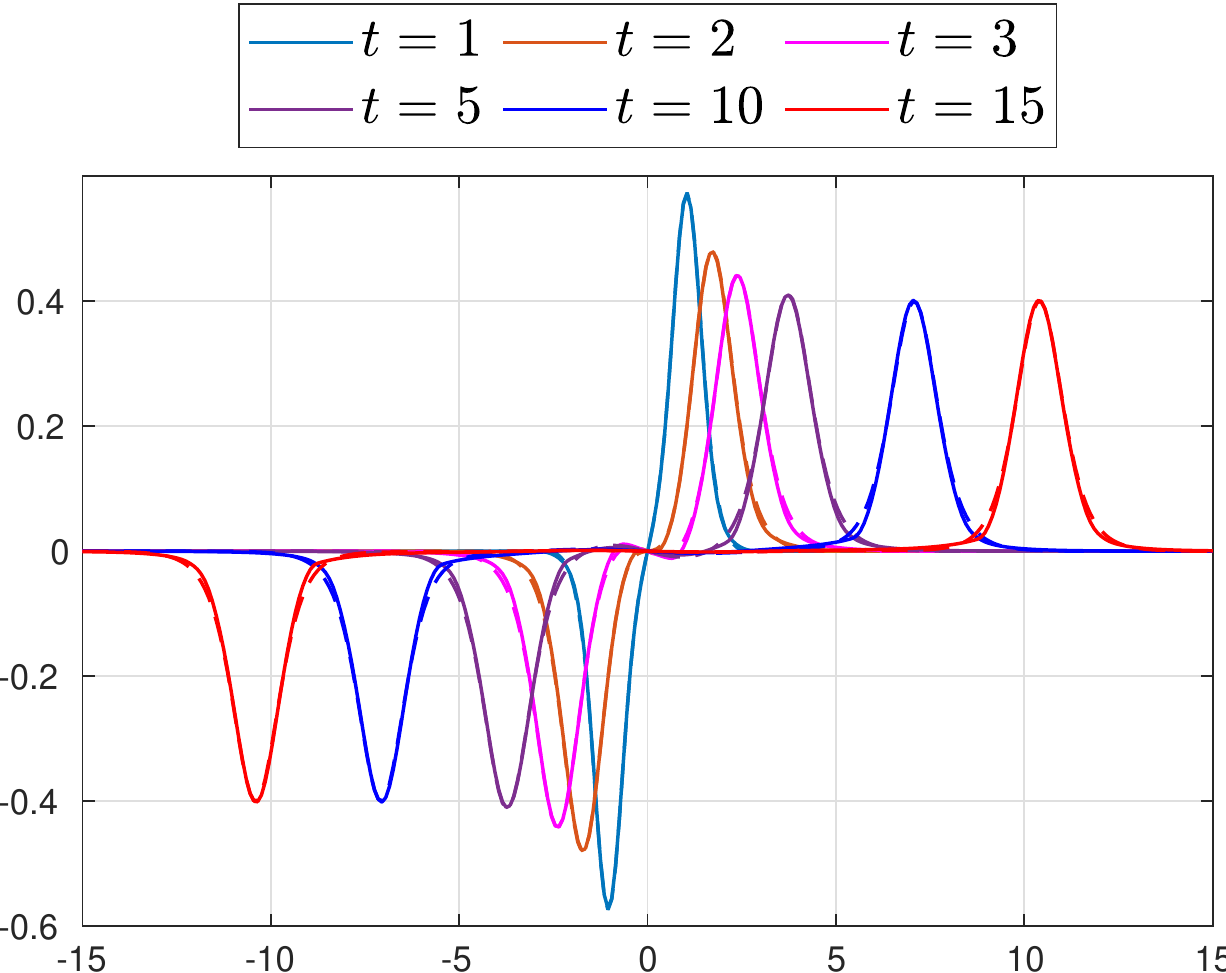}} 
  \caption{Solution of the colliding flow at different times. The solid lines are the numerical solution of the adaptive algorithm and the dashed lines are the reference solution.}
  \label{fig:ex1_sol}
\end{figure}

The evolution of the distribution of $M_0$ is provided in Figure
\ref{fig:ex1_M0}. Initially, we set $M_0 = 15$ on all grid cells. In a few
time steps, this drops to $3$ almost everywhere except the center of
the domain. Afterward, the evolution of $M_0$ well agrees with the
evolution of the non-equilibrium. During the simulation, most part of
the domain is in the local equilibrium state, where $M_0$ stays at its lowest value $3$, requiring much less computational cost. Consequently, as shown in
Table \ref{tab:ex1}, the total CPU time is significantly reduced
compared with the simulation using a uniform $M_0$. Moreover,
the evaluation of the error indicator only takes a
relatively small portion of the total computational time, which agrees
with the goal we set in Section \ref{sec:intro}. It is also worth mentioning that the total computational time is 3993.29s, which is even longer than the sum of all our test runs.

\begin{figure}[!htb]
% \colorbox{black}
  \centering
  \subfloat[$M_0, t = 1$]{\includegraphics[width=0.49\textwidth,clip]{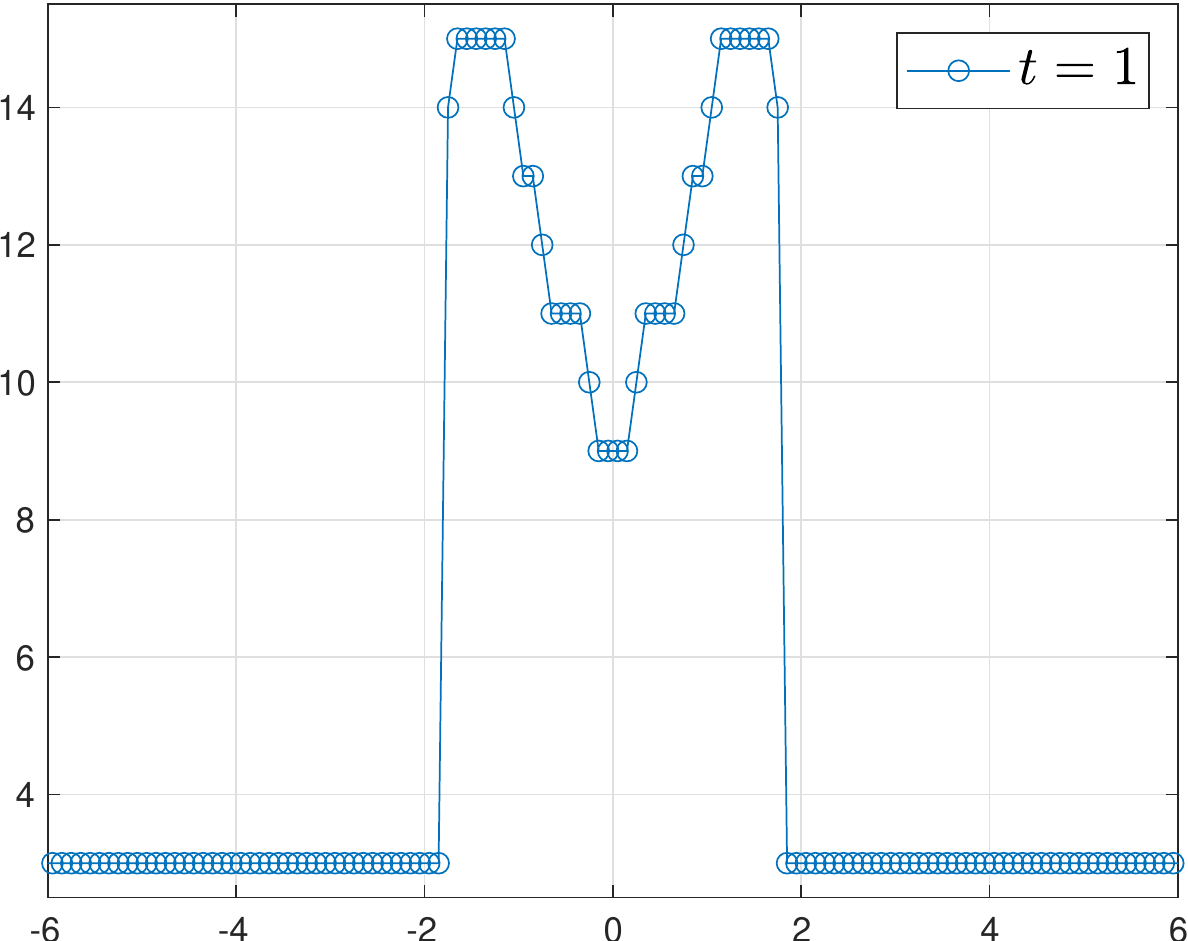}}\hfill
  \subfloat[$M_0, t = 2$]{\includegraphics[width=0.49\textwidth,clip]{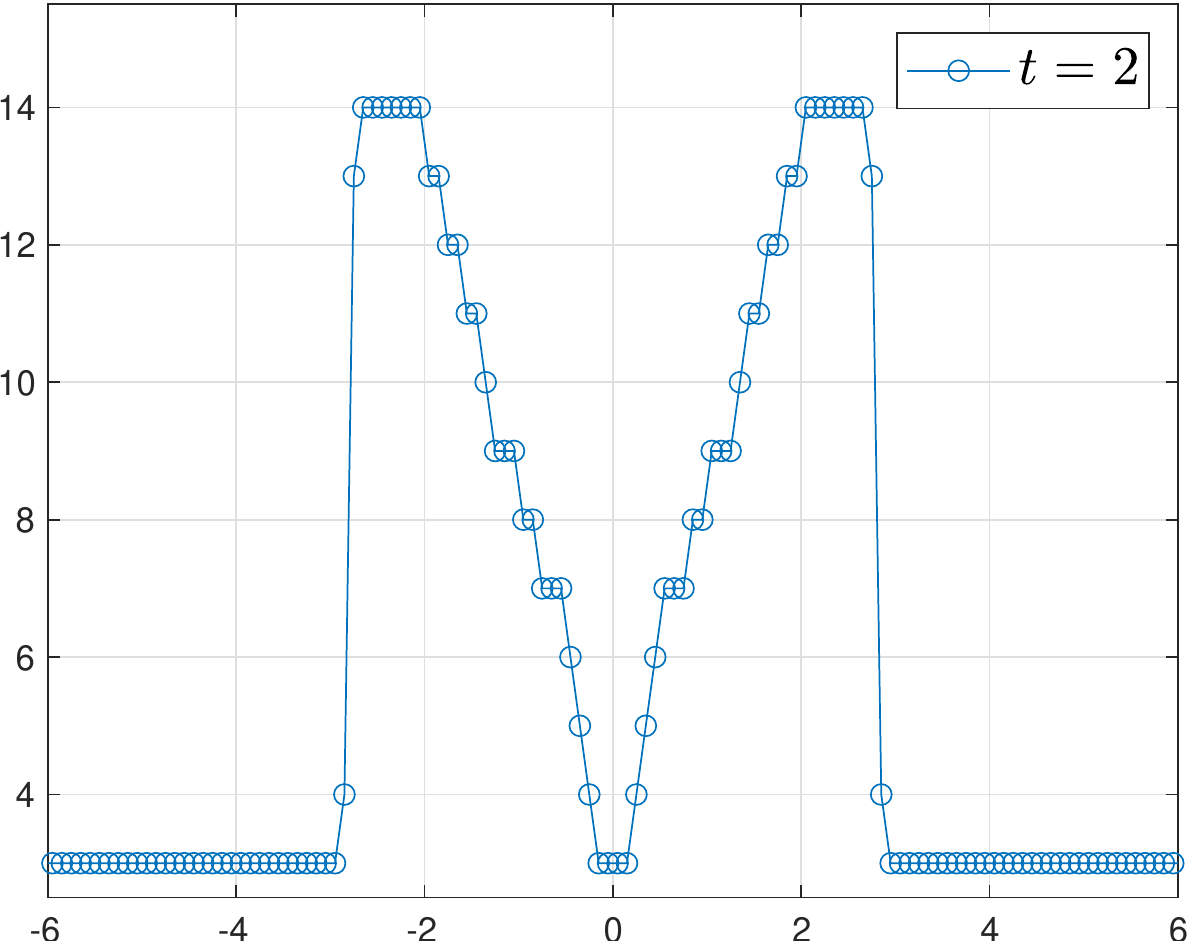}} \\
  \subfloat[$M_0, t = 3$]{\includegraphics[width=0.49\textwidth,clip]{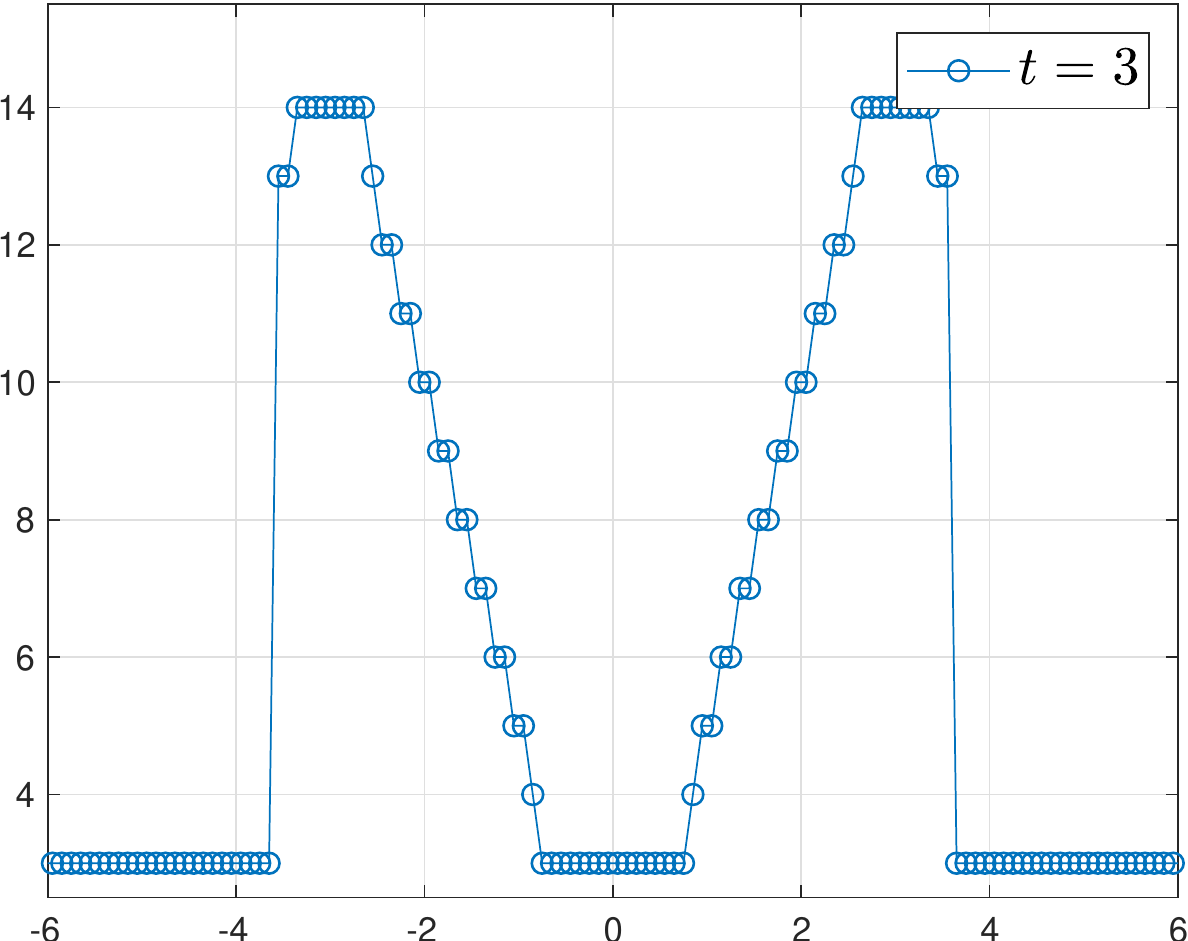}}\hfill
  \subfloat[$M_0, t = 5, 10, 15$]{\includegraphics[width=0.49\textwidth,clip]{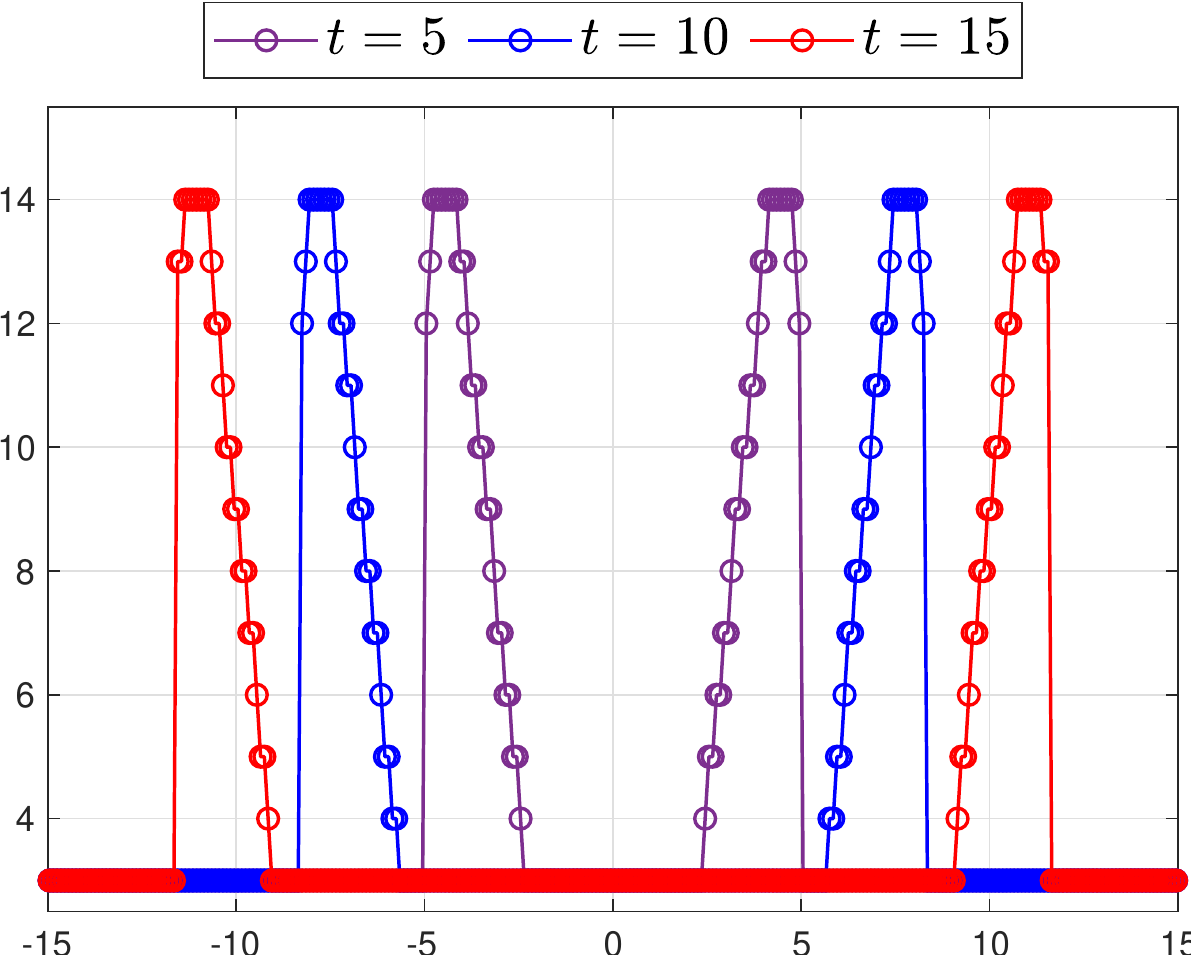}} 
   \caption{Distribution of $M_0$ for the colliding flow at different times.}
  \label{fig:ex1_M0}
\end{figure}

\begin{table}[!ht]
\centering
\caption{Statistical data for the colliding flow. $T_{\mathrm{ref}}$ and $T_{\mathrm{adp}}$ refer to the average CPU time per time step for the reference solution and the self-adaptive solution, respectively, and $T_{\mathrm{ind}}$ refers to the average CPU time per time step for the computation of the error indicator.}
\label{tab:ex1}
\begin{tabular}{c@{\qquad}c@{\qquad}c@{\qquad}c@{\qquad}c}
\hline
$T_{\mathrm{ref}}$ &  $T_{\mathrm{adp}}$ & $T_{\mathrm{ind}}$ & $1 - T_{\mathrm{adp}} / T_{\mathrm{ref}}$ & $T_{\mathrm{ind}} / T_{\mathrm{adp}}$ \\
\hline
$16.96$s & $1.06$s & $0.108$s & $93.7$\% & $10.1$\% \\
\hline
\end{tabular}
\end{table}

\subsubsection{Planar Couette flow} 
The planar Couette flow is a commonly used benchmark problem for the one-dimensional Boltzmann equation.  We assume that the gas between two infinite parallel plates has an initial temperature $\theta = 1$, and the two plates move in the opposite directions with velocities parallel to the plates. The speeds of both plates are $0.5$, and the distance between the two plates is $L = 1$. Both plates are assumed to be completely diffusive, meaning that for any particle hitting the wall, the reflected velocity is completely independent of the incident velocity. Instead, the distribution of the reflected velocity follows the Maxwellian with the wall velocity being the center and wall temperature being the variance. The implementation of the boundary condition in the Burnett spectral method has been detailed in \cite[Section 3.4]{Hu2020b}. The initial state of the fluid is set to be a uniform Maxwellian with density $\rho=1$, velocity $\bu = 0$ and temperature $\theta = 1$. Driven by the motion of the plates, the flow will reach a steady state as time approaches infinity. Here we choose the Knudsen number to be $0.5$, for which a strong non-equilibrium can be expected, especially on the boundary of the domain where the distribution function is discontinuous.

Numerically, we set $\bar{\bu} = \boldsymbol{0}$, $\bar{\theta} = 1$ and $M=30$ in \eqref{eq:truncated_series}. A uniform grid with $200$ cells is used for spatial discretization, and the thresholds of the error indicator are set to be $\epsilon_1 = 1$ and $\epsilon_2 = 8$. Figure \ref{fig:ex2_sol} shows the numerical solution of the four moments defined in \eqref{eq:moments1}\eqref{eq:moments_u} and \eqref{eq:moments}, and Figure \ref{fig:ex2_M0} shows the evolution of the parameter $M_0$. During the evolution to the steady state, some small differences between the self-adaptive solution and the reference solution can be observed. The discrepancy of the velocity profiles appears to be the most significant due to its small magnitude. In this example, large $M_0$ only appears near the boundary of the domain for small $t$, since the central part of the domain is still mostly in the initial equilibrium state. As the boundary effect propagates inward, the value of $M_0$ gradually increases. Interestingly, the distribution of $M_0$ reaches the ``steady state'' earlier than the fluid does. As shown in Figures \ref{fig:ex2_t1} and \ref{fig:ex2_steady}, at $t = 1.0$, while the fluid structure is still evolving, $M_0$ does not change with time any more. Compared to the example in Section \ref{sec:colliding_flow}, the non-equilibrium spreads more widely in this case, resulting in less reduction of the computational cost (see Table \ref{tab:ex2}). Nevertheless, the CPU time per time step is still reduced to nearly one-sixth. Here, the computation of the indicator takes a smaller portion since longer time is spent on the evaluation of the collision term.

\begin{figure}[!htb]
% \colorbox{black}
  \centering
  \subfloat[$\rho$ ]{\includegraphics[width=0.49\textwidth,clip]{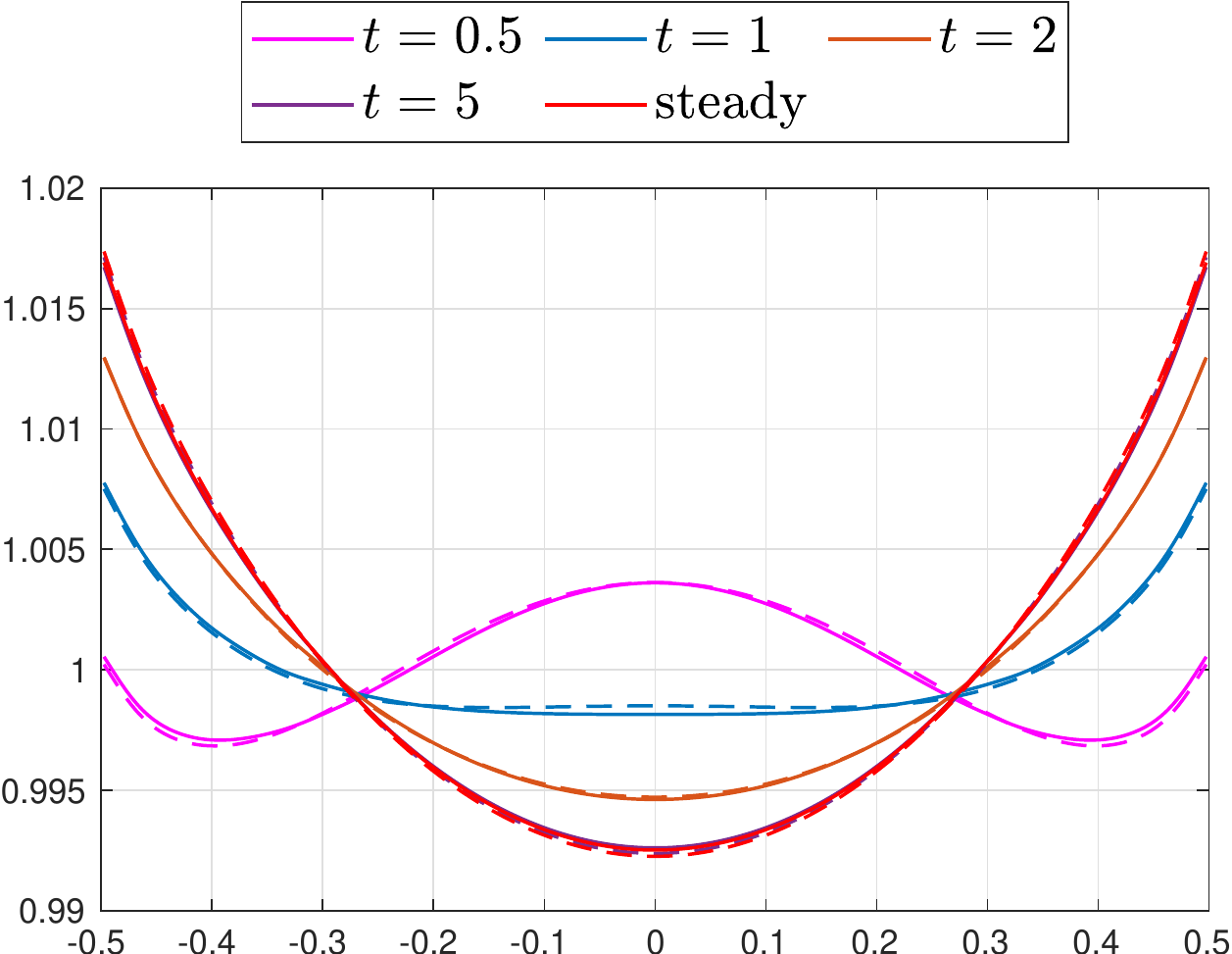}}\hfill
  \subfloat[$u_1$]{\includegraphics[width=0.49\textwidth,clip]{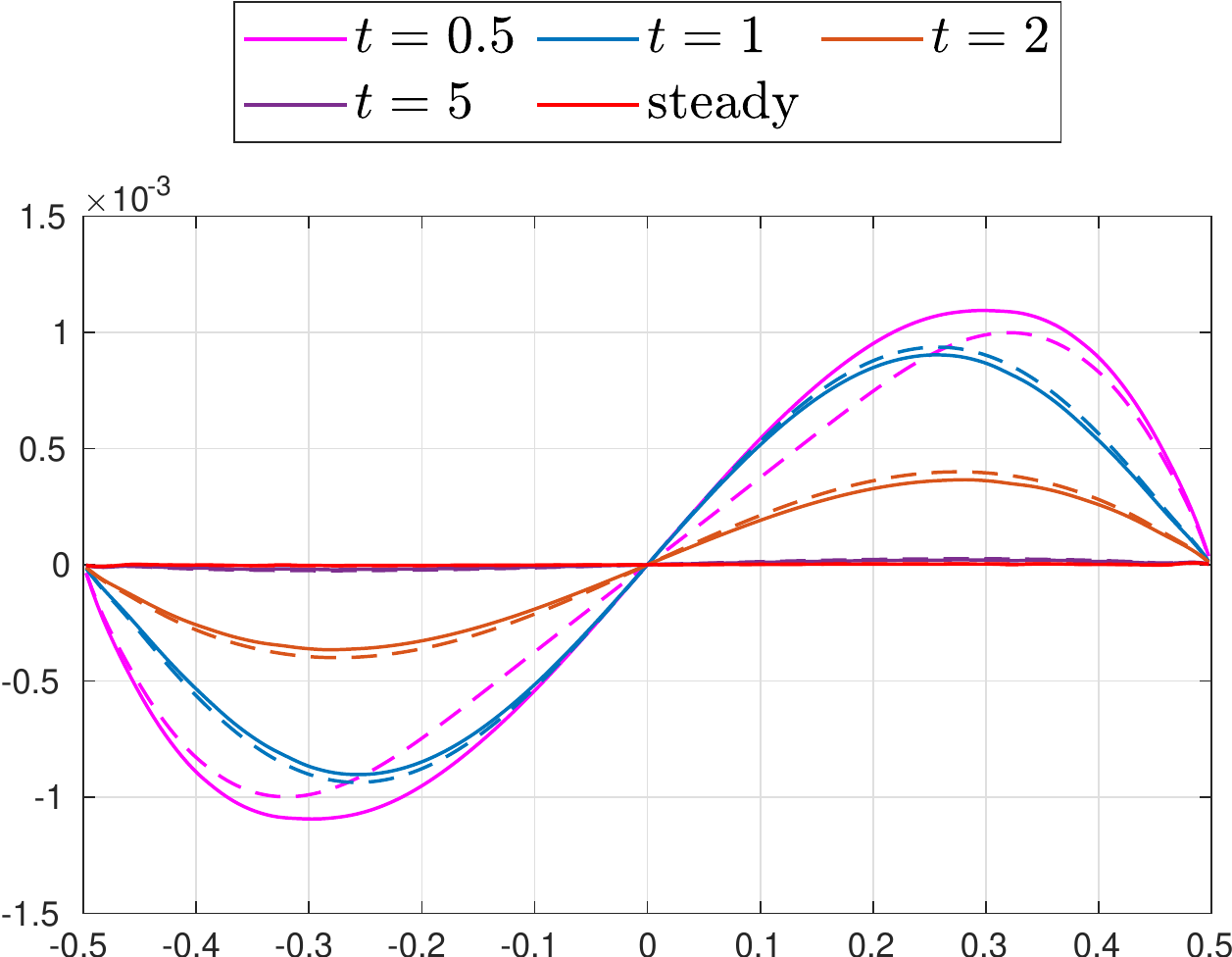}} \\
  \subfloat[$\theta$]{\includegraphics[width=0.49\textwidth,clip]{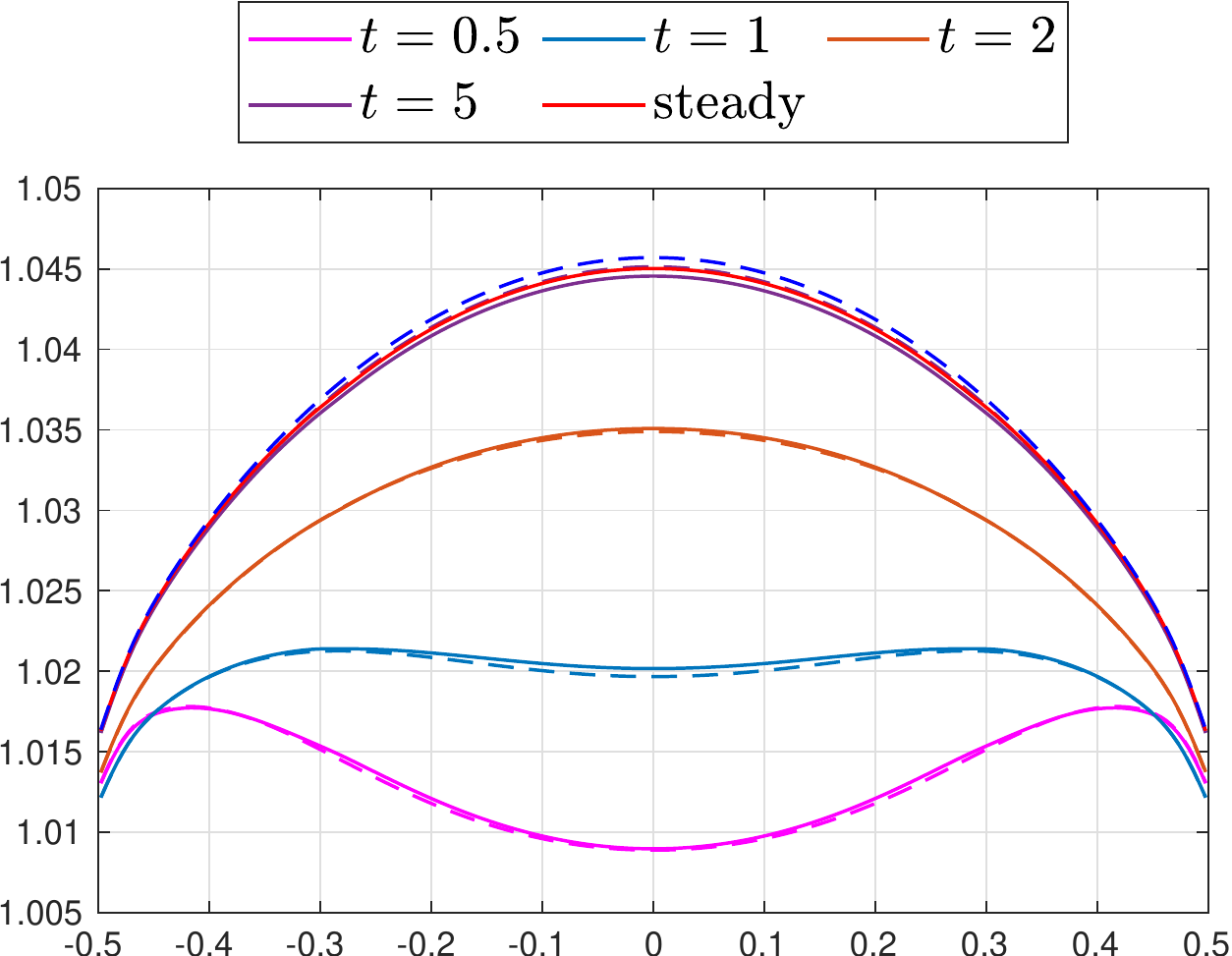}}\hfill
  \subfloat[$q_1$]{\includegraphics[width=0.49\textwidth,clip]{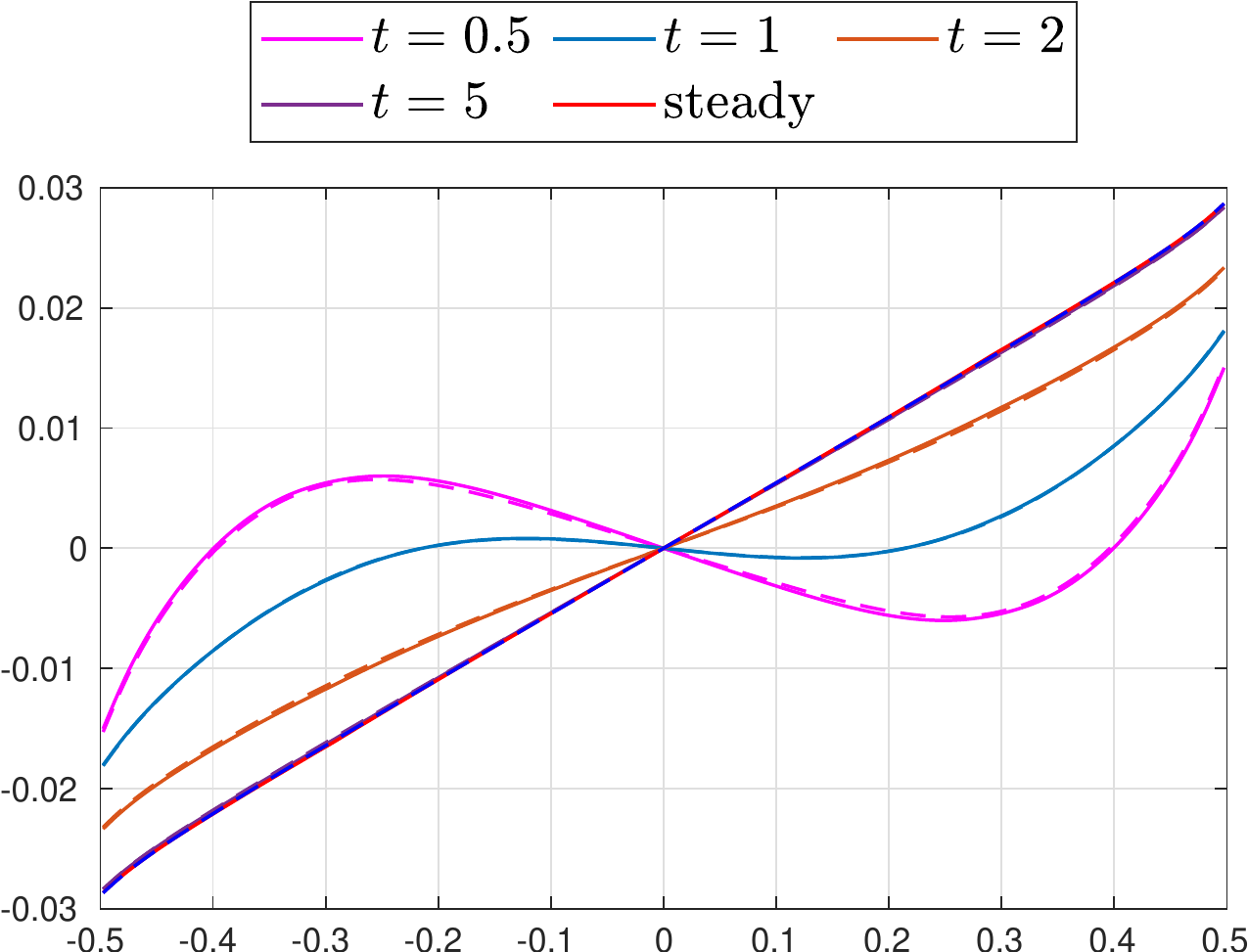}} 
 \caption{Solution of the Couette flow at different times. The solid lines are the numerical solution of the adaptive algorithm and the dashed lines are the reference solution. }
  \label{fig:ex2_sol}
\end{figure}

\begin{figure}[!htb]
% \colorbox{black}
  \centering
  \subfloat[$t = 0.1$ ]{\includegraphics[width=0.49\textwidth]{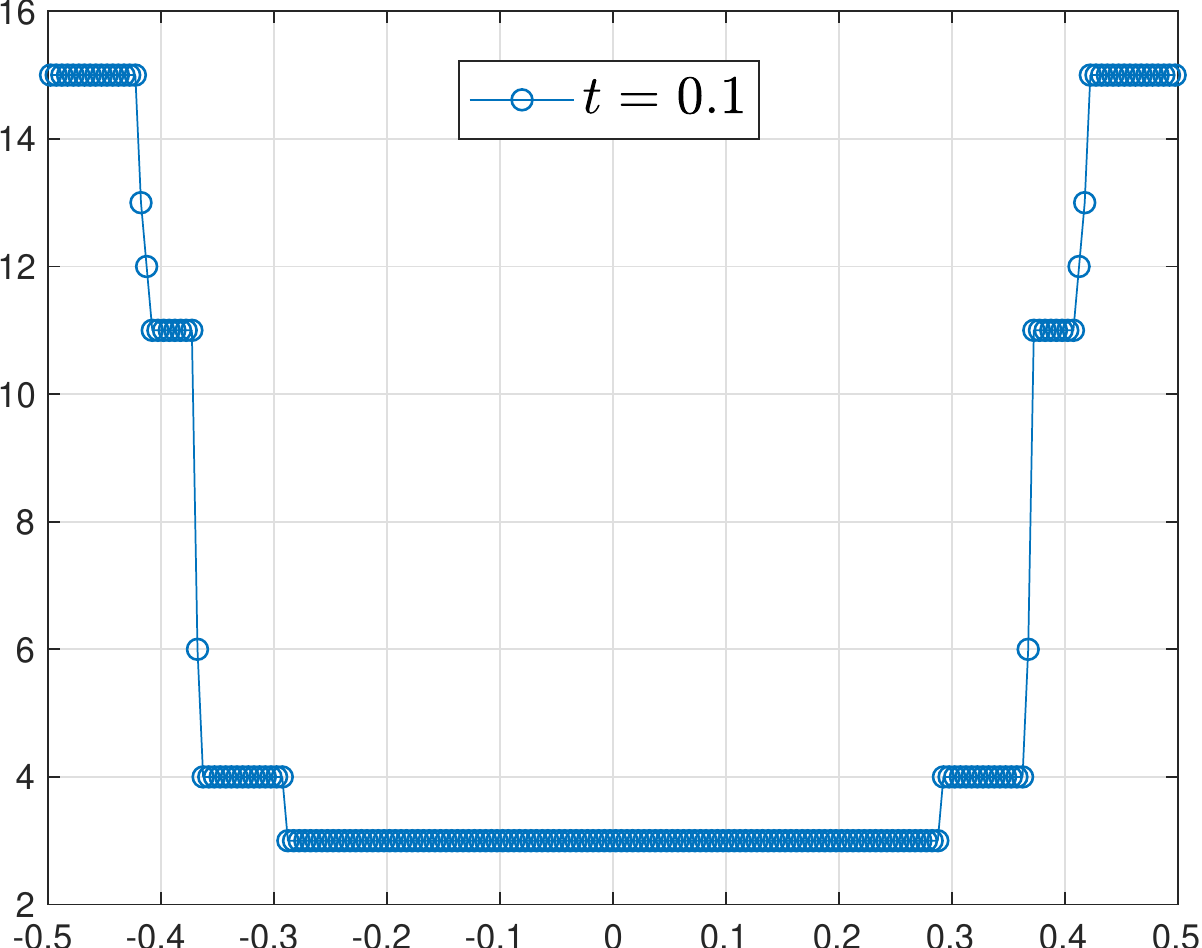}}\hfill
  \subfloat[$t = 0.5$]{\includegraphics[width=0.49\textwidth]{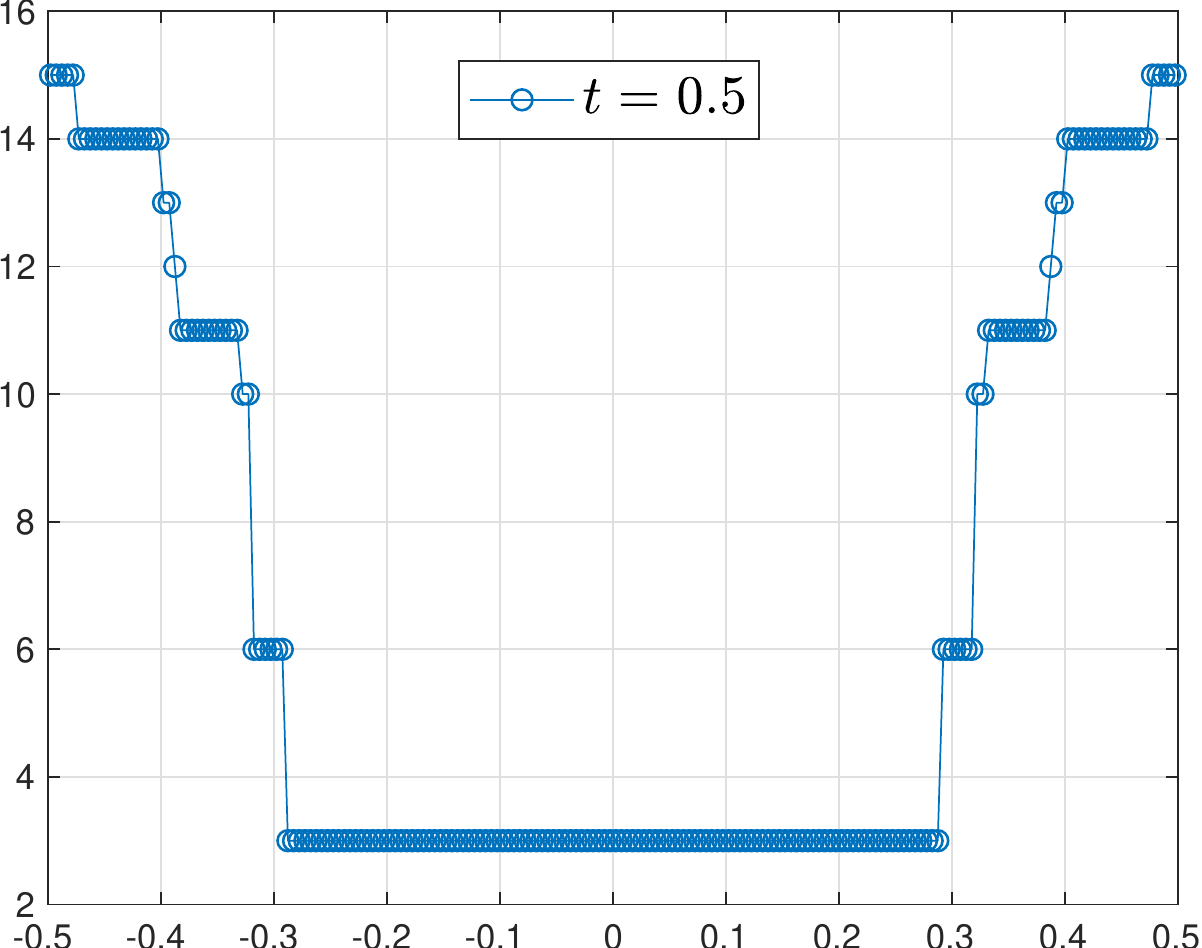}} \\
  \subfloat[$t = 1.0$]{\label{fig:ex2_t1}\includegraphics[width=0.49\textwidth]{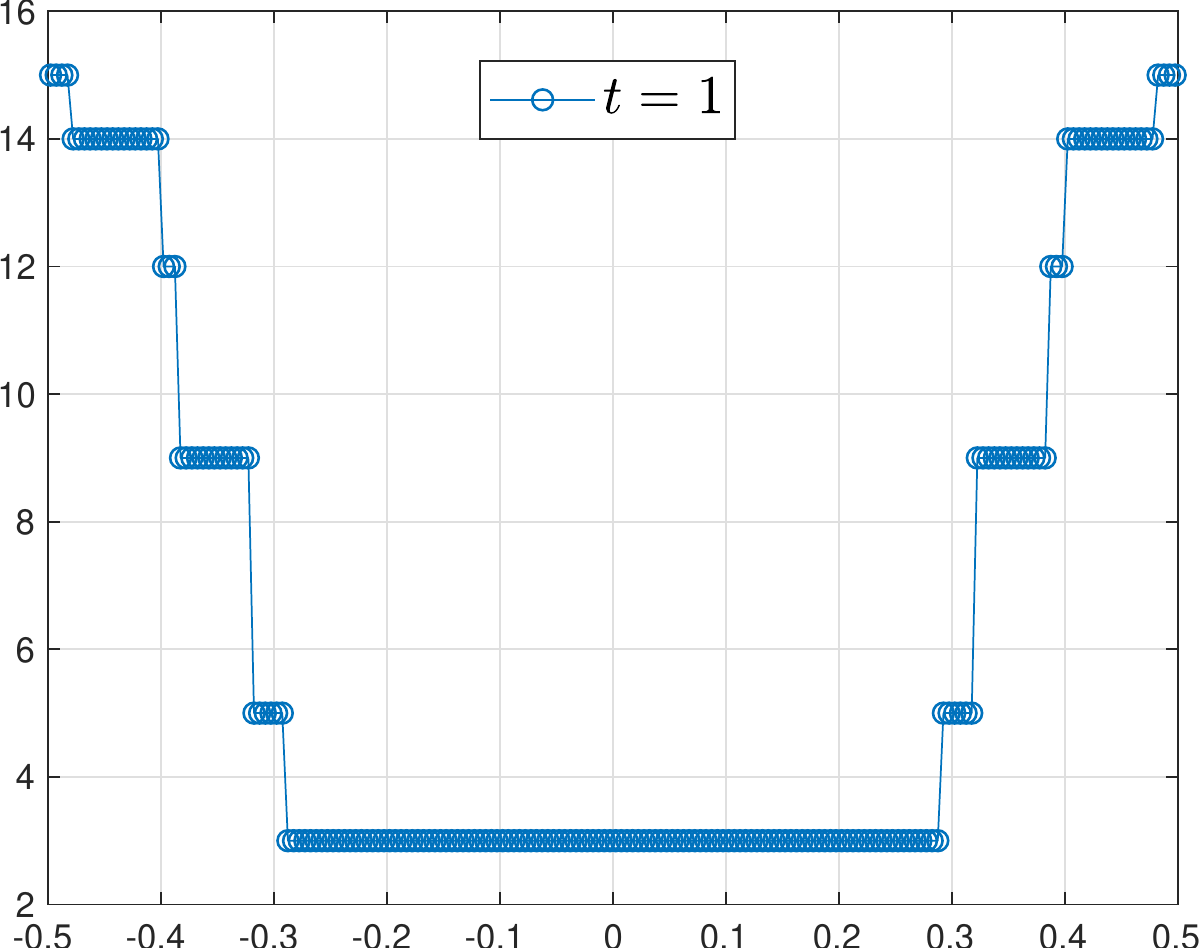}}\hfill
  \subfloat[steady state]{\label{fig:ex2_steady}\includegraphics[width=0.49\textwidth]{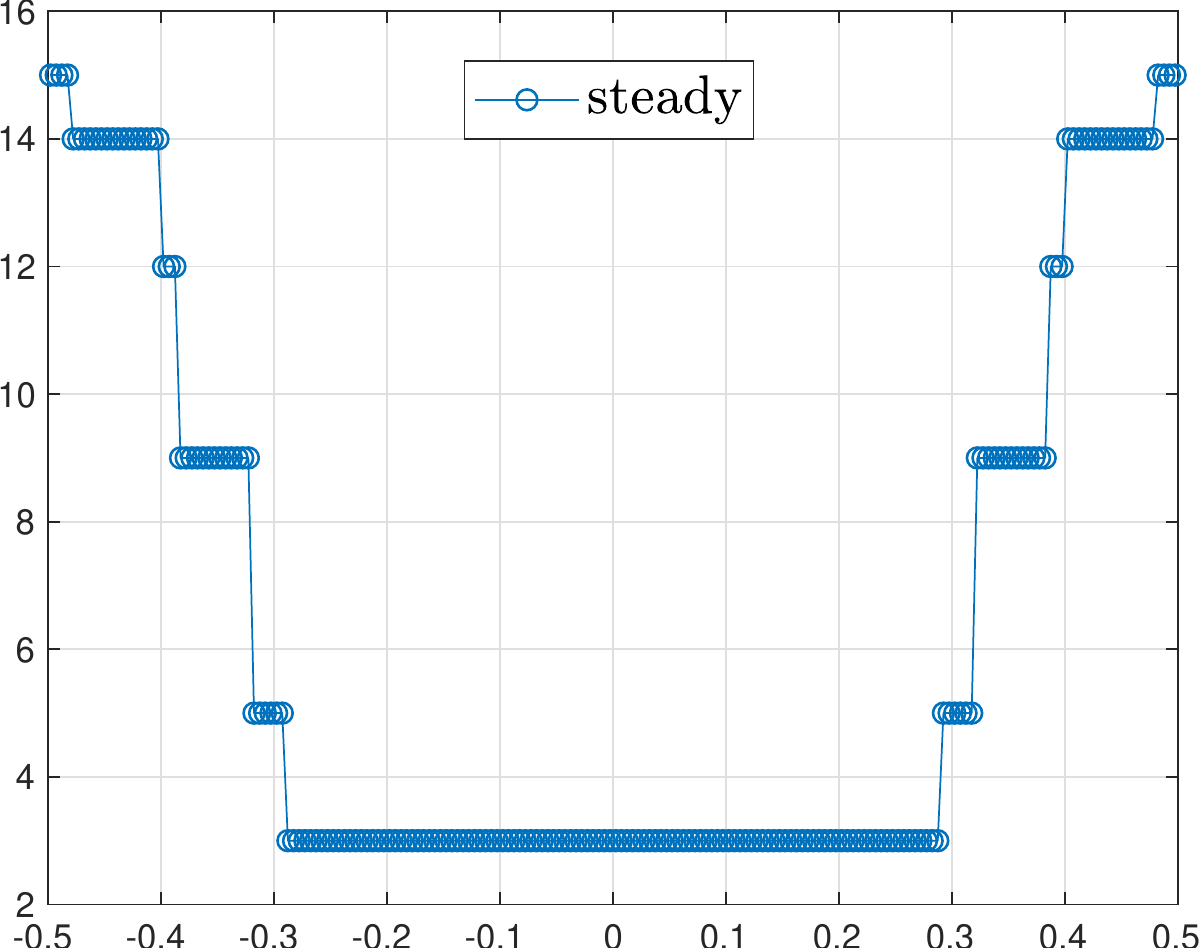}} 
 \caption{Distribution of $M_0$ for the Couette flow at different times.}
  \label{fig:ex2_M0}
\end{figure}

\begin{table}[!ht]
\centering
\caption{Statistical data for the planar Couette flow. $T_{\mathrm{ref}}$ and $T_{\mathrm{adp}}$ refer to the CPU time per time step for the reference solution and the self-adaptive solution, respectively, and $T_{\mathrm{ind}}$ refers to the CPU time per time step for the computation of the error indicator.}
\label{tab:ex2}
\begin{tabular}{c@{\qquad}c@{\qquad}c@{\qquad}c@{\qquad}c}
\hline
$T_{\mathrm{ref}}$ &  $T_{\mathrm{adp}}$ & $T_{\mathrm{ind}}$ & $1 - T_{\mathrm{adp}} / T_{\mathrm{ref}}$ & $T_{\mathrm{ind}} / T_{\mathrm{adp}}$ \\
\hline
$8.54$s & $1.51$s & $0.057$s & $82.3$\% & $3.78$\% \\
\hline
\end{tabular}
\end{table}

\subsection{Two-dimensional examples}
In our two-dimensional examples, we also consider the variable hard sphere model with $\nu = 5/9$, and $M_0$ is still capped at $15$. Two examples with and without boundary conditions will be considered in the following two subsections.

\subsubsection{Fluid diffusion} \label{sec:diffusion}
%Our second 2D example comes from \cite{Rana2015}, where the same initial condition is adopted as the cavity flow with all sides of the cavity are stationary. The temperature of the bottom side is $\theta = 2$ with all others set as $1$. 
Our first two-dimensional example considers the initial data
\begin{displaymath}
f(x_1, x_2, \bv, 0) = \frac{\rho(x_1, x_2)}{(2\pi \theta(x_1, x_2))^{3/2}} \exp \left( -\frac{|\bv -  \bu(x_1, x_2)|^2}{2\theta(x_1, x_2)} \right),
\end{displaymath}
where $\bu(x_1, x_2) = \boldsymbol{0}$ and $\theta(x_1, x_2) = 1$ for all $x_1$ and $x_2$, while $\rho(x)$ is set to be
\begin{displaymath}
\rho(x_1, x_2) = \left\{\begin{array}{@{}ll}
10, & \text{if } |x_1| \leqslant 0.05 \text{ and } |x_2| \leqslant 0.05, \\
1, & \text{otherwise.}
\end{array} \right.
\end{displaymath}
We set the computational domain to be
$\Omega = [-0.5, 0.5] \times [-0.5, 0.5]$ and apply the Neumann
boundary condition to simulate the flow in the unbounded domain. In
this example, there is a high density region in the center of the
domain, and we are interested in the dynamics of its diffusion into
the background fluid. Here we choose the Knudsen number to be
$\Kn = 0.05$. The parameters in \eqref{eq:truncated_series} are set to
be $M = 30$, $\bar{\bu} = \boldsymbol{0}$ and $\bar{\theta} = 1$. A
uniform grid of size $200 \times 200$ is utilized here to discretize
the physical space. The thresholds of the error indicator are chosen as
$(\epsilon_1, \epsilon_2) = (2.5, 8.5)$.

We plot the evolution of the fluid states in Figure \ref{fig:ex4_sol},
with reference solutions computed by using $M_0 = 15$
everywhere. Besides the equilibrium variables density $\rho$ and
temperature $\theta$, we have also plotted the shear stress
$\sigma_{12}$, which is related to the distribution function by
\begin{equation}
  \sigma_{12} = \int_{\mathbb{R}^3} (v_1 - u_1)(v_2 - u_2) f(\bv) \,\mathrm{d}\bv.
\end{equation}
It can be seen that the density in the center of the domain gradually
decreases, and as the mass flows out, the temperature also starts to
decrease so that the total energy can be conserved. Due to the
symmetry of the initial data, the value of the non-equilibrium
variable $\sigma_{12}$ equals zero on both $x$- and $y$-axes. As the
fluid evolves with time, the non-equilibrium effect spreads out, while
the peak values of $\sigma_{12}$ start to decrease. With our adaptive
method, these phenomena can be accurately captured.
To get a clearer view of the difference between the adaptive solutions and the reference solutions, we define 
\begin{equation}
    \label{eq:ex3_error}
    \mathcal{E}_{\rho} = \rho^{\rm adp} - \rho^{\rm ref}, \qquad
    \mathcal{E}_{\theta} = \theta^{\rm adp} - \theta^{\rm ref}, \qquad
    \mathcal{E}_{\sigma_{12}} = \sigma_{12}^{\rm adp} - \sigma_{12}^{\rm ref},
\end{equation}
where the superscripts ``adp'' and ``ref'' denote the adaptive solution and the reference solution, respectively. These quantities are plotted in Figure \ref{fig:ex4_l2_error}, and the corresponding relative $L^2$ differences are given in Table \ref{tab:ex4_l2_error}. One can observe that the difference between the two solutions increases with time due to the accumulation of the error. This also implies that the thresholds $\epsilon_1$ and $\epsilon_2$, whose values stay the same throughout the simulation, do not directly correspond to the error of the solution. Our error indicator only estimates the local truncation error, which may accumulate in time-dependent problems. In such circumstances, to ensure the numerical accuracy for longer simulations, one may need to choose smaller values of $\epsilon_1$ and $\epsilon_2$. Such an effect is automatically incorporated into the procedure of parameter selection introduced in Section \ref{sec:adaptive} if the test runs are also preformed until the desired final time.

\begin{table}[!ht]
\centering
\def\arraystretch{1.3}
\caption{Relative $L_2$ difference between the self-adaptive solution and the reference solution.}
\label{tab:ex4_l2_error}
\footnotesize
\begin{tabular}{c|@{\qquad}c@{\qquad}c@{\qquad}c@{\qquad}c}
 &  $t = 0.04$ & $t = 0.08$ & $t = 0.12$ & $t = 0.15$ \\
\hline
$\|\mathcal{E}_{\rho}\|_{L^2} / \|\rho^{\mathrm{ref}}\|_{L^2}$ & $0.01$\% & $0.03$\% & $0.11$\% & $0.24$\% \\
$\|\mathcal{E}_{\theta}\|_{L^2} / \|\theta^{\mathrm{ref}}\|_{L^2}$ & $0.01$\% & $0.04$\% & $0.12$\% & $0.20$\% \\
$\|\mathcal{E}_{\sigma_{12}}\|_{L^2} / \|\sigma_{12}^{\mathrm{ref}}\|_{L^2}$ & $0.30$\% & $0.84$\% & $1.29$\% & $2.01$\% 
\end{tabular}
\end{table}

\begin{figure}[!htb]
% \colorbox{black}
  \centering
  \subfloat[$\rho, t = 0.04$]{\includegraphics[bb=18 21 584 500, width=0.33\textwidth,clip]{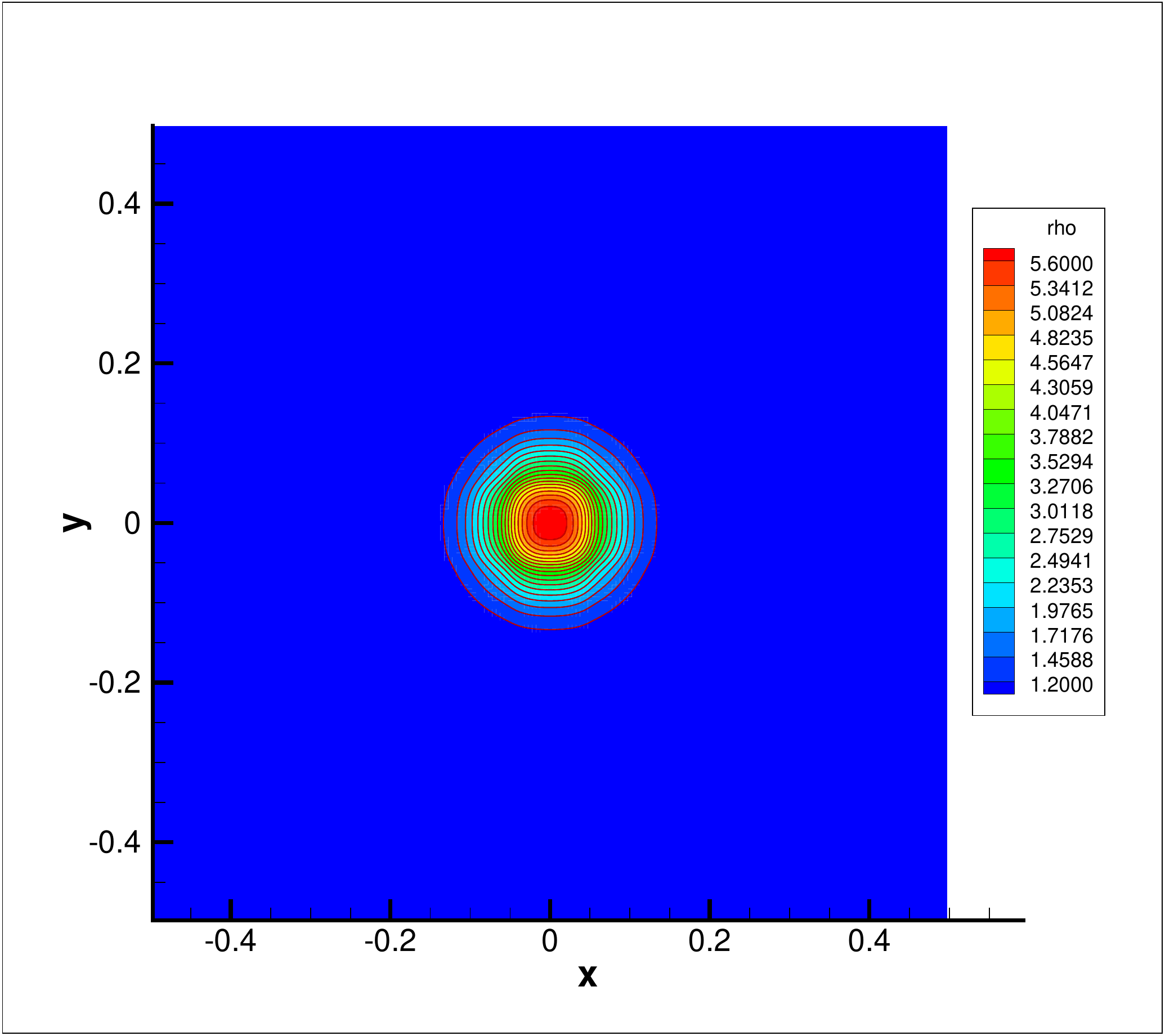}}\hfill
  \subfloat[$\theta, t = 0.04$]{\includegraphics[bb=18 21 584 500, width=0.33\textwidth,clip]{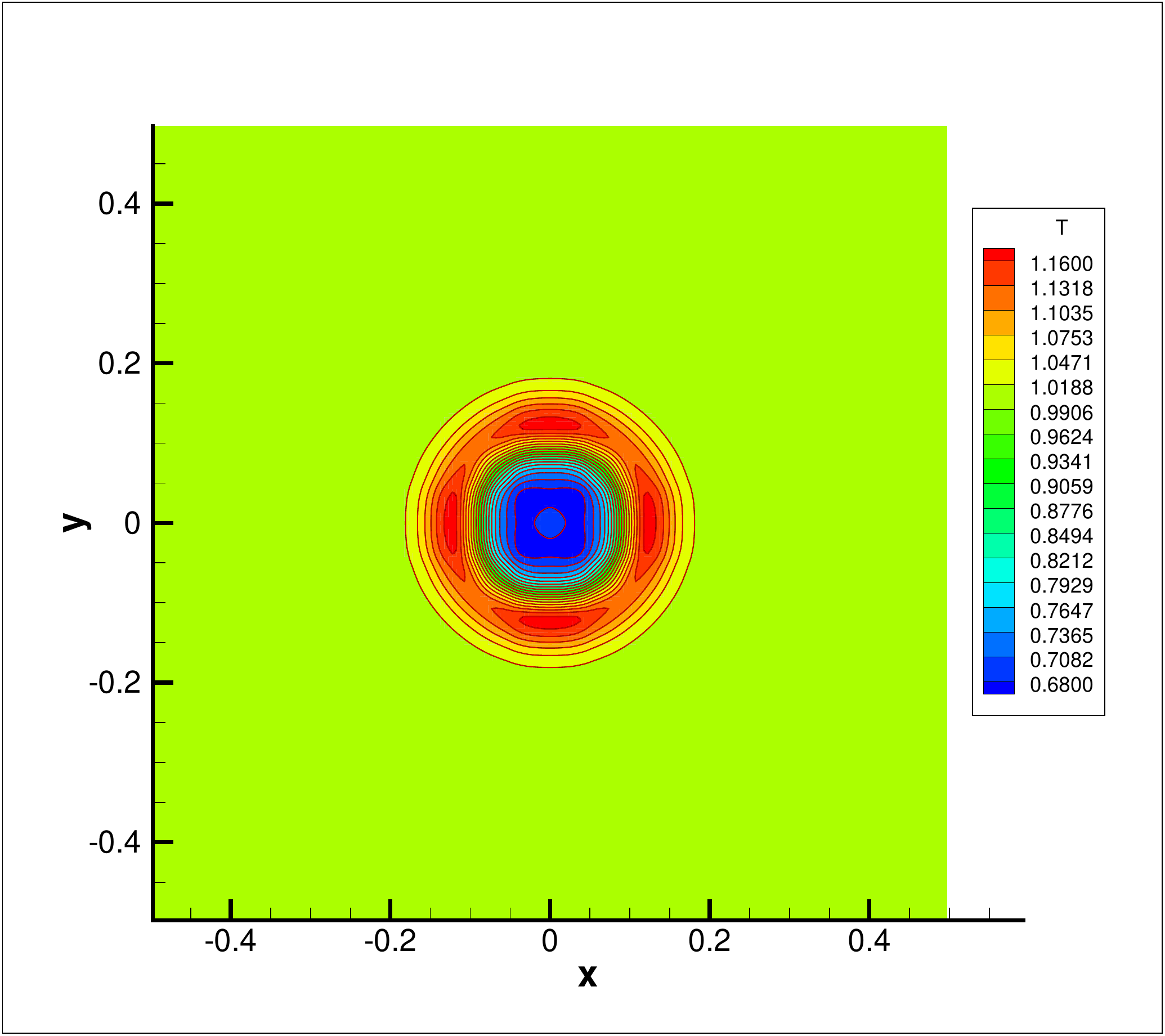}}\hfill
  \subfloat[$\sigma_{12}, t = 0.04$ ]{\includegraphics[bb=18 21 584 500, width=0.33\textwidth,clip]{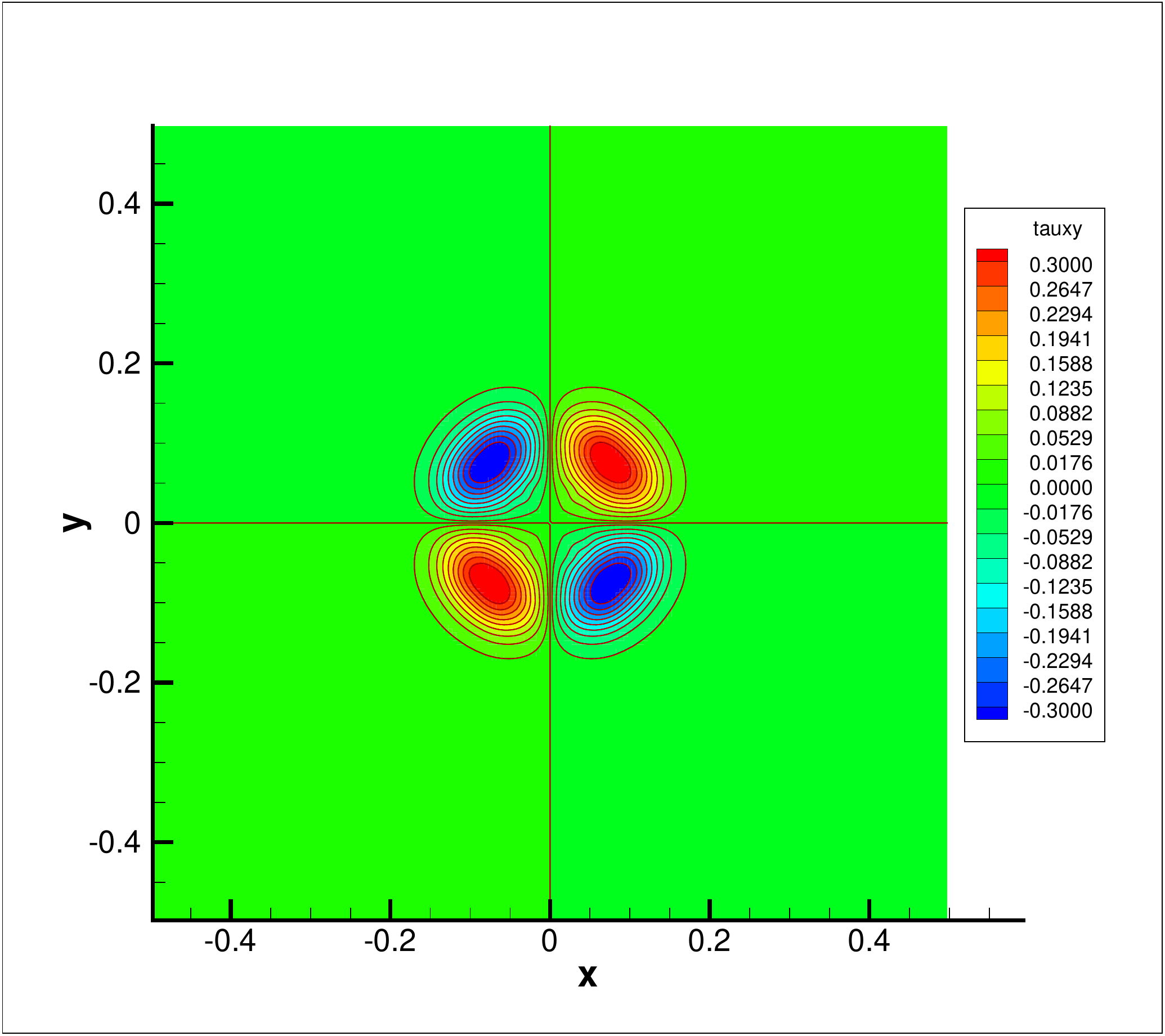}}\hfill \\
  \subfloat[$\rho, t = 0.08$]{\includegraphics[bb=18 21 584 500, width=0.33\textwidth,clip]{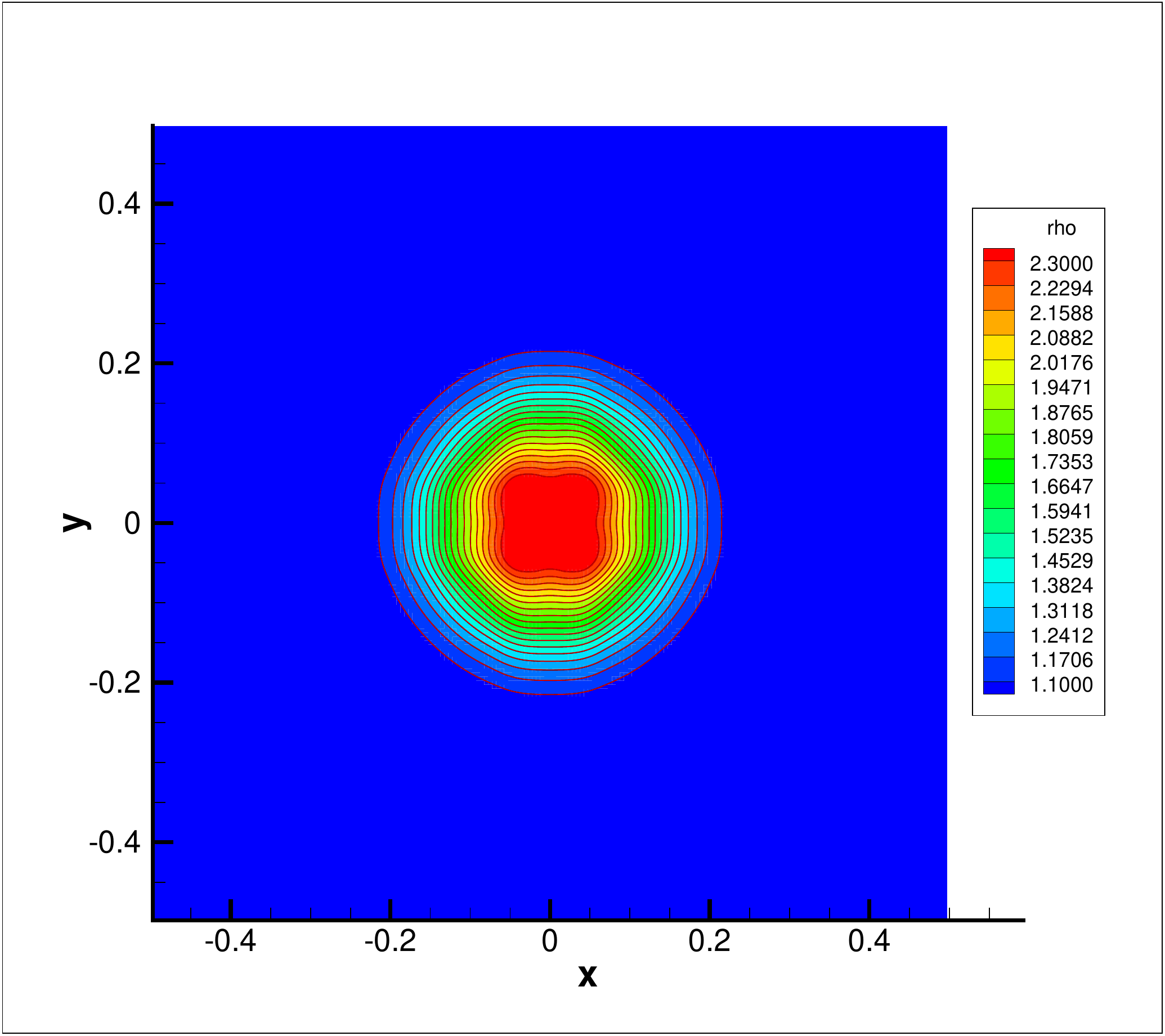}} \hfill
  \subfloat[$\theta, t = 0.08$]{\includegraphics[bb=18 21 584 500, width=0.33\textwidth,clip]{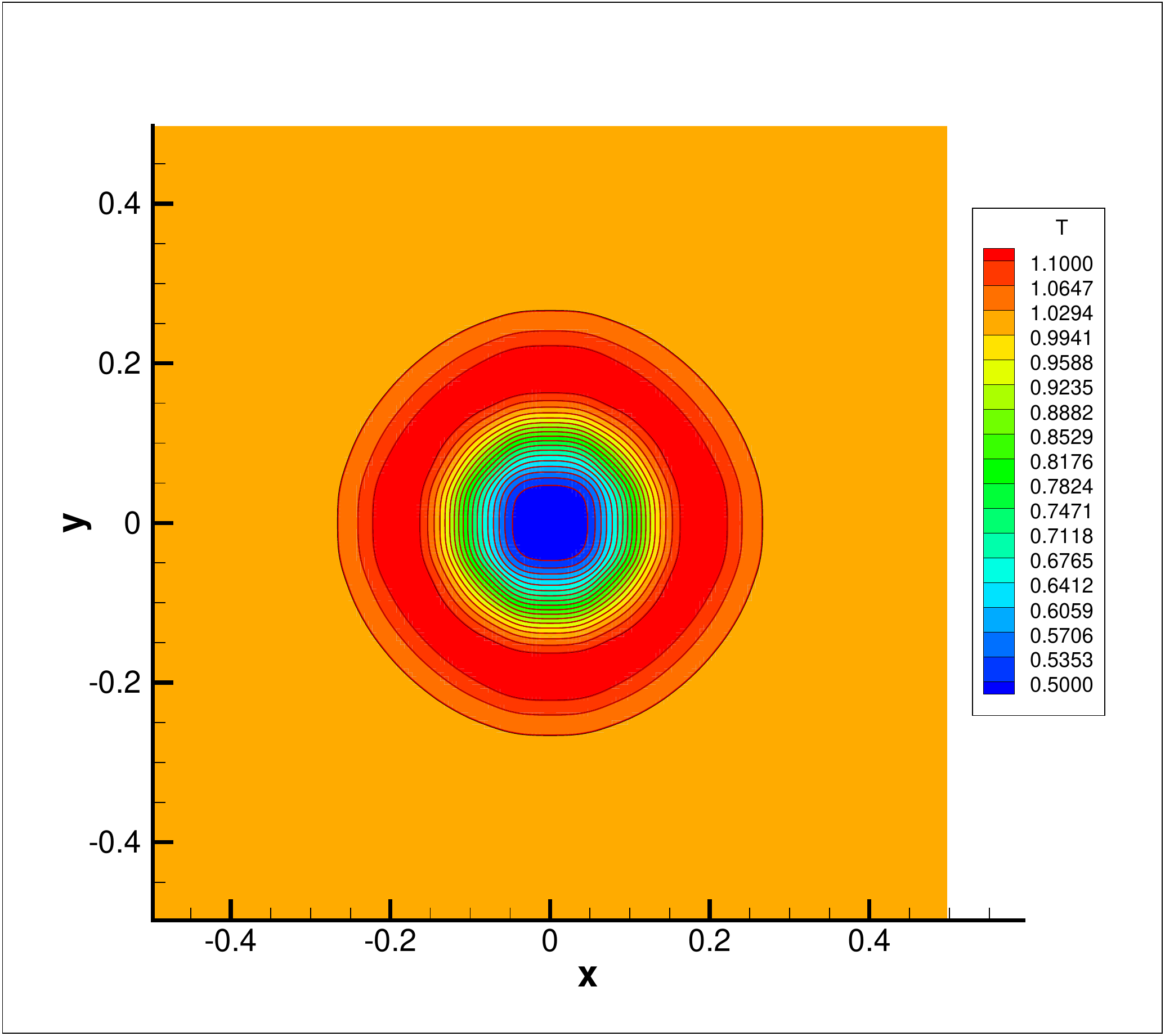}} \hfill
  \subfloat[$\sigma_{12}, t = 0.08$]{\includegraphics[bb=18 21 584 500, width=0.33\textwidth,clip]{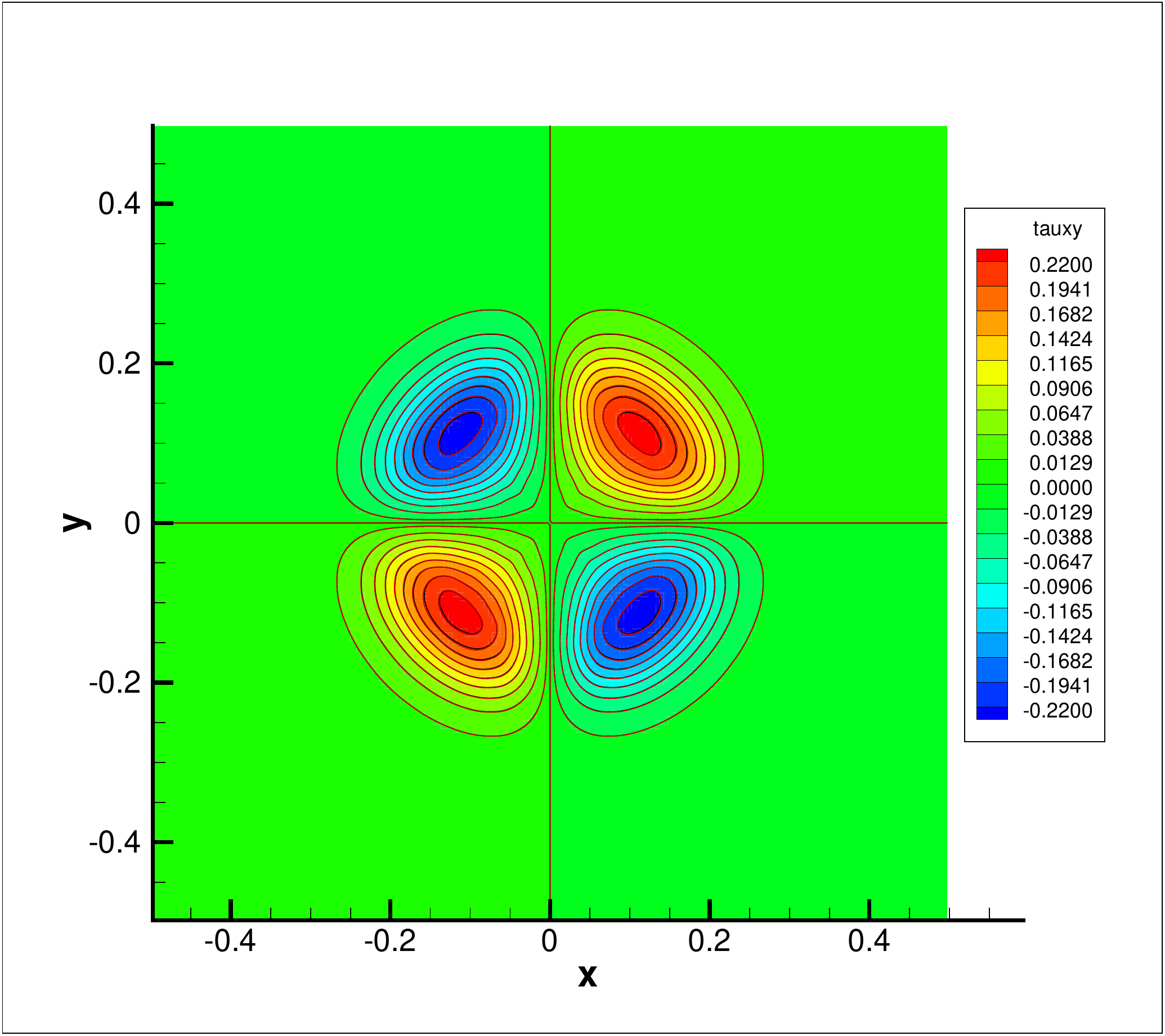}} \hfill \\
  \subfloat[$\rho, t = 0.12$]{\includegraphics[bb=18 21 584 500,width=0.33\textwidth,clip]{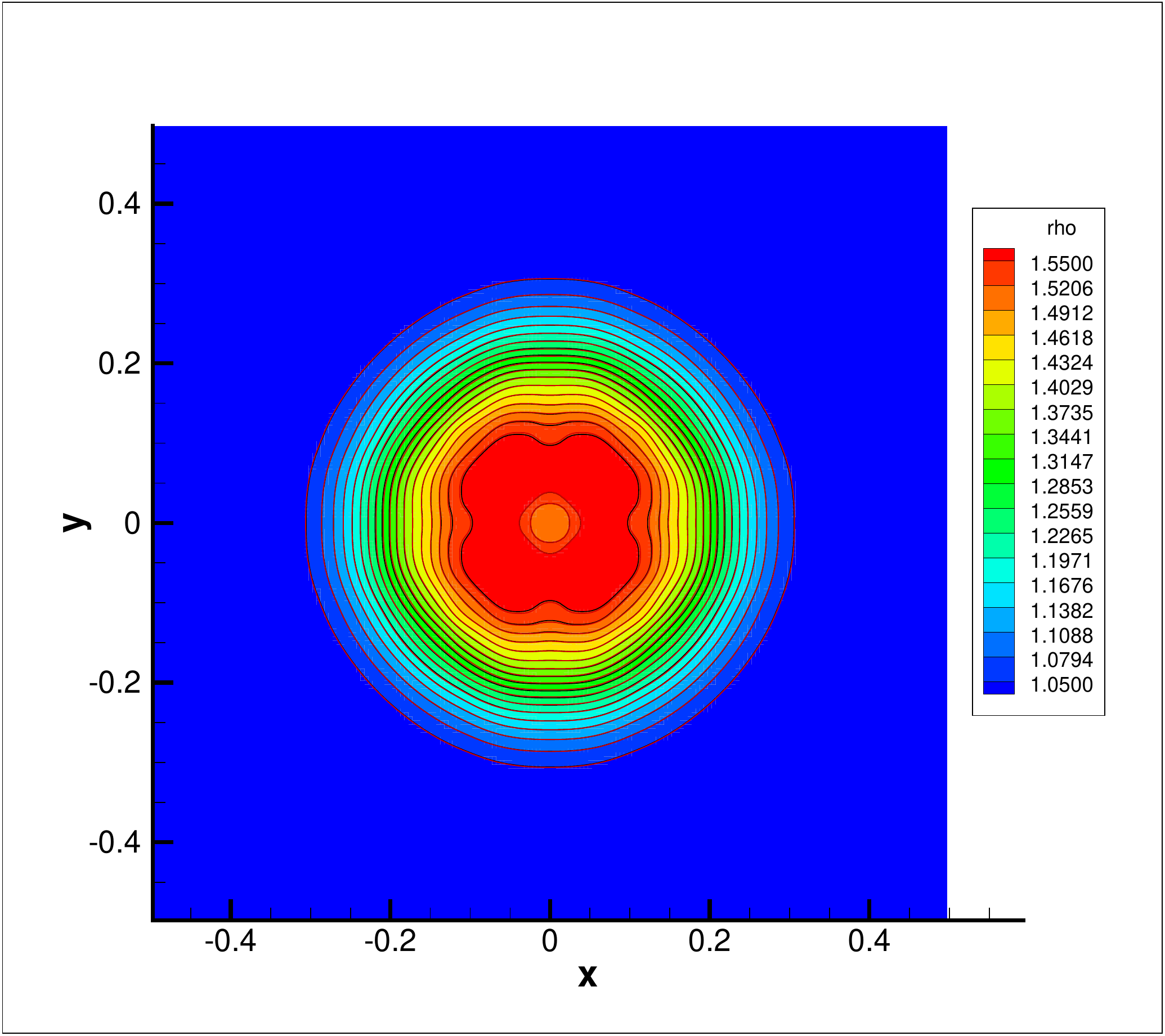}} \hfill
  \subfloat[$\theta, t = 0.12$]{\includegraphics[bb=18 21 584 500, width=0.33\textwidth,clip]{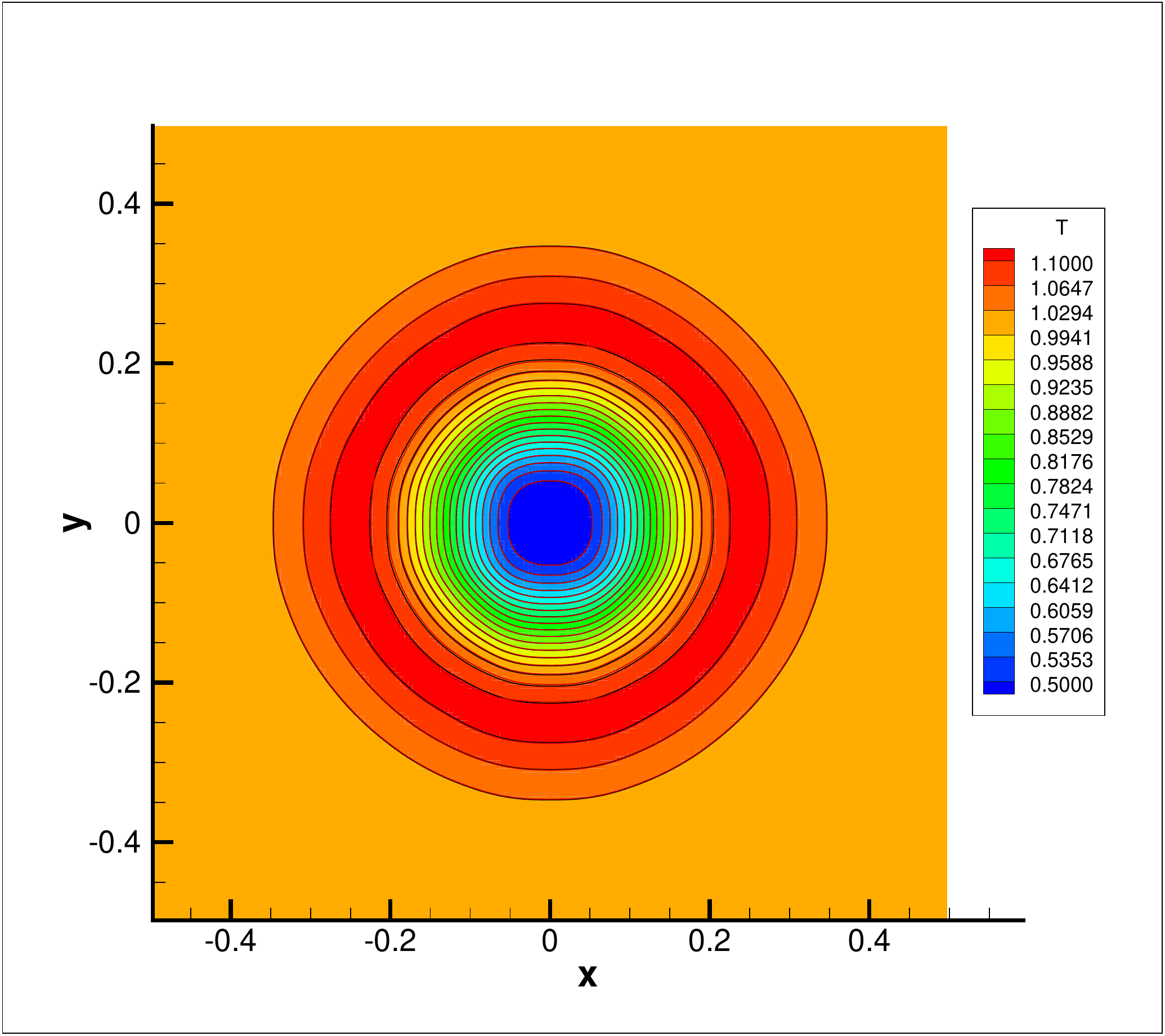}} \hfill
  \subfloat[$\sigma_{12}, t = 0.12$]{\includegraphics[bb=18 21 584 500, width=0.33\textwidth,clip]{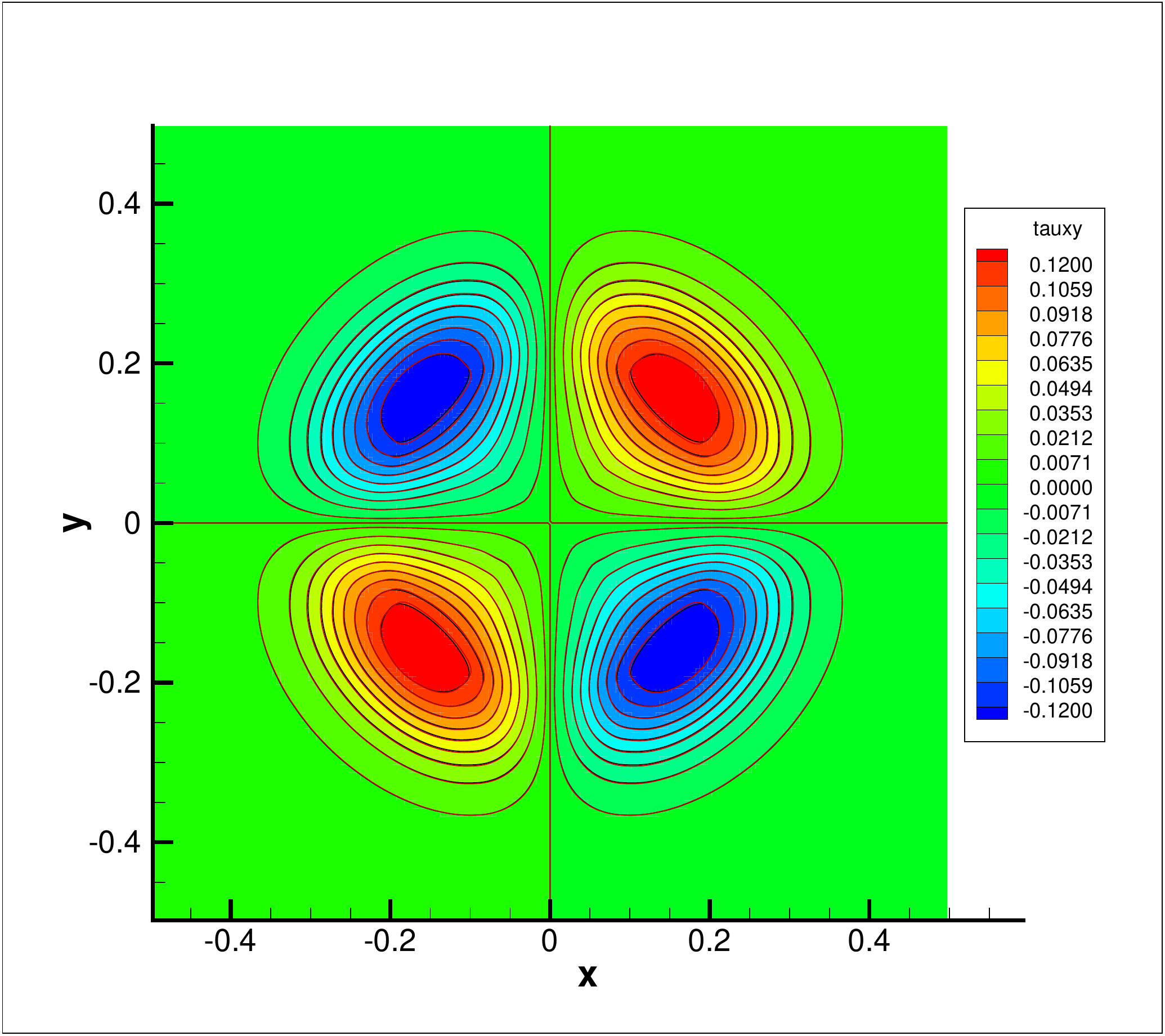}} \\
  \subfloat[$\rho, t = 0.15$]{\includegraphics[bb=18 21 584 500, width=0.33\textwidth,clip]{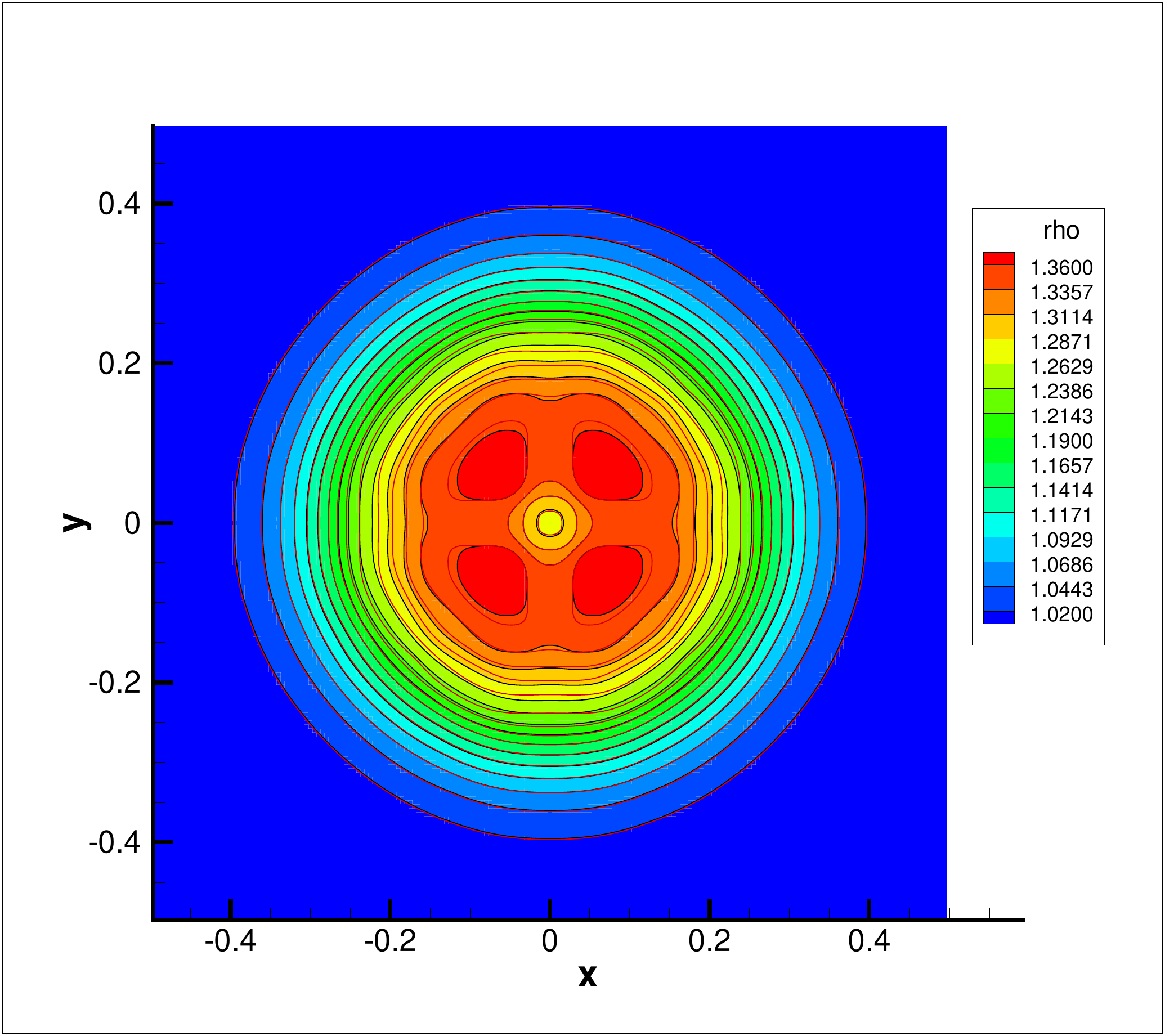}}\hfill
  \subfloat[$\theta, t = 0.15$]{\includegraphics[bb=18 21 584 500, width=0.33\textwidth,clip]{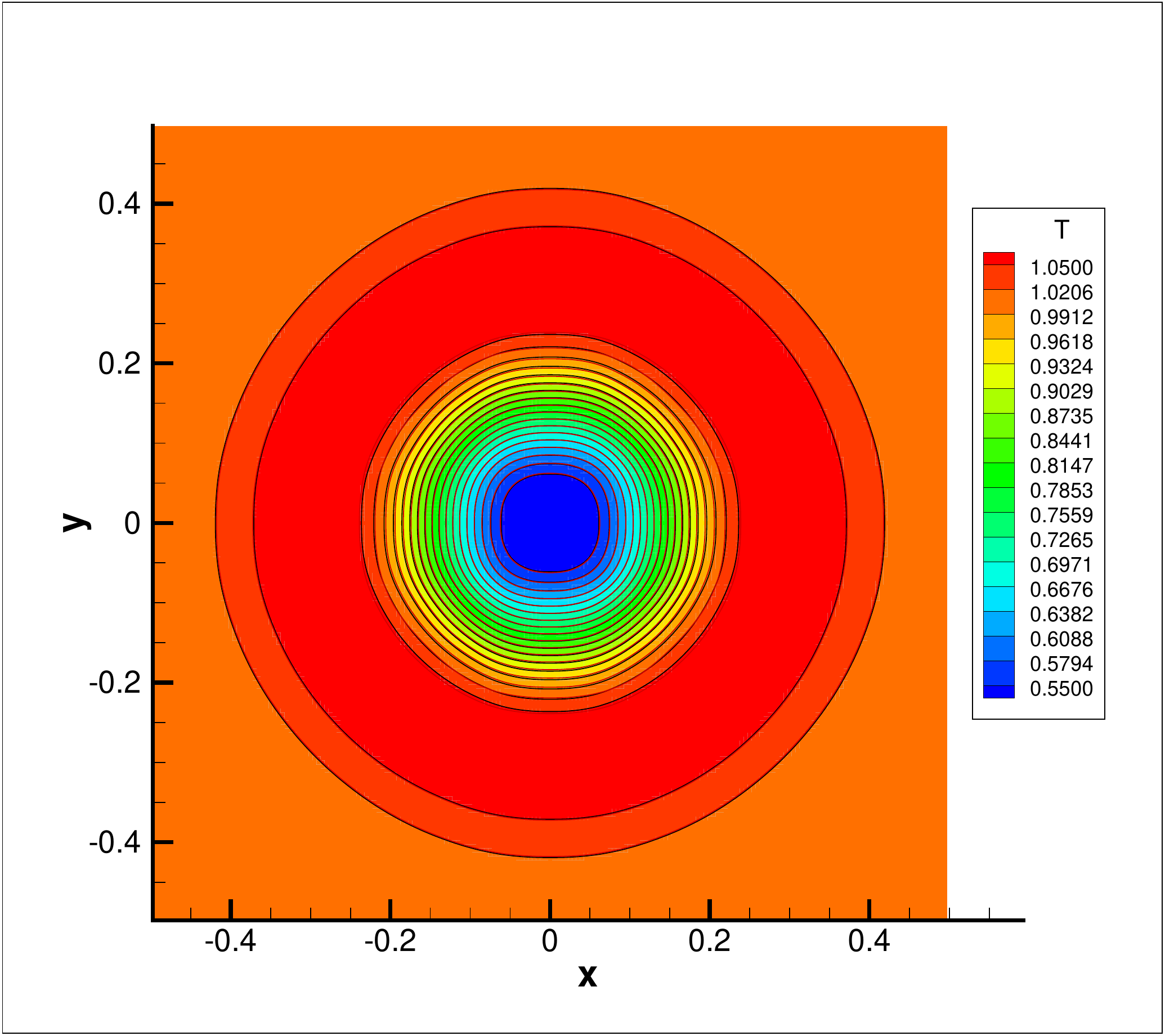}} \hfill
  \subfloat[$\sigma_{12}, t=0.15$]{\includegraphics[bb=18 21 584 500, width=0.33\textwidth,clip]{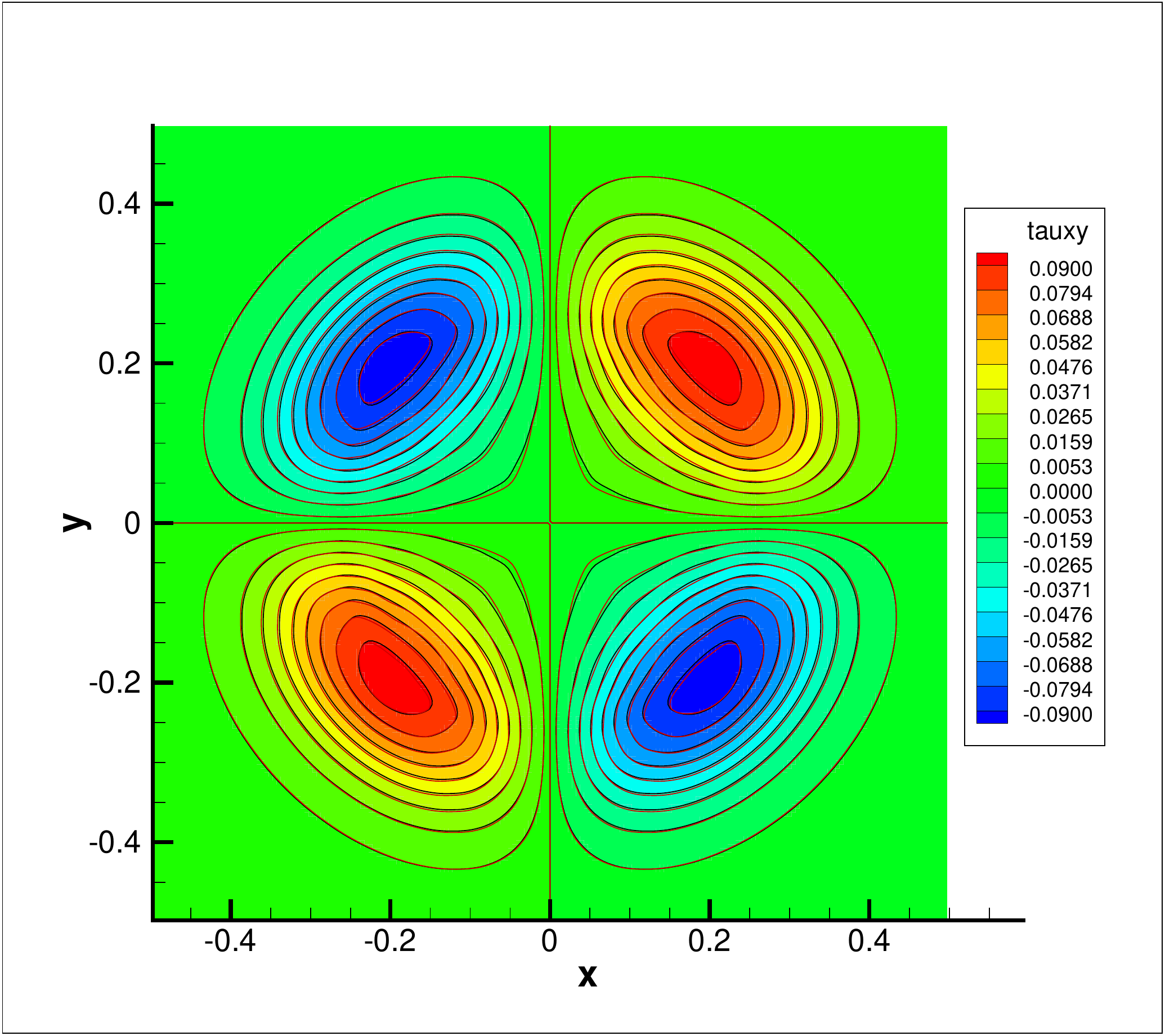}}
  \caption{Solution of the fluid diffusion problem at different
    times. The red contours are the numerical solutions of the adaptive method and the black
    contours are the reference solutions.}
  \label{fig:ex4_sol}
\end{figure}

\begin{figure}[!htb]
% \colorbox{black}
  \centering
  \subfloat[$\mathcal{E}_{\rho}, t = 0.04$]{\includegraphics[bb=18 21 584 500, width=0.33\textwidth,clip]{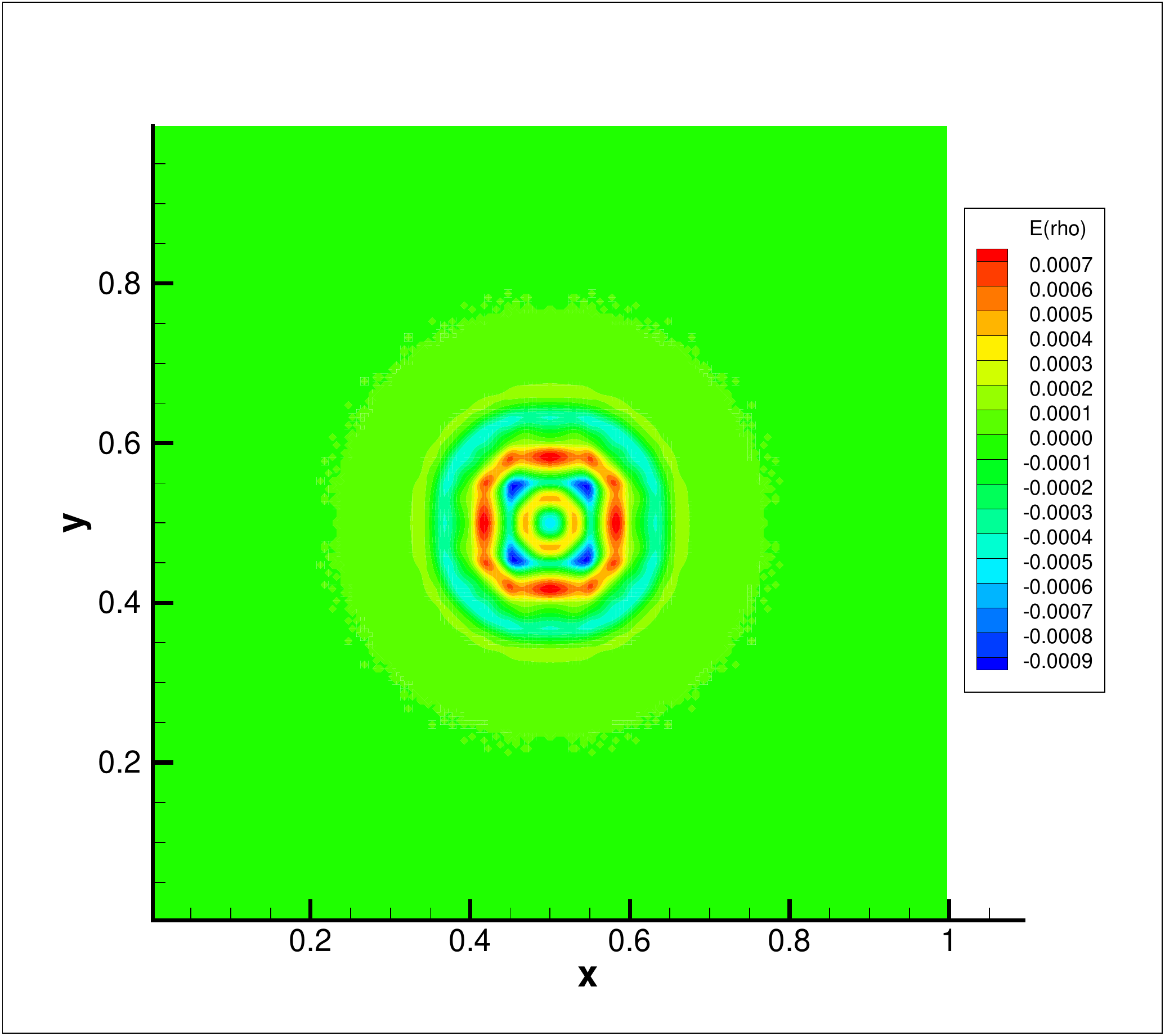}}\hfill
  \subfloat[$\mathcal{E}_{\theta}, t = 0.04$]{\includegraphics[bb=18 21 584 500, width=0.33\textwidth,clip]{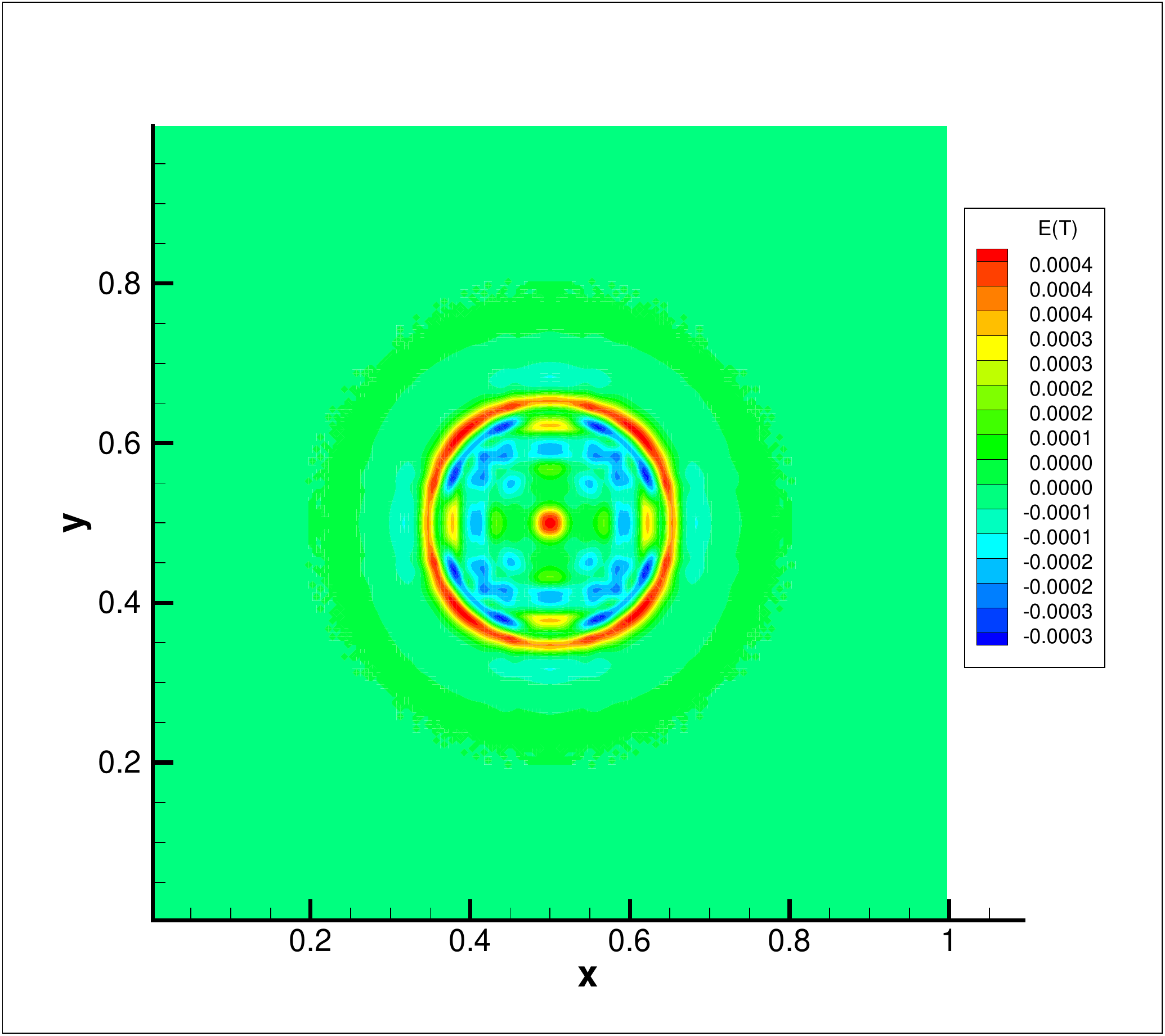}}\hfill
  \subfloat[$\mathcal{E}_{\sigma_{12}}, t = 0.04$ ]{\includegraphics[bb=18 21 584 500, width=0.33\textwidth,clip]{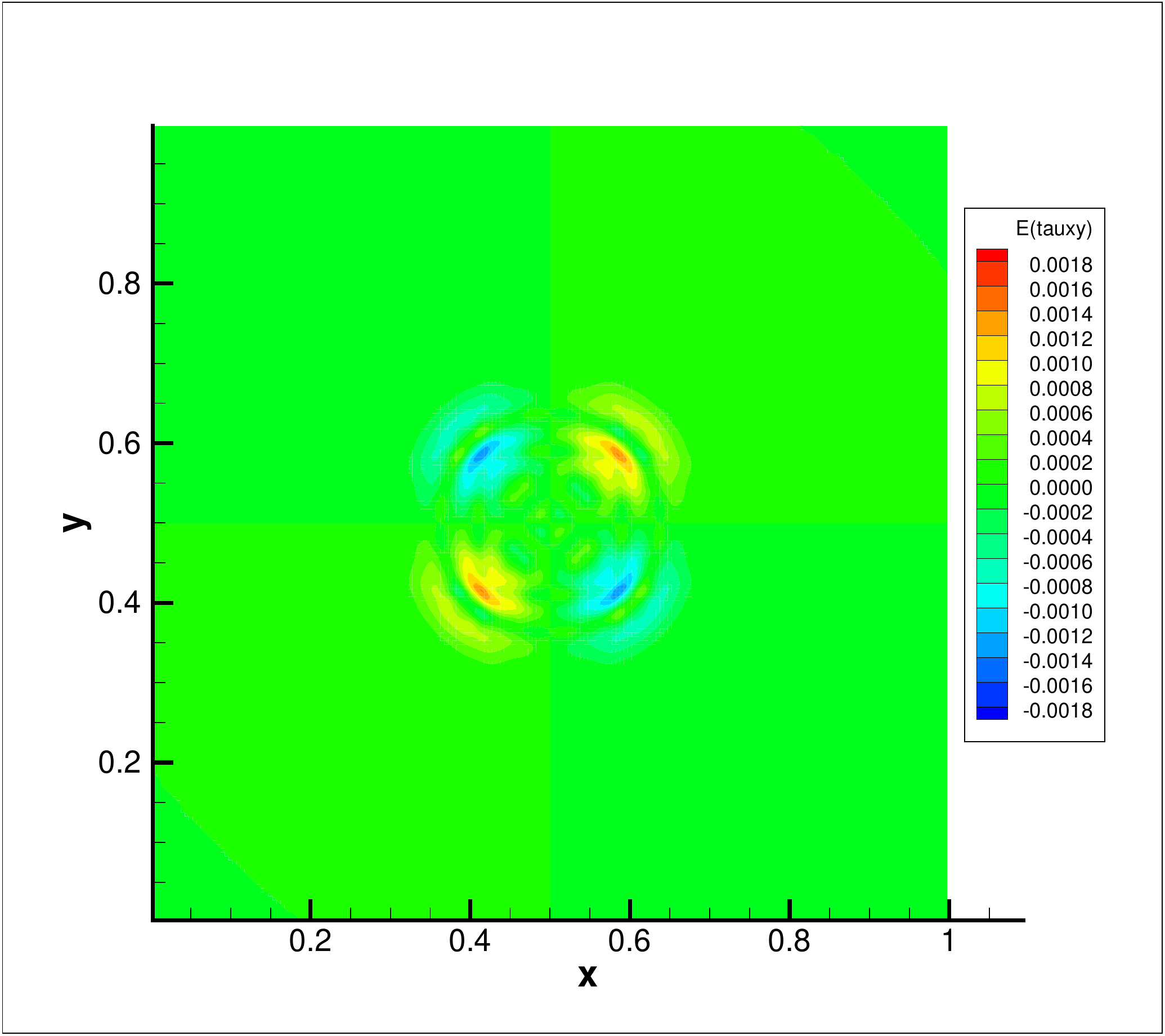}}\hfill \\
  \subfloat[$\mathcal{E}_{\rho}, t = 0.08$]{\includegraphics[bb=18 21 584 500, width=0.33\textwidth,clip]{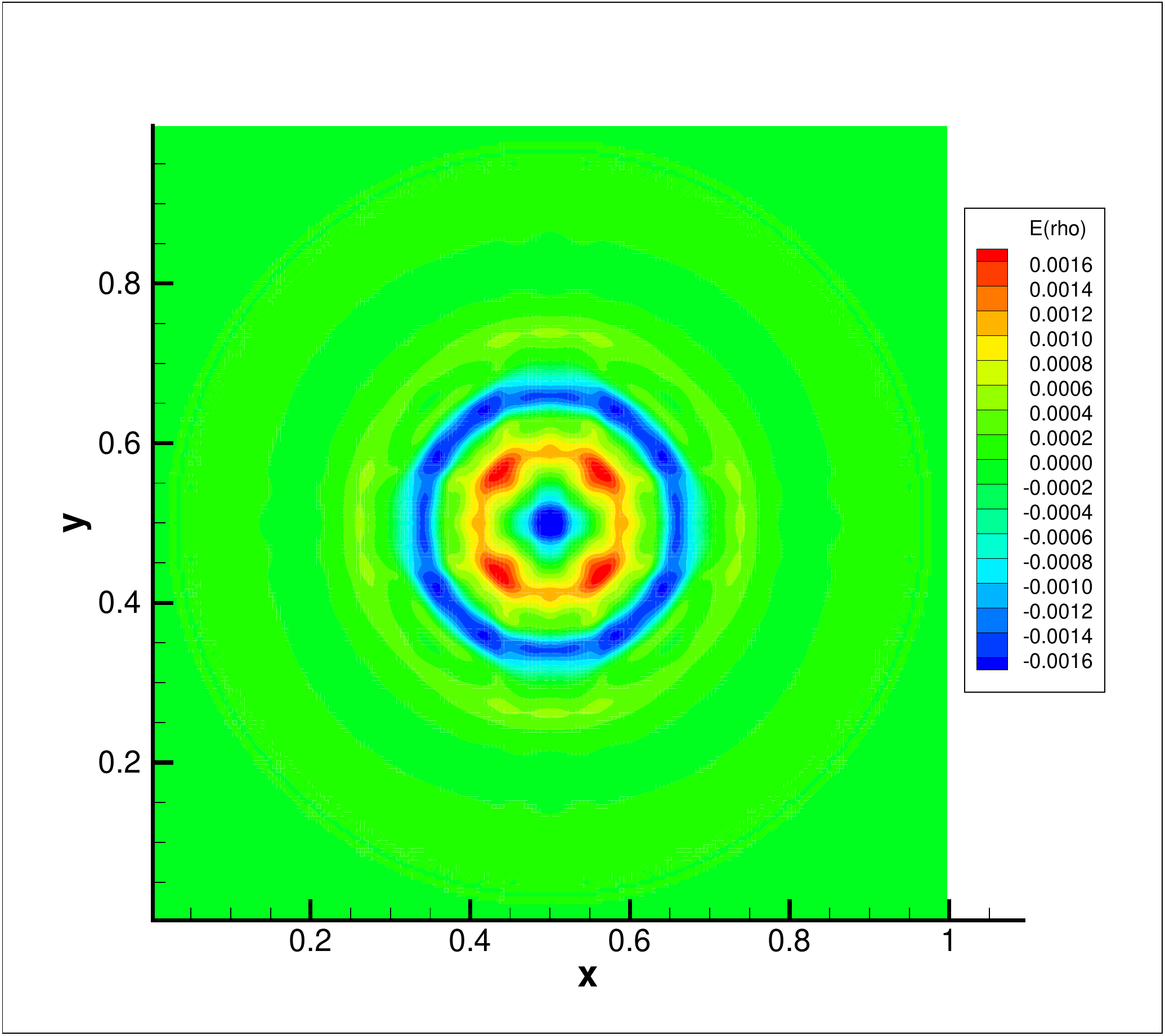}} \hfill
  \subfloat[$\mathcal{E}_{\theta}, t = 0.08$]{\includegraphics[bb=18 21 584 500, width=0.33\textwidth,clip]{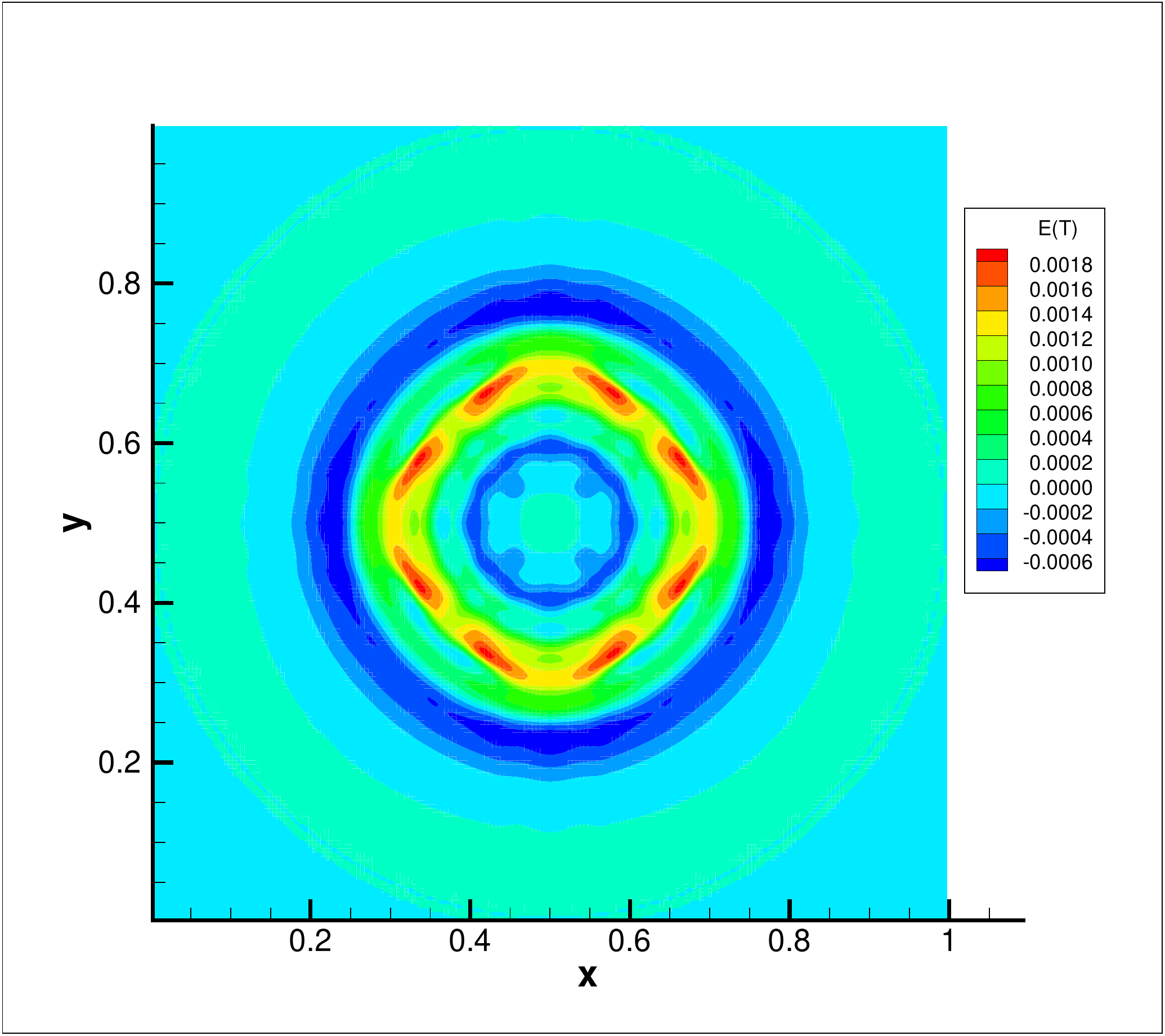}} \hfill
  \subfloat[$\mathcal{E}_{\sigma_{12}}, t = 0.08$]{\includegraphics[bb=18 21 584 500, width=0.33\textwidth,clip]{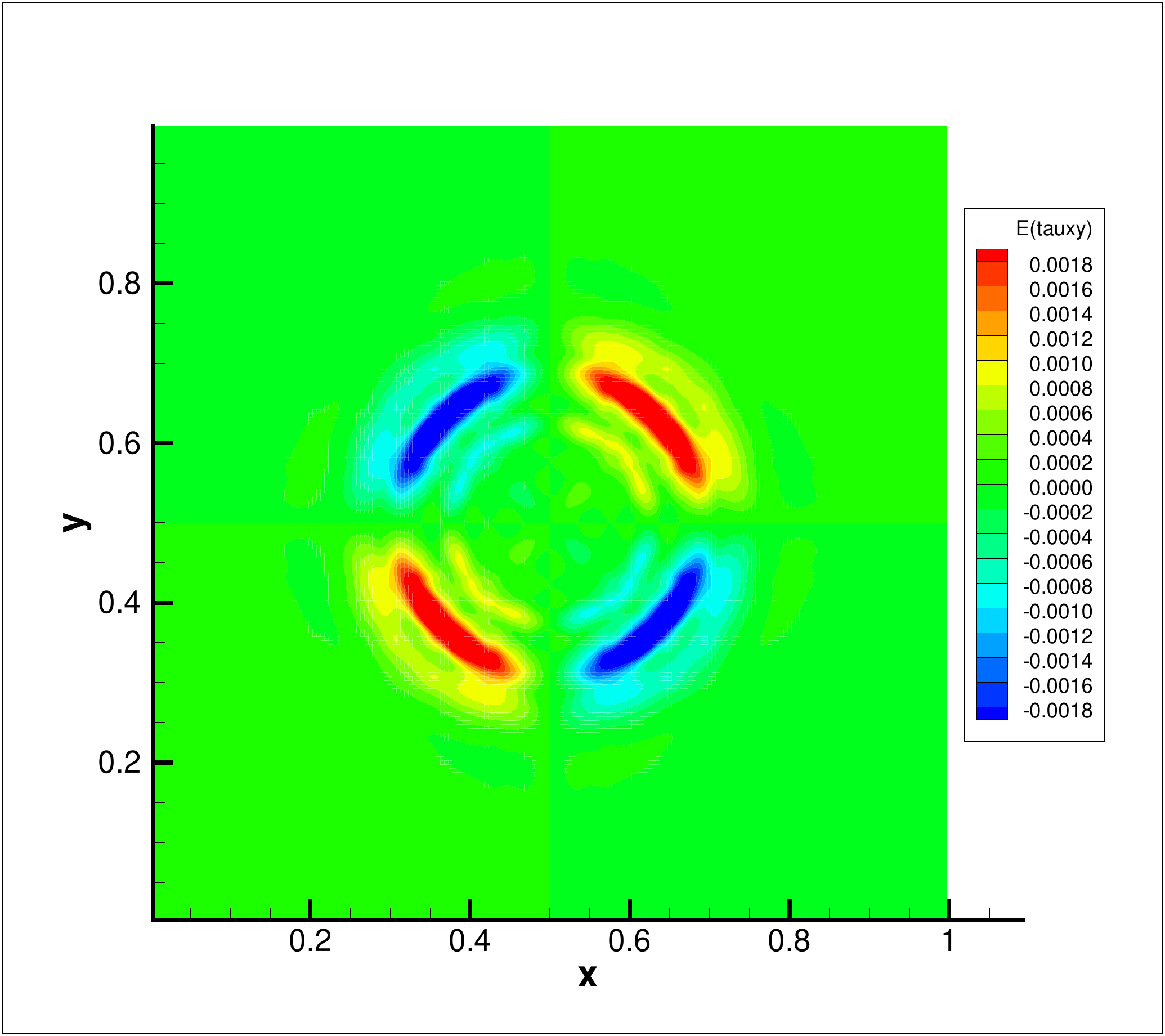}} \hfill \\
  \subfloat[$\mathcal{E}_{\rho}, t = 0.12$]{\includegraphics[bb=18 21 584 500,width=0.33\textwidth,clip]{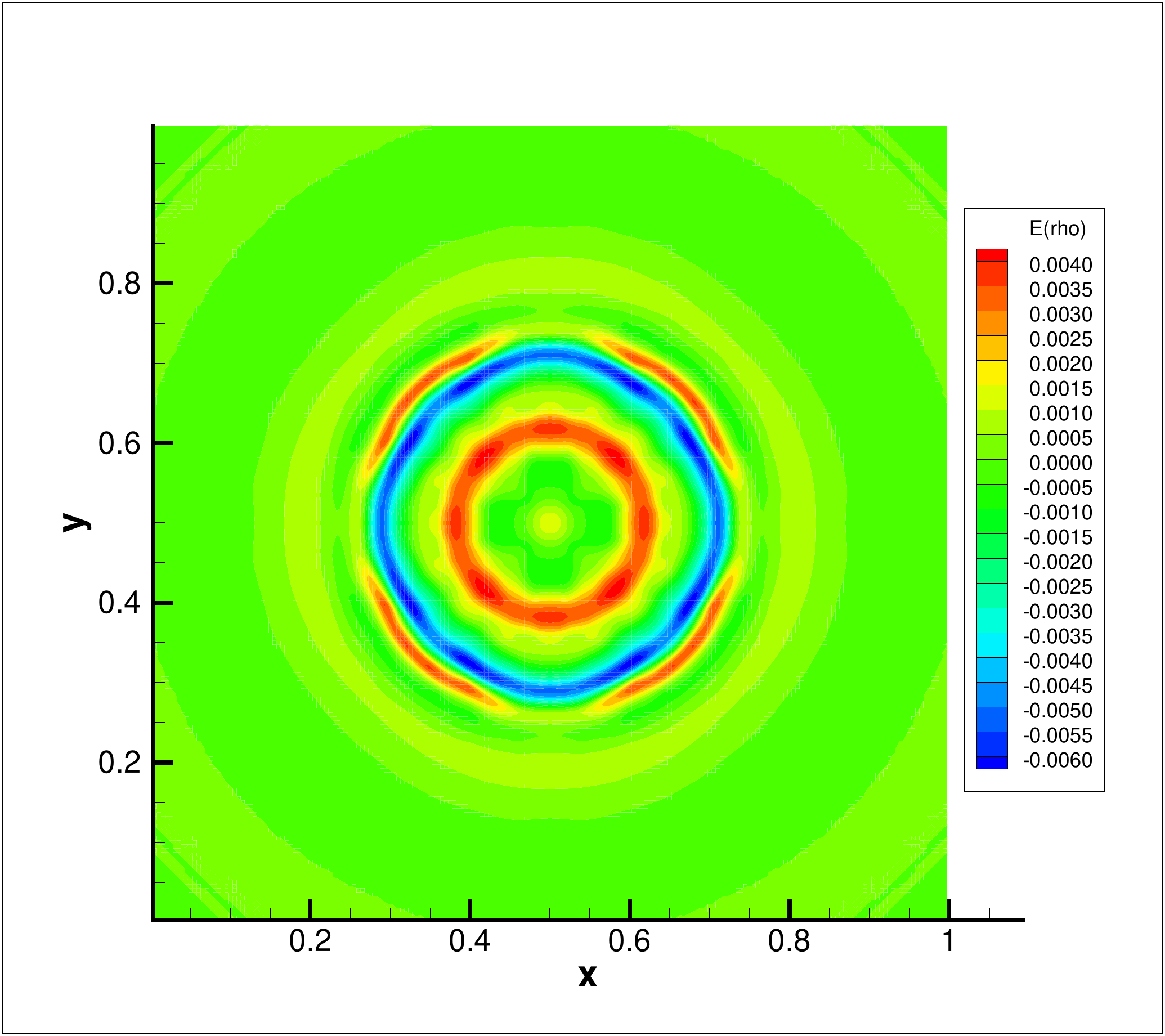}} \hfill
  \subfloat[$\mathcal{E}_{\theta}, t = 0.12$]{\includegraphics[bb=18 21 584 500, width=0.33\textwidth,clip]{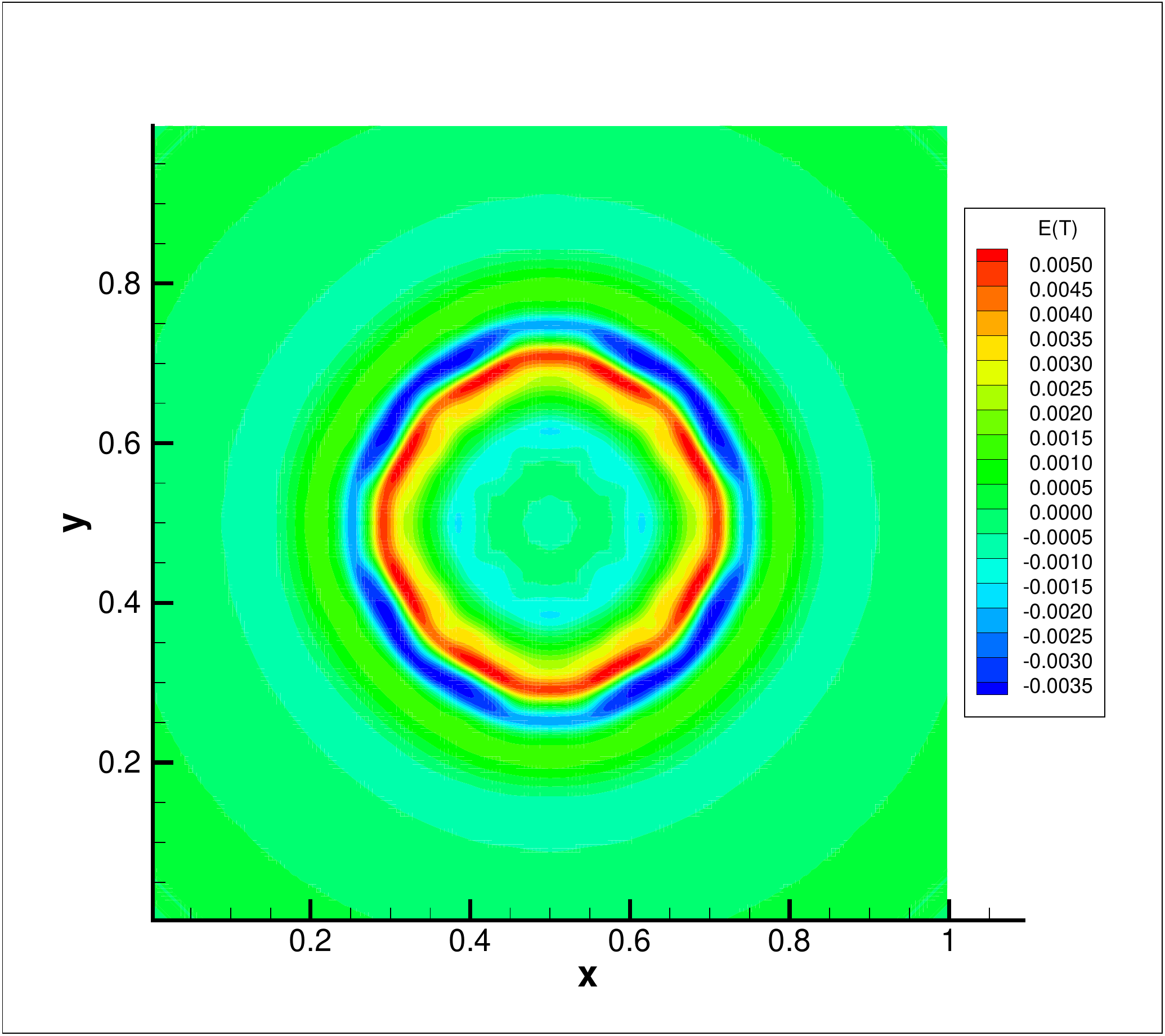}} \hfill
  \subfloat[$\mathcal{E}_{\sigma_{12}}, t = 0.12$]{\includegraphics[bb=18 21 584 500, width=0.33\textwidth,clip]{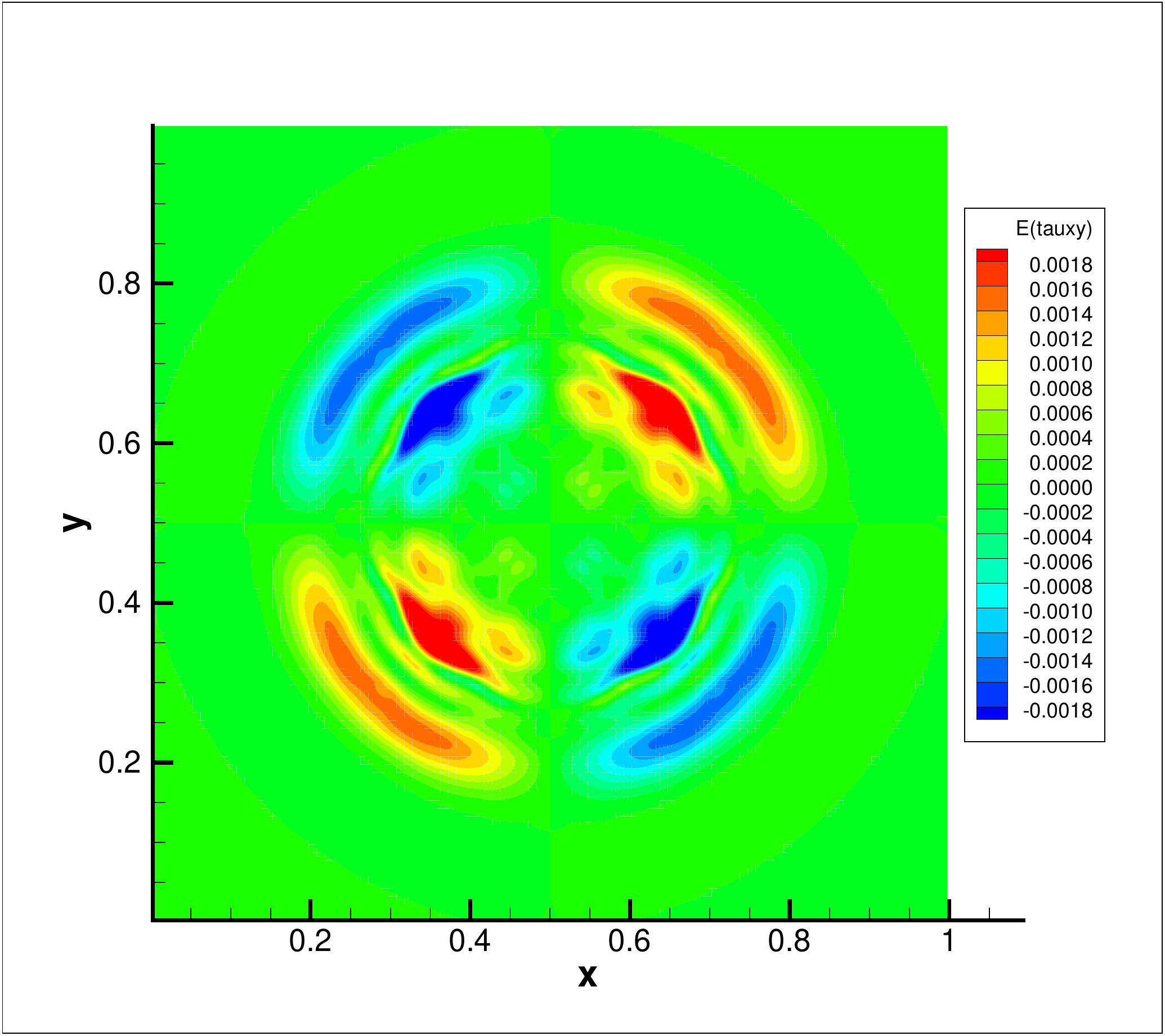}} \\
  \subfloat[$\mathcal{E}_{\rho}, t = 0.15$]{\includegraphics[bb=18 21 584 500, width=0.33\textwidth,clip]{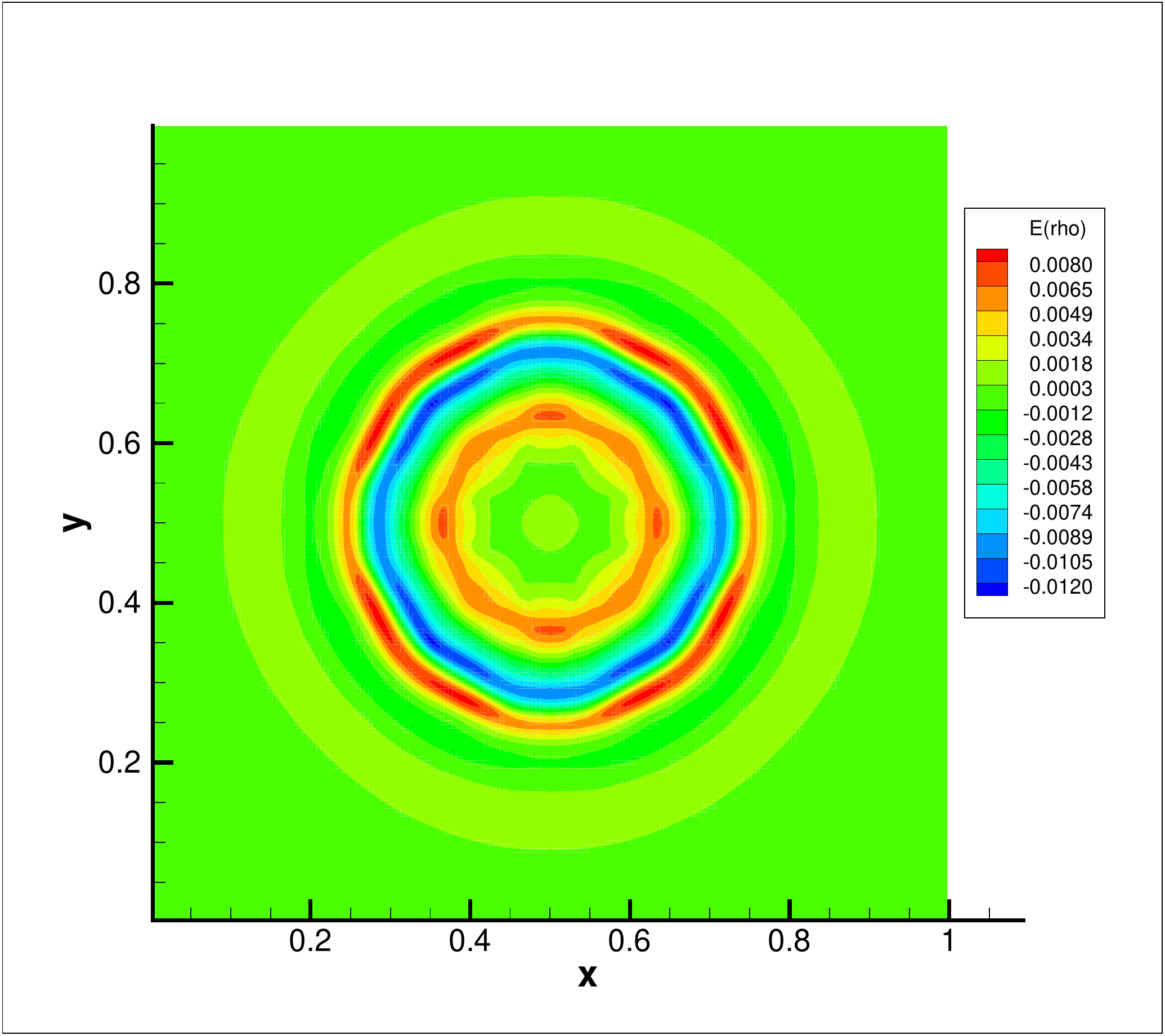}}\hfill
  \subfloat[$\mathcal{E}_{\theta}, t = 0.15$]{\includegraphics[bb=18 21 584 500, width=0.33\textwidth,clip]{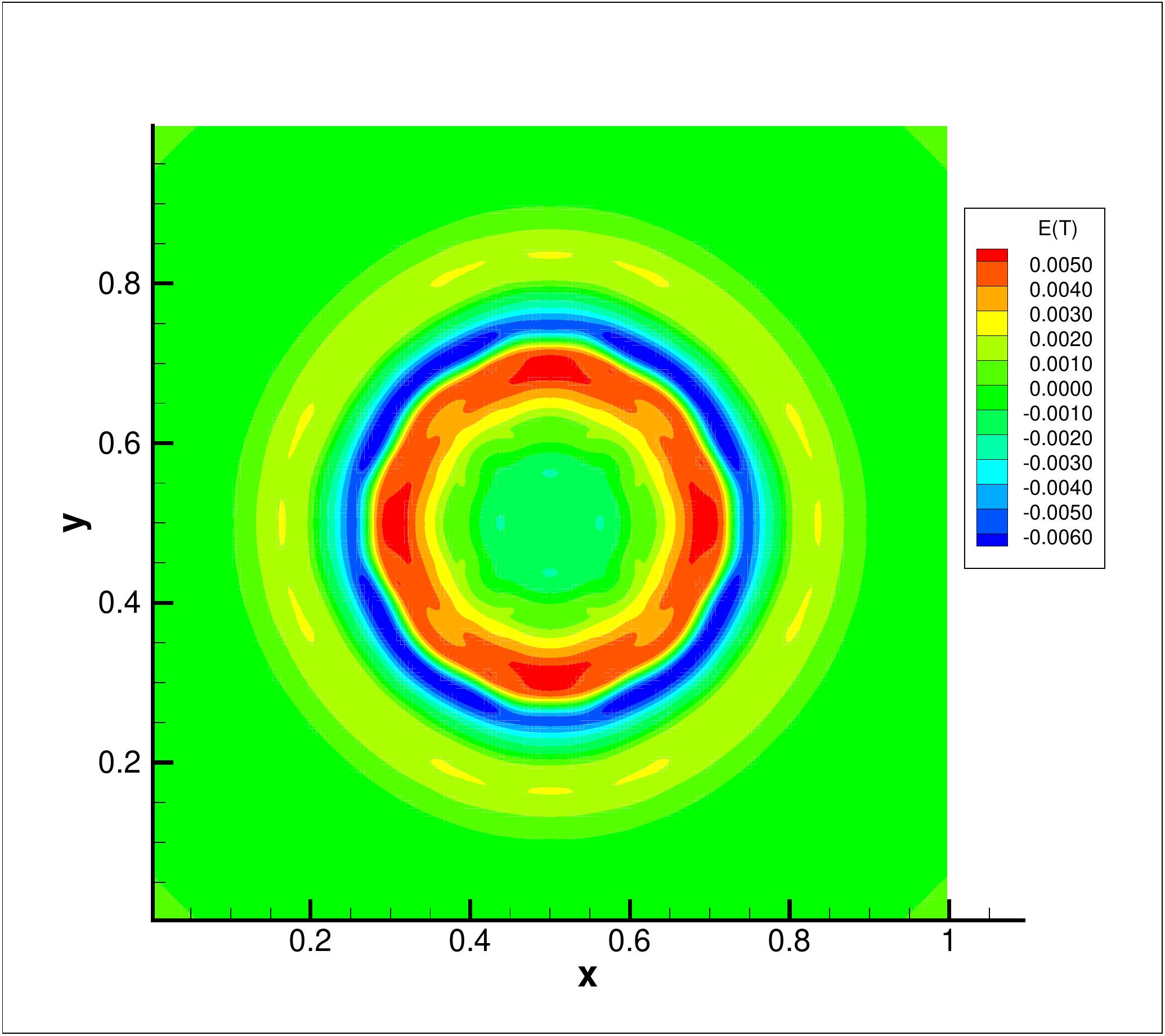}} \hfill
  \subfloat[$\mathcal{E}_{\sigma_{12}}, t=0.15$]{\includegraphics[bb=18 21 584 500, width=0.33\textwidth,clip]{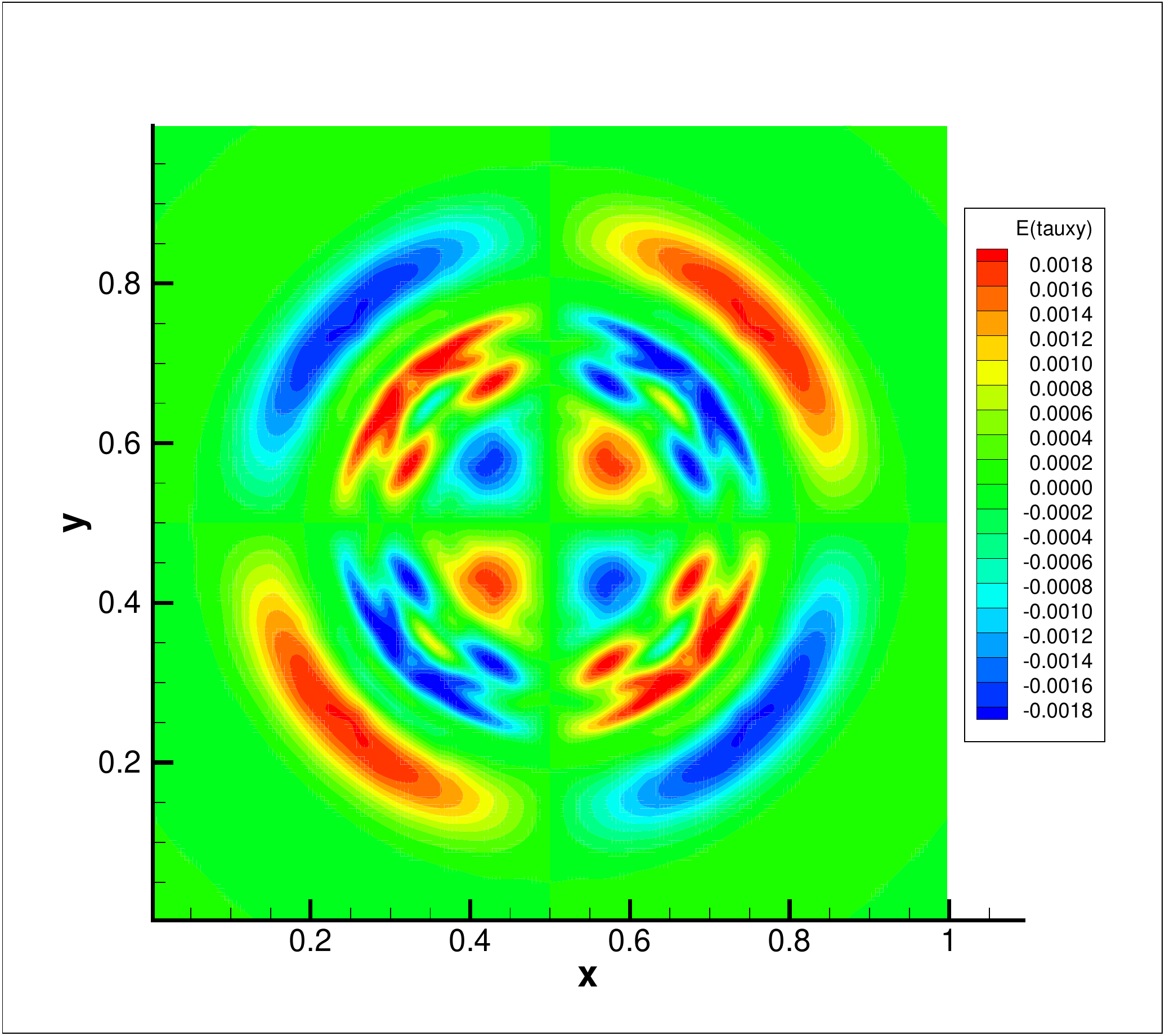}}
  \caption{Error \eqref{eq:ex3_error} of the fluid diffusion problem at different
    times.}
  \label{fig:ex4_l2_error}
\end{figure}

\clearpage

The evolution of the distribution of $M_0$ is given in Figure
\ref{fig:ex4_M0}. Initially, the mixing of two fluid regions creates
some non-equilibrium. However, due to the high density in the central
part of the domain, fast collisions of particles keep the fluid
near its local equilibrium, so that $M_0$ is generally
not too large at $t = 0.04$. As $t$ increases, both the density and
the temperature in the central area decrease, and correspondingly,
$M_0$ needs to be increased to capture the non-equilibrium effects.
From $t = 0.12$ to $t = 0.15$, although the fluid has spread more
widely, the outside layers are almost in the equilibrium states, and
therefore the distribution of $M_0$ does not change significantly.
Compared with the reference solution, the average CPU time per time
step is reduced from $1473$ seconds to $123.7$ seconds using our
method.

\begin{figure}[!htb]
% \colorbox{black}
  \centering
  \hfill \subfloat[$t = 0. 04$]{\includegraphics[bb=18 21 584 500,width=0.33\textwidth,clip]{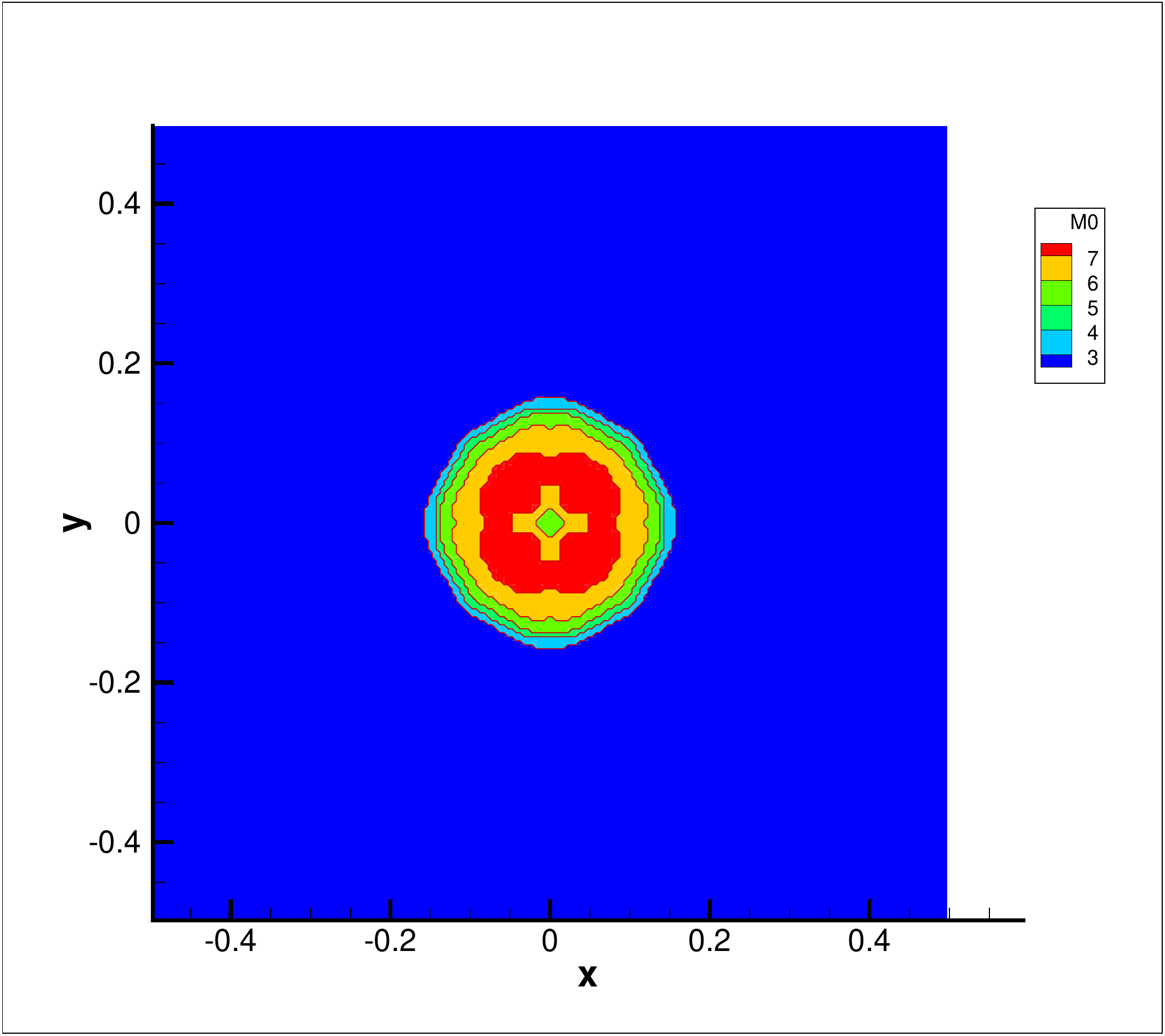}}\hfill
  \subfloat[$t = 0.08$]{\includegraphics[bb=18 21 584 500,width=0.33\textwidth,clip]{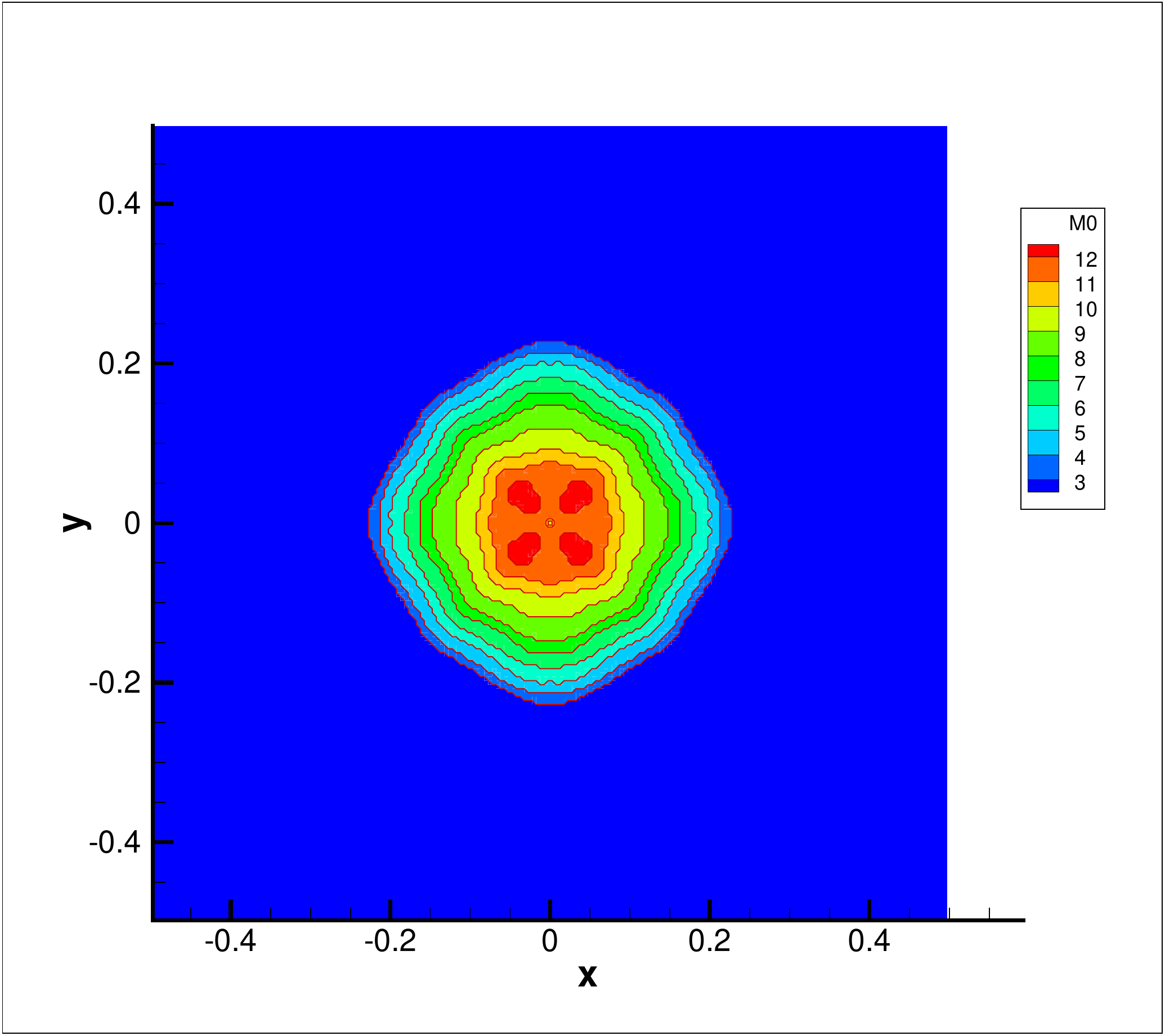}} \hfill \mbox{}\\
  \hfill \subfloat[$t = 0.12$]{\includegraphics[bb=18 21 584 500,width=0.33\textwidth,clip]{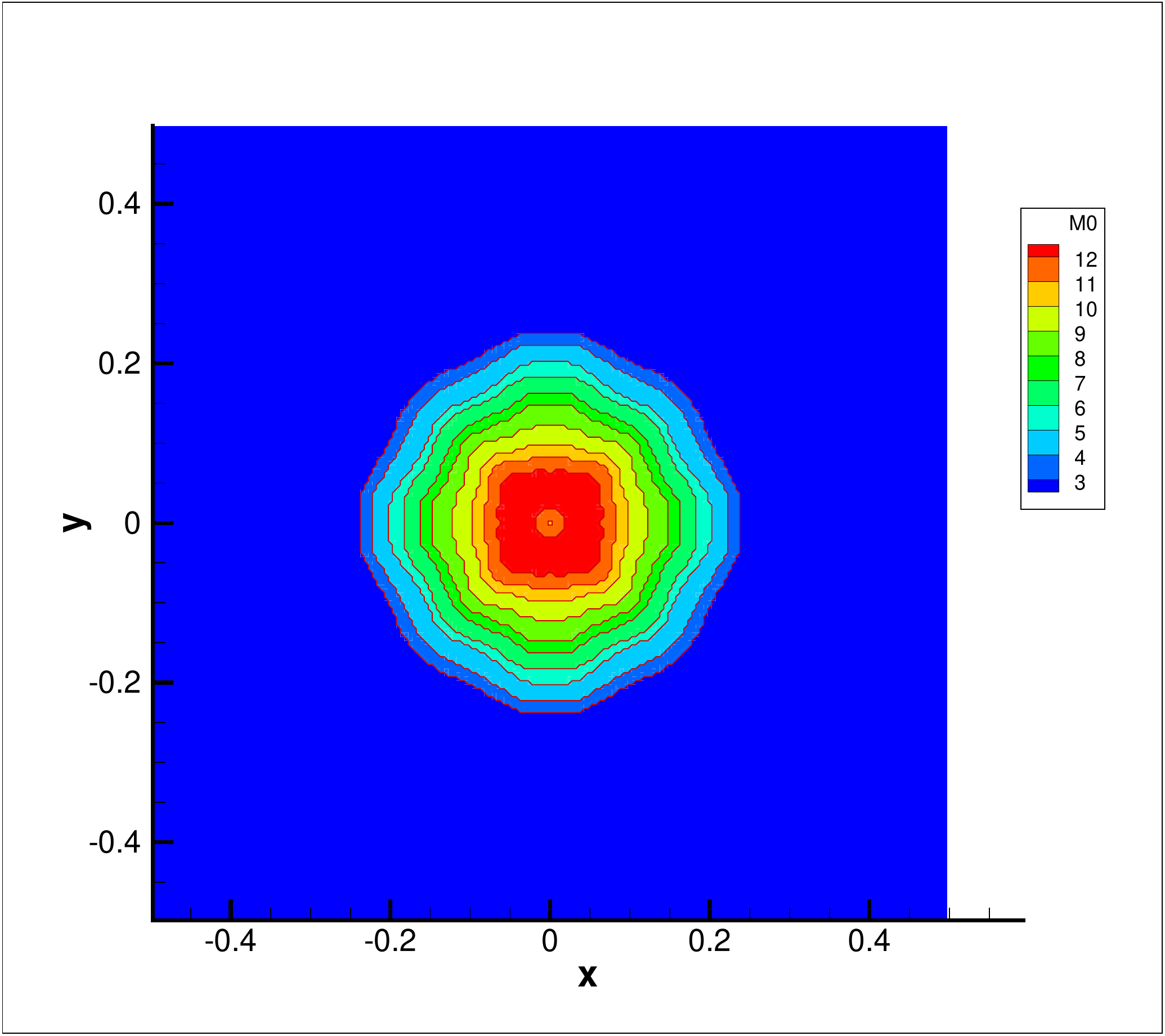}}\hfill
  \subfloat[$t = 0.15$]{\includegraphics[bb=18 21 584 500,width=0.33\textwidth,clip]{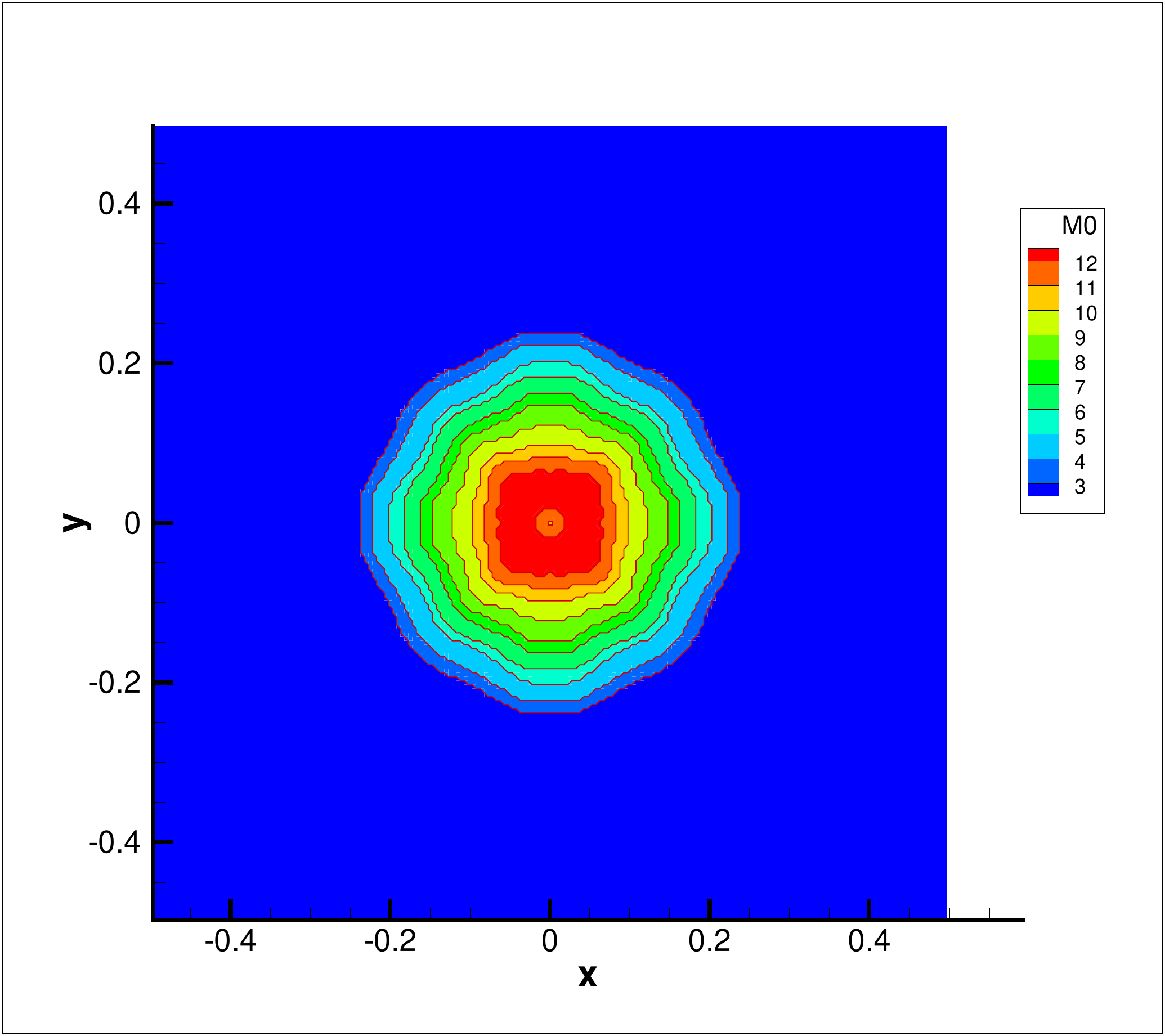}} \hfill \mbox{} 
 \caption{Distribution of $M_0$ for fluid diffusion at different times. } \label{fig:ex4_M0}
\end{figure}

\subsubsection{Lid-driven cavity flow} \label{sec:cavity_flow}
Our second example assumes that the gas is in a square cavity $\Omega = [0,1] \times [0,1]$, and we scale the collision term such that the Knudsen number is $0.1$. The top lid of the cavity moves horizontally at a constant speed $v = 0.0208$, and all the four sides of $\Omega$ are assumed to be fully diffusive. Initially, the fluid is in the equilibrium state with density $\rho = 1$, velocity $\bu = \boldsymbol{0}$ and temperature $\theta = 1$. The friction between the lid and the gas causes the rotation of the fluid, and a steady state will be developed after a sufficiently long time. Such an example has been widely studied in the literature \cite{John2010, Liu2016}. Here we discretize the domain $\Omega$ with a uniform grid of size $100\times 100$. The parameters in \eqref{eq:truncated_series} are set to be $M = 40$, $\bar{\bu} = \boldsymbol{0}$ and $\bar{\theta} = 1$. The reference solution will again be provided by the uniform $M_0$ equal to $15$.

To study the effect of adaptive parameters, we consider two groups of thresholds:  $(\epsilon_1^1, \epsilon_2^1) = (0.05, 0.20)$ and $(\epsilon_1^2, \epsilon_2^2) = (0.025, 0.08)$. Clearly, the second set of parameters is tighter and will lead to larger $M_0$ in the simulations. The evolution of the fluid states is plotted in Figure \ref{fig:ex3_sol}, which includes the density $\rho$, the temperature $\theta$, and shear stress $\sigma_{12}$  at time $t =0.5, 1, 5$ and the steady state. One can observe the singular flow structure in the top two corners of the cavity, where the distributions are distorted due to the inconsistent boundary velocities. In Figure \ref{fig:ex3_sol}, three sets of solutions generally agree with each other. Some differences can be observed in the second column representing the temperature contours. Despite this, the relative difference in temperature between our results and the reference solution is well below $0.05\%$. In the first and third columns, all the three sets of contour lines almost coincide with each other, indicating the effectiveness of our error indicator.

The distribution of $M_0$ is given in Figure \ref{fig:ex3_M0}. Since the flow is driven by the movement of the top lid, non-equilibrium emerges from the upper part of the domain, and then expands downward as $t$ increases. For the first set of parameters (left column), the value of $M_0$ reaches the cap $15$ only near the boundary of the domain, where the distribution function is discontinuous, while in the right column, more than a half of the grid cells are covered by the collision term with $M_0 = 15$, which is consistent with our prediction.

\begin{figure}[!htb]
% \colorbox{black}
  \centering
  \subfloat[$\rho, t = 0.5$]{\includegraphics[bb=18 21 584 500, width=0.33\textwidth,clip]{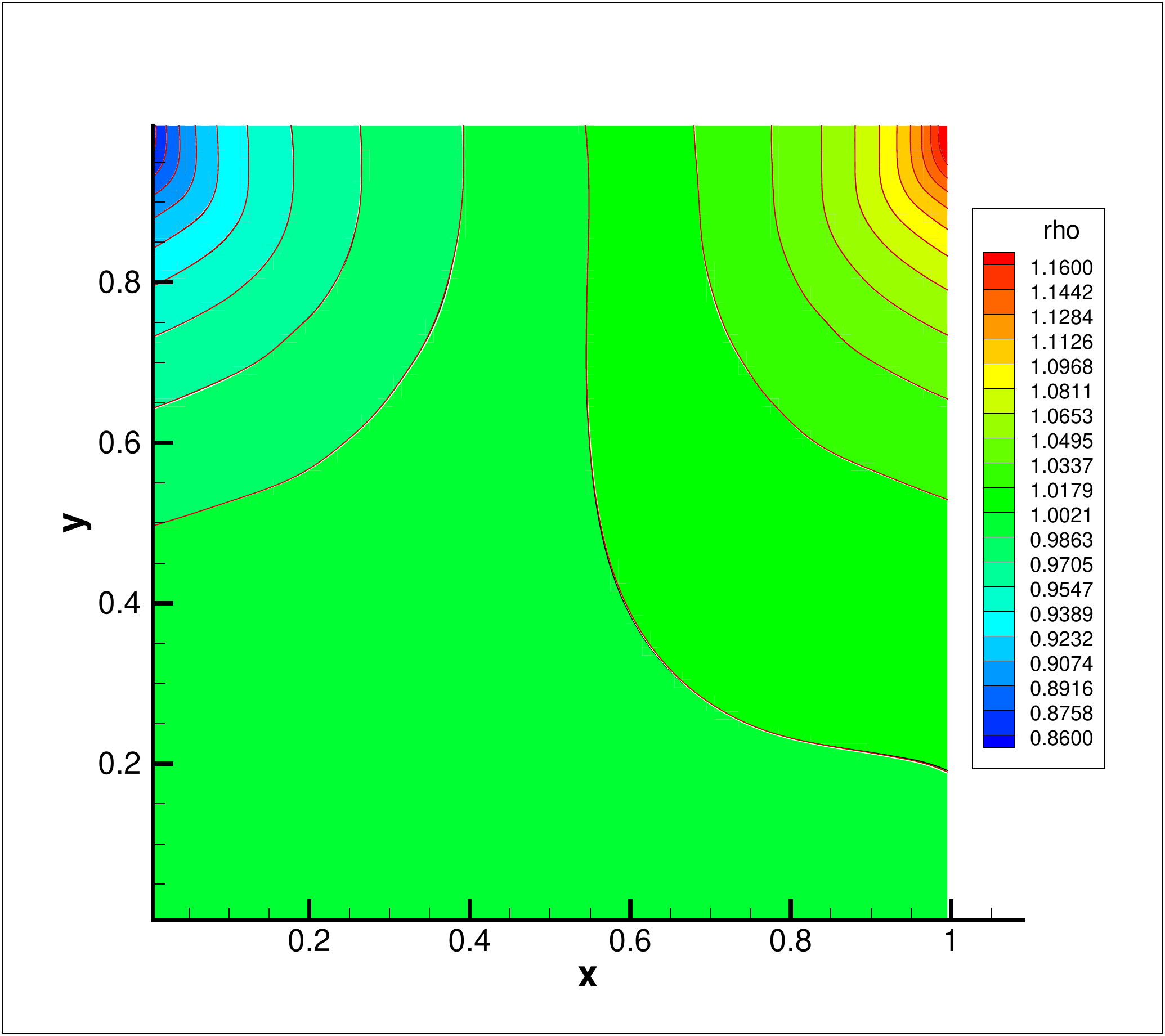}}\hfill
  \subfloat[$\theta, t = 0.5$ ]{\includegraphics[bb=18 21 584 500, width=0.33\textwidth,clip]{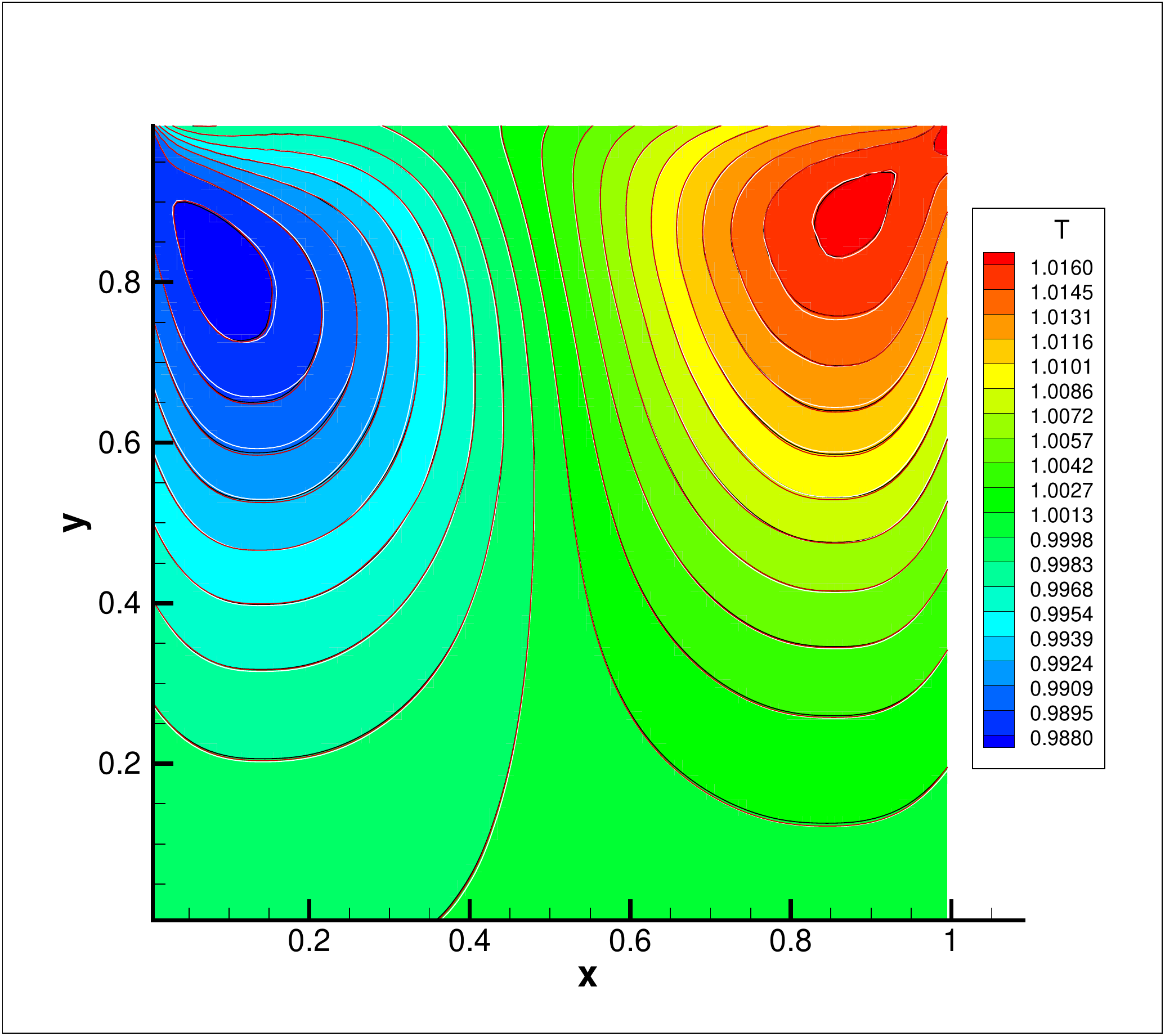}}\hfill
  \subfloat[$\sigma_{12}, t = 0.5$ ]{\includegraphics[bb=18 21 584 500, width=0.33\textwidth,clip]{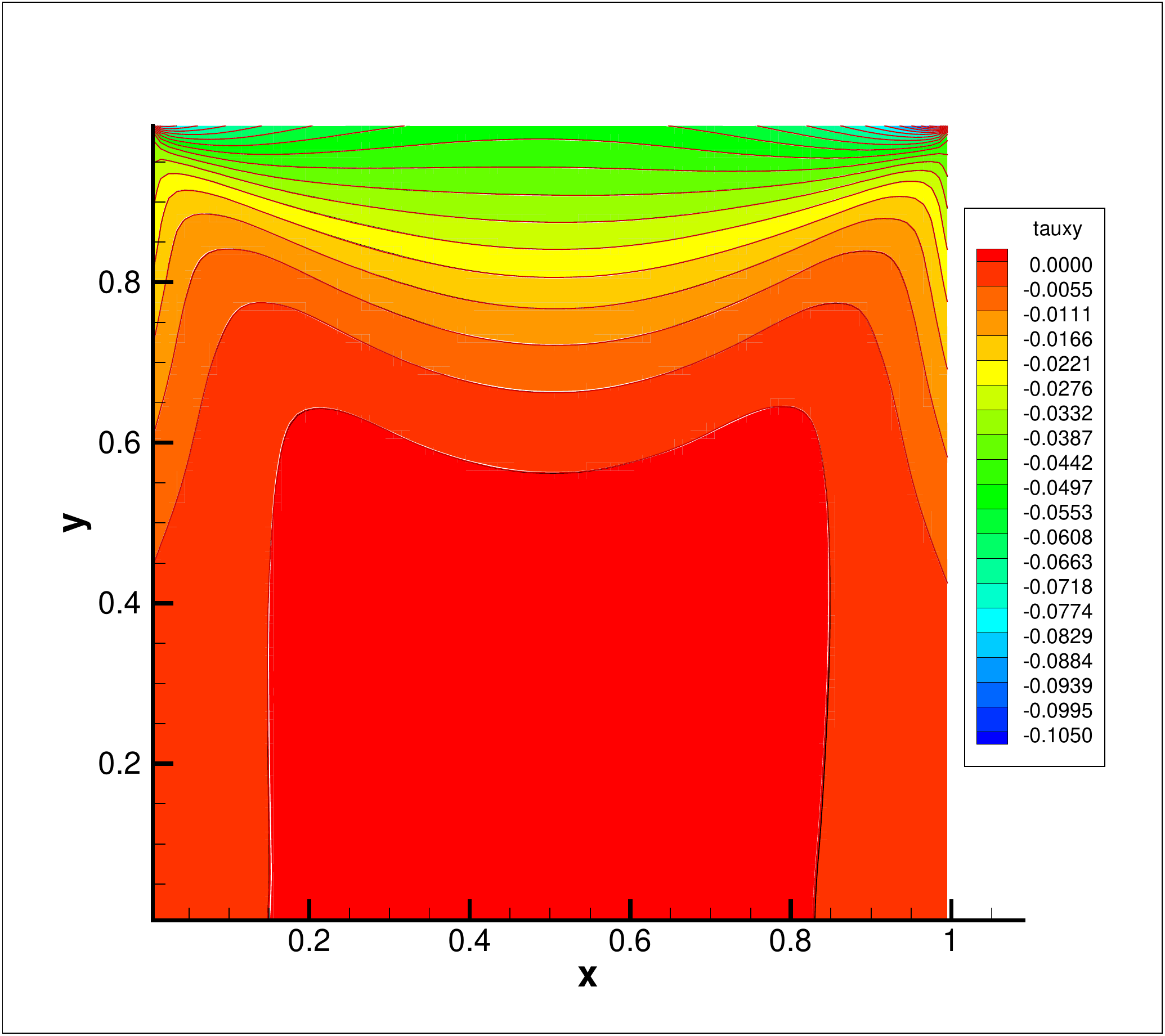}}\hfill \\
  \subfloat[$\rho, t = 1$]{\includegraphics[bb=18 21 584 500, width=0.33\textwidth,clip]{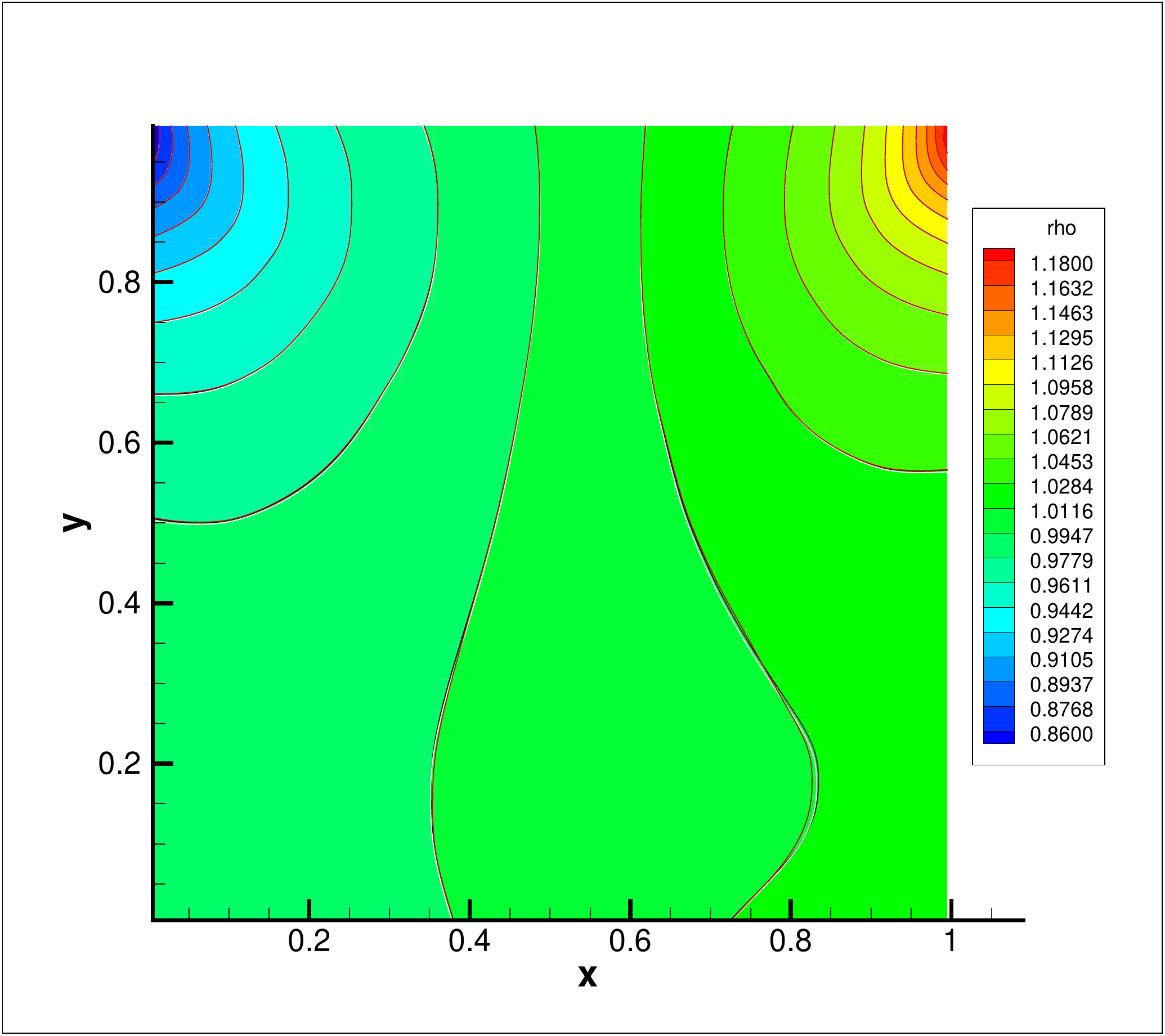}} \hfill
  \subfloat[$\theta, t = 1$]{\includegraphics[bb=18 21 584 500, width=0.33\textwidth,clip]{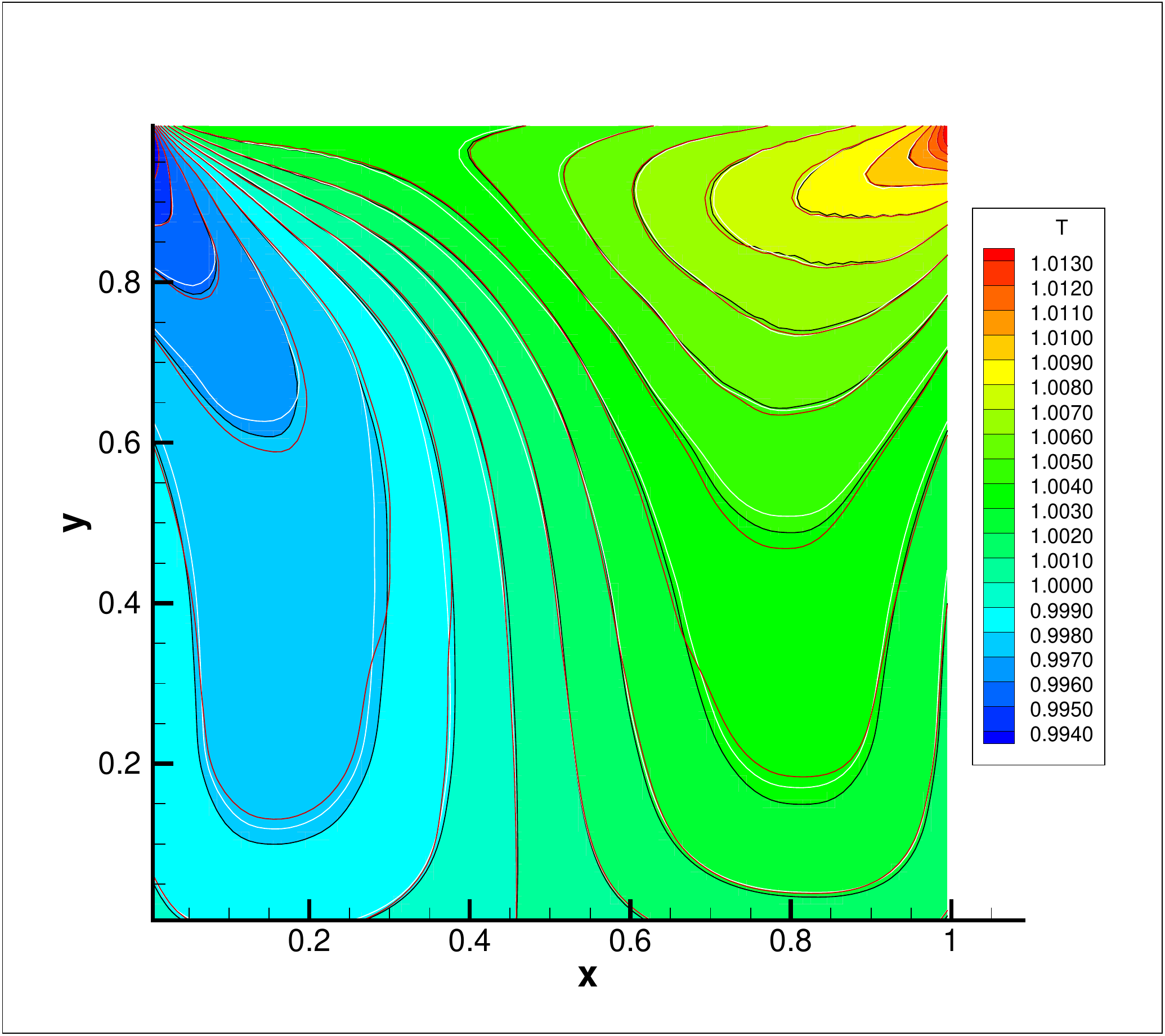}} \hfill
  \subfloat[$\sigma_{12}, t = 1$]{\includegraphics[bb=18 21 584 500, width=0.33\textwidth,clip]{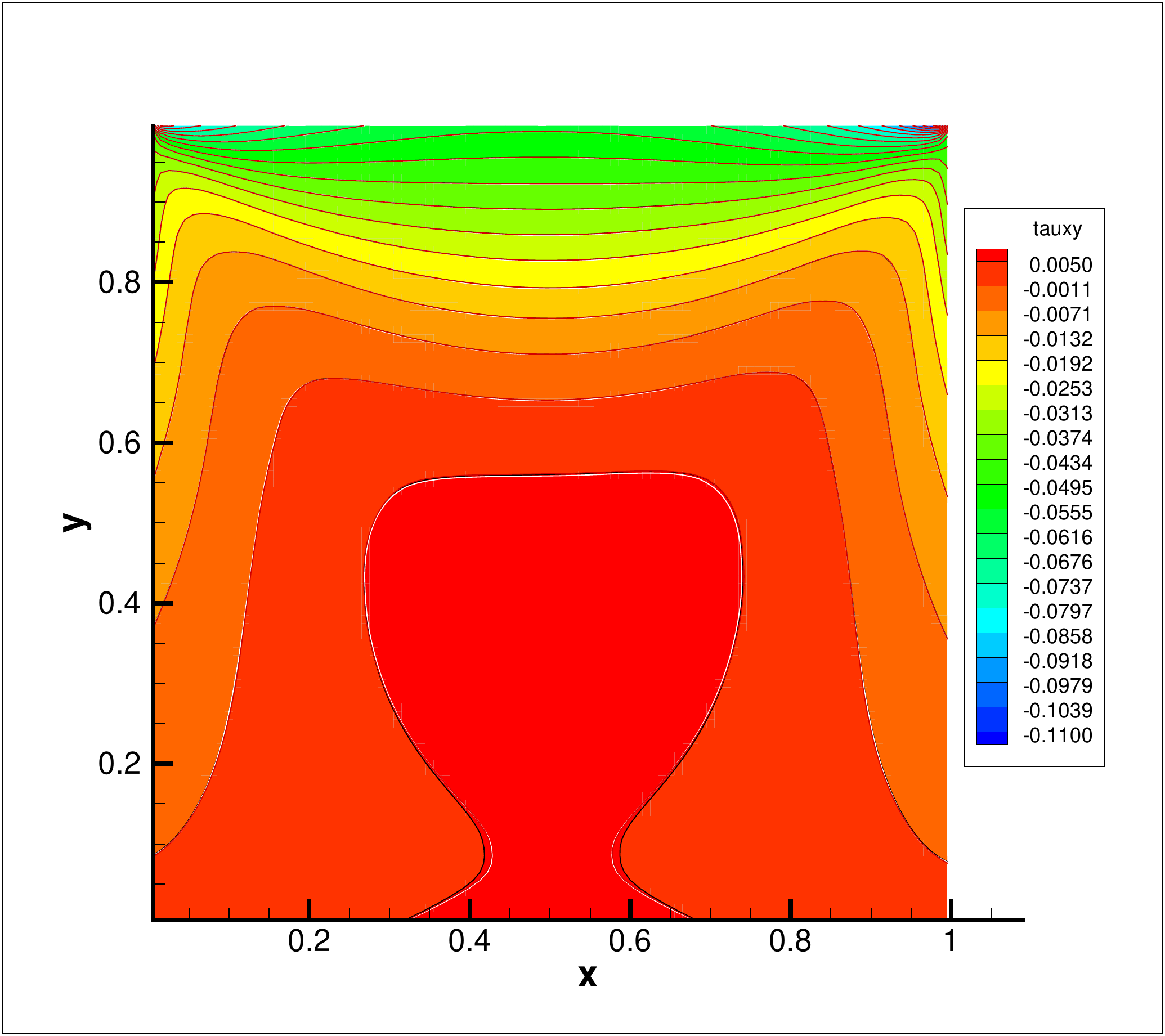}} \hfill \\
  \subfloat[$\rho, t = 5$]{\includegraphics[bb=18 21 584 500,width=0.33\textwidth,clip]{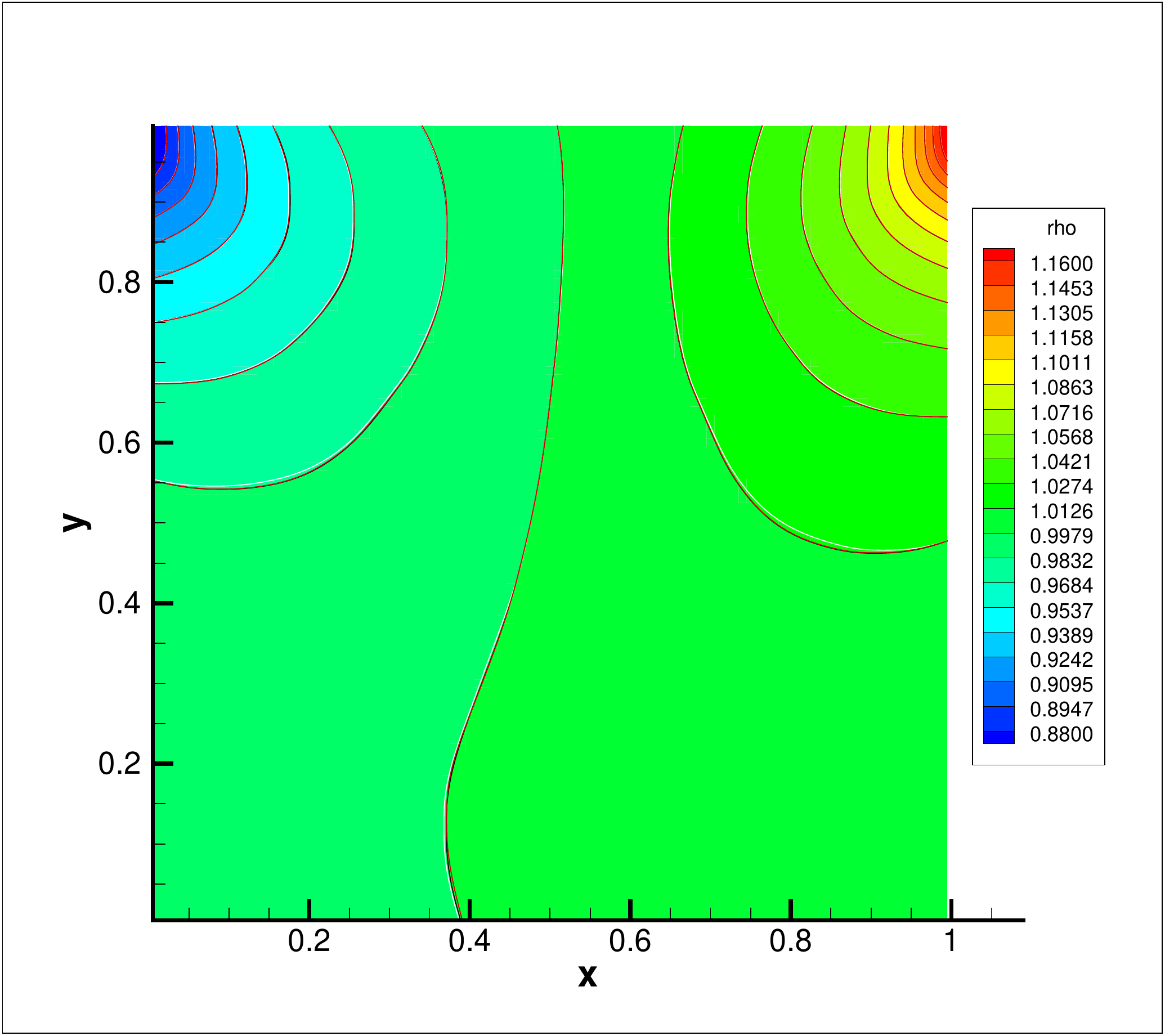}} \hfill
  \subfloat[$\theta, t = 5$]{\includegraphics[bb=18 21 584 500, width=0.33\textwidth,clip]{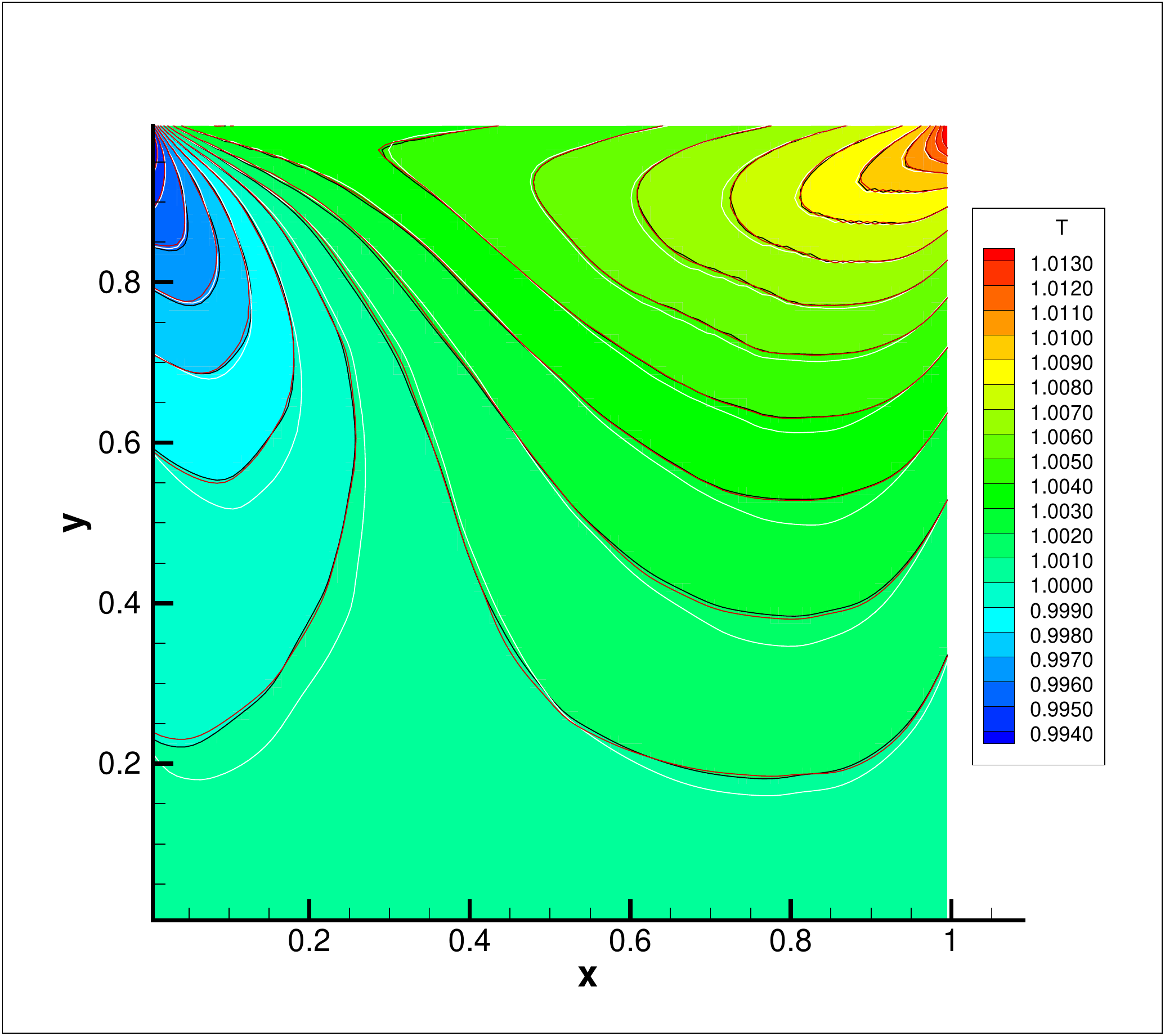}} \hfill
  \subfloat[$\sigma_{12}, t = 5$]{\includegraphics[bb=18 21 584 500, width=0.33\textwidth,clip]{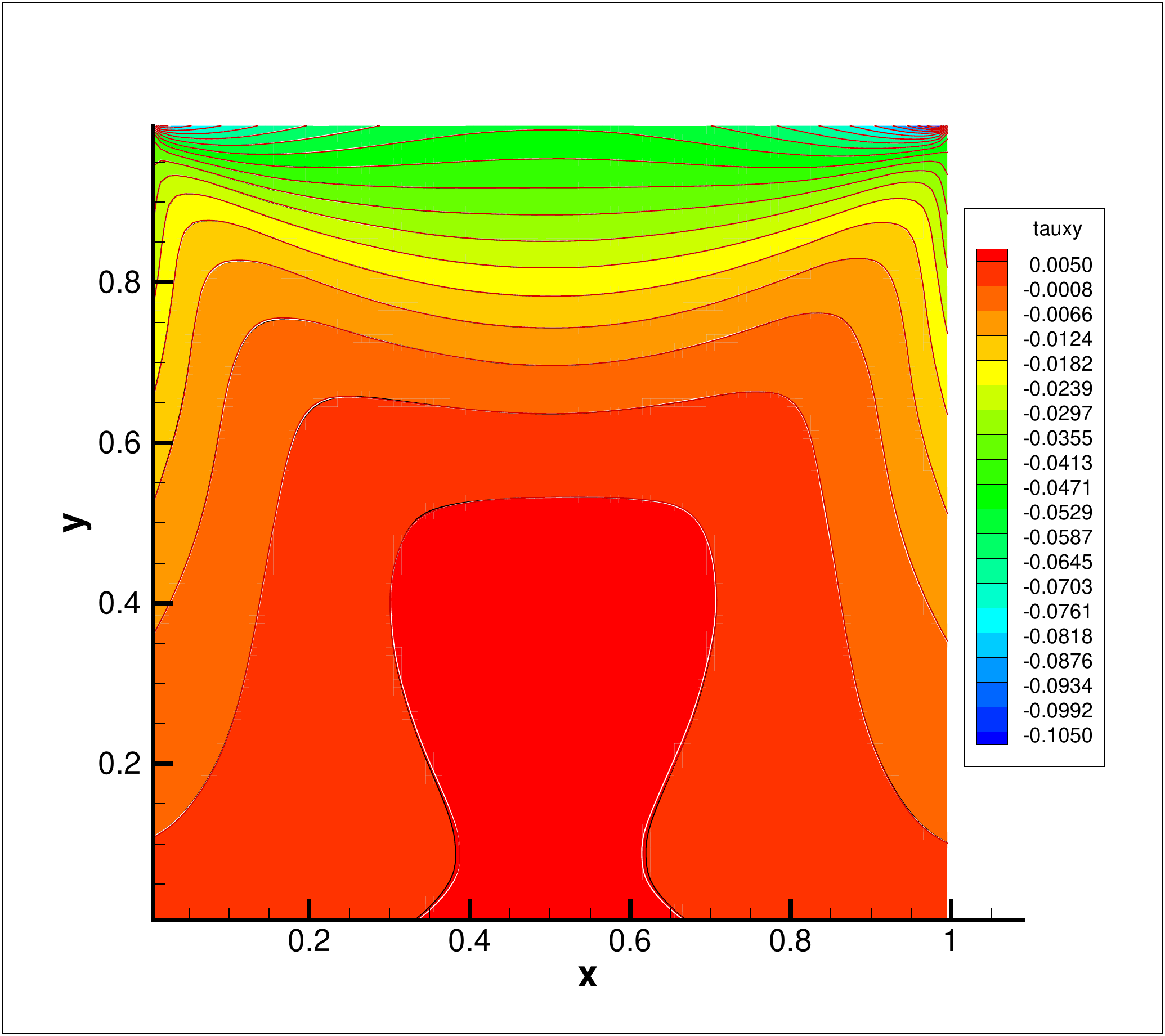}} \\
  \subfloat[$\rho$, steady state]{\includegraphics[bb=18 21 584 500, width=0.33\textwidth,clip]{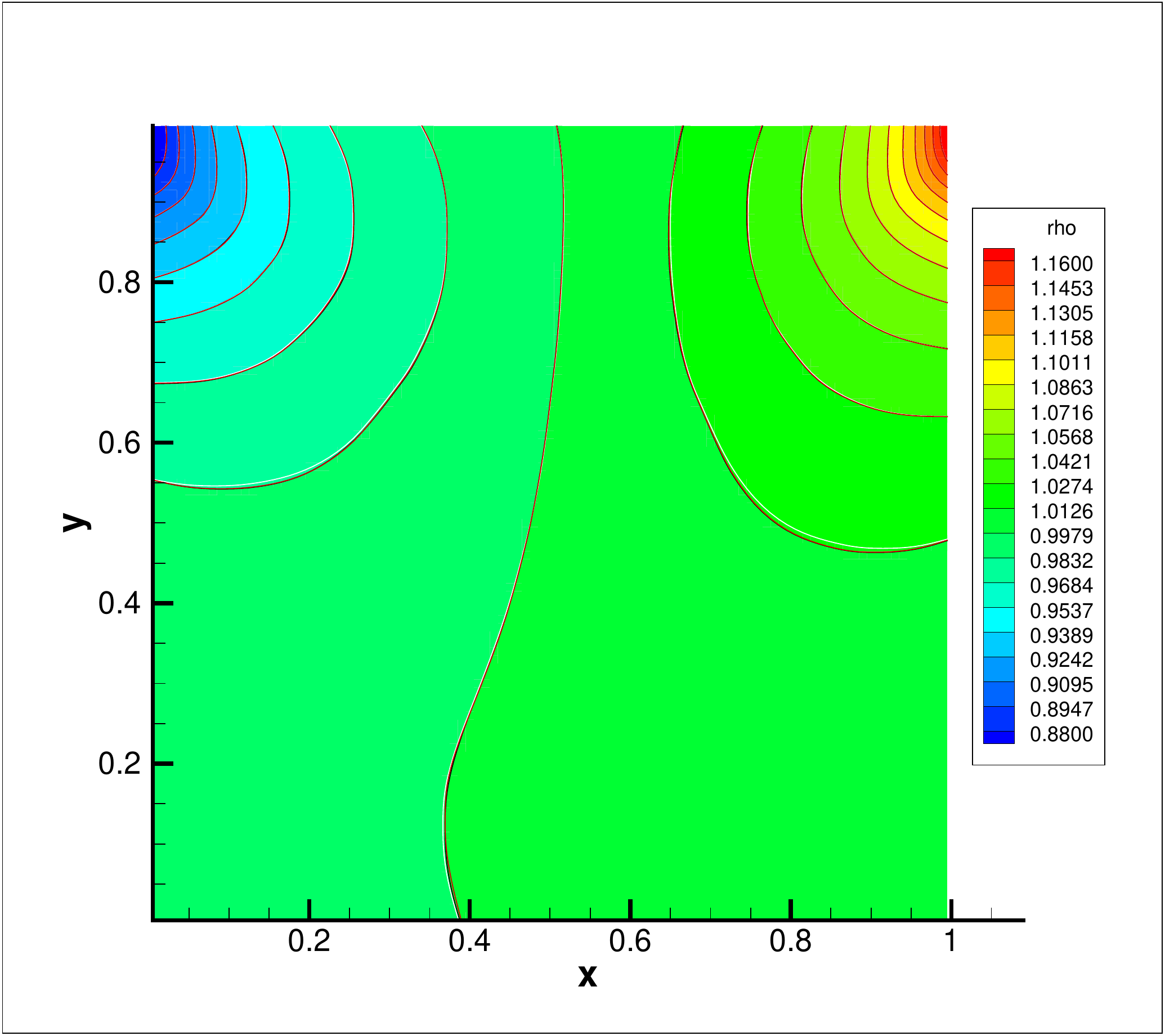}}\hfill
  \subfloat[$\theta$, steady state]{\includegraphics[bb=18 21 584 500, width=0.33\textwidth,clip]{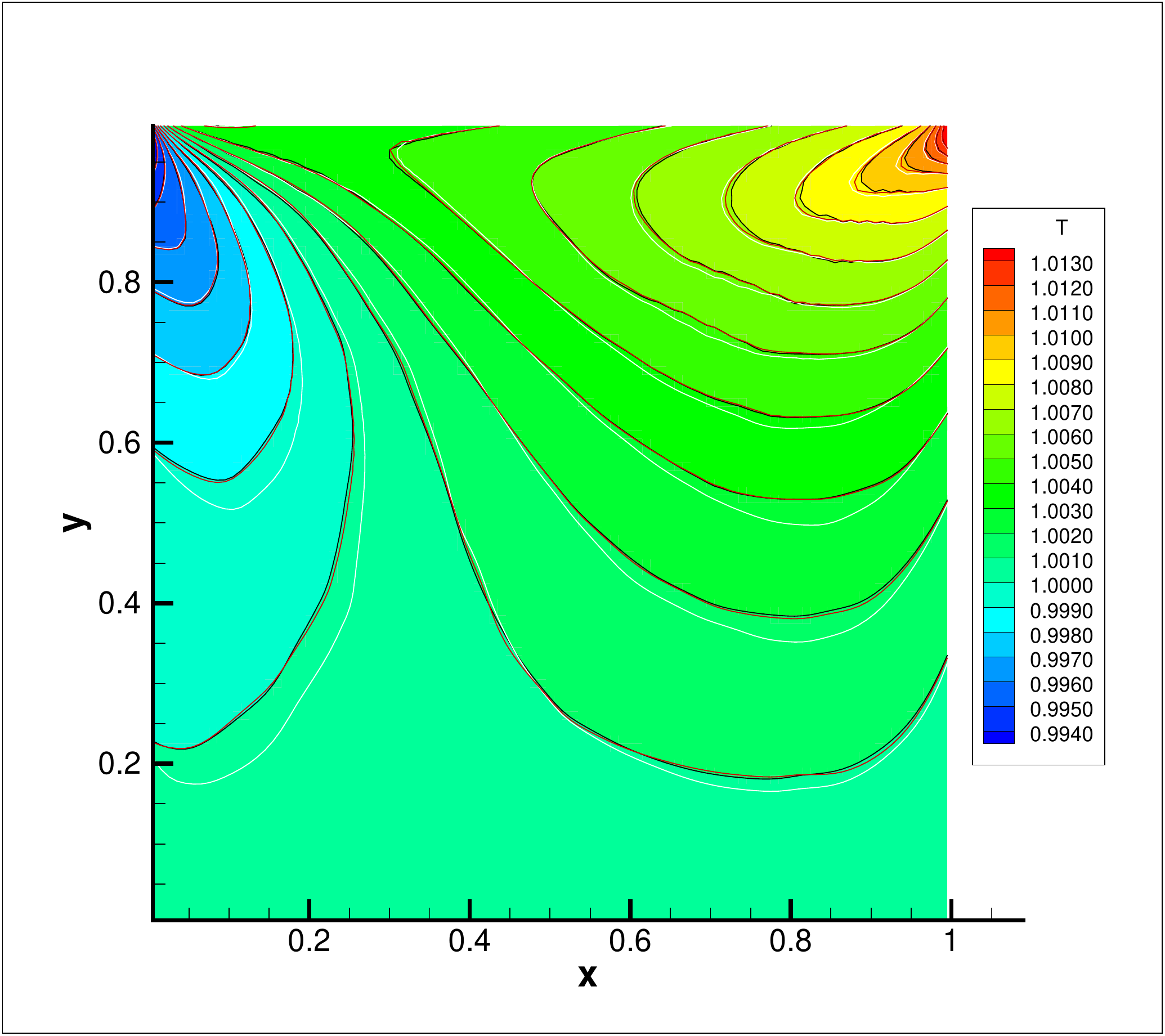}} \hfill
  \subfloat[$\sigma_{12}$, steady state]{\includegraphics[bb=18 21 584 500, width=0.33\textwidth,clip]{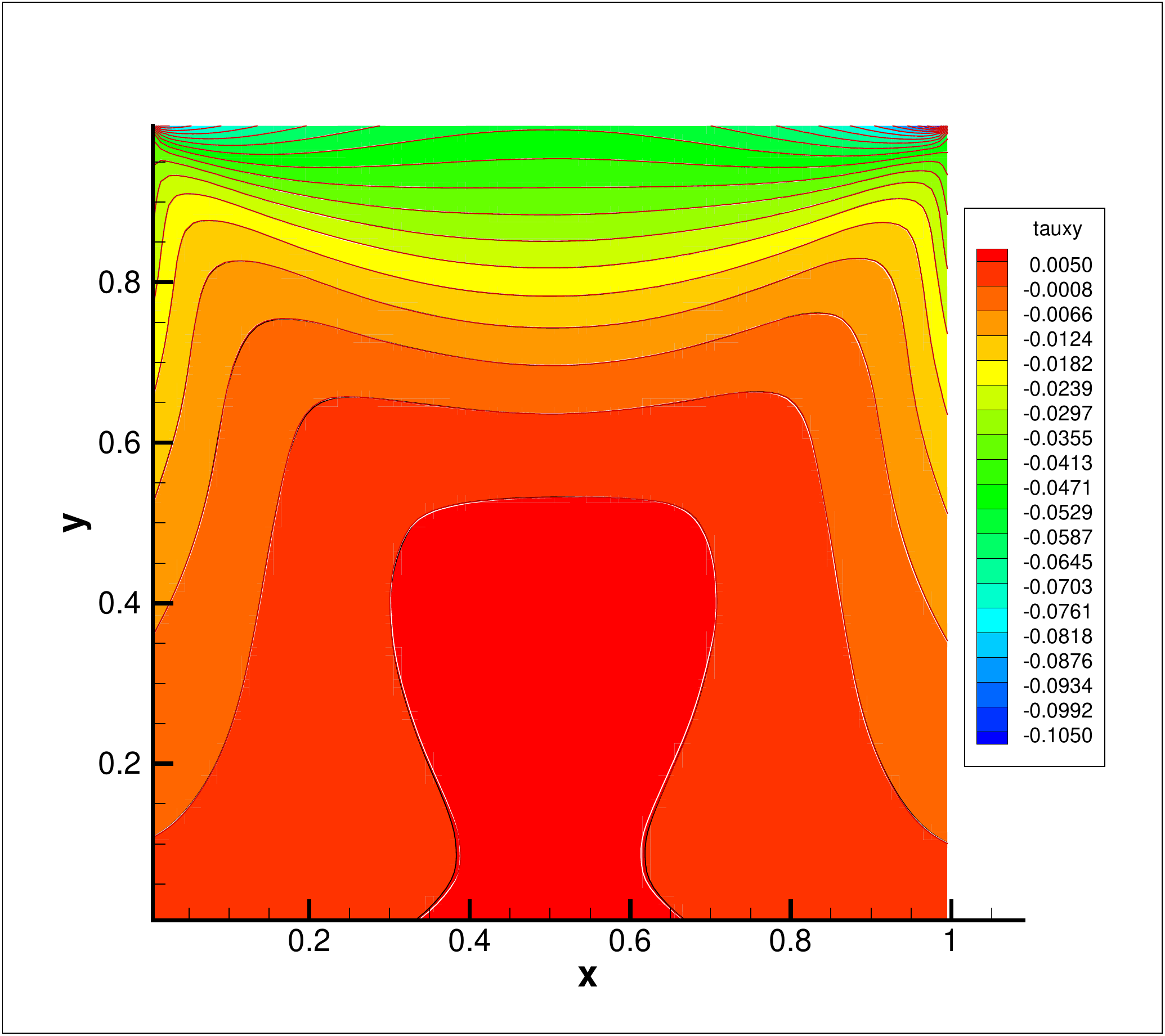}}
   \caption{Solution of the lid-driven cavity flow at different times. The white contours and the red contours are the numerical solutions with threshold parameters $(\epsilon_1^1, \epsilon_2^1)$ and $(\epsilon_1^2, \epsilon_2^2)$, respectively. The black contours are the reference solution.}
  \label{fig:ex3_sol}
\end{figure}

\begin{figure}[!htb]
% \colorbox{black}
  \centering
  \hfill \subfloat[$(\epsilon_1^1, \epsilon_2^1), t = 0.5$]{\includegraphics[bb=18 21 584 500,width=0.33\textwidth,clip]{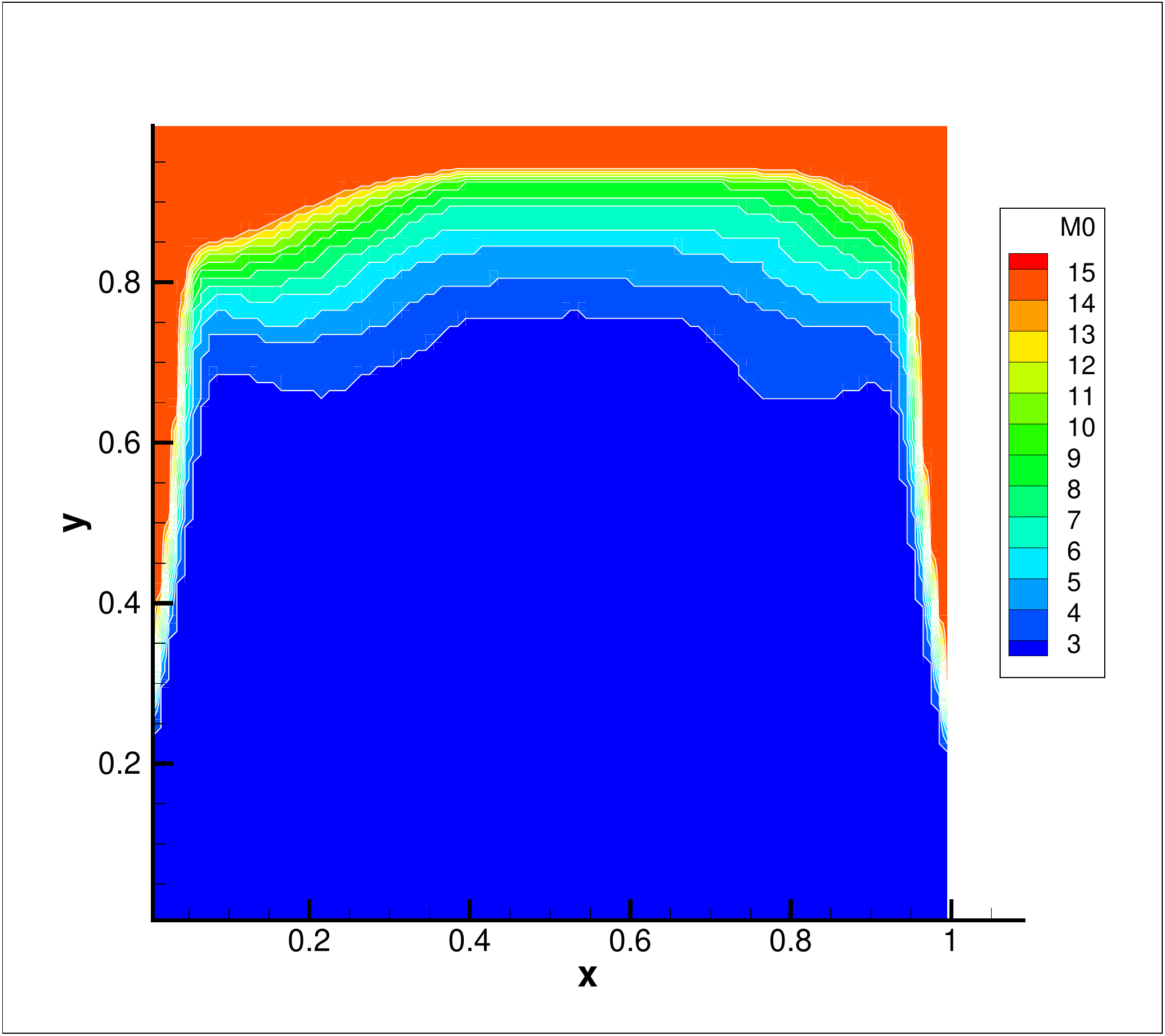}}\hfill
  \subfloat[$(\epsilon_1^2, \epsilon_2^2), t = 0.5$]{\includegraphics[bb=18 21 584 500,width=0.33\textwidth,clip]{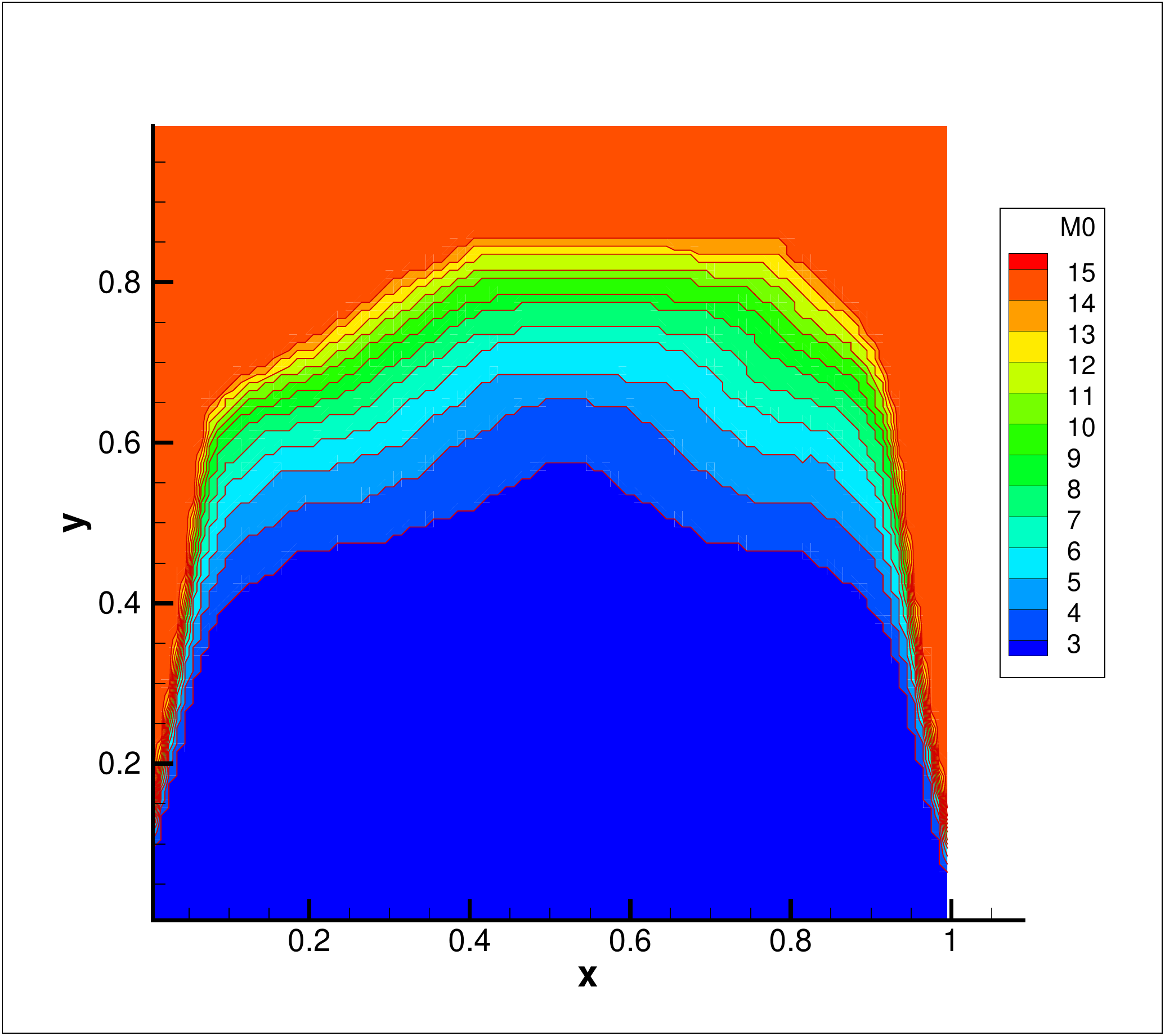}} \hfill \mbox{}\\
  \hfill \subfloat[$(\epsilon_1^1, \epsilon_2^1), t = 1$]{\includegraphics[bb=18 21 584 500,width=0.33\textwidth,clip]{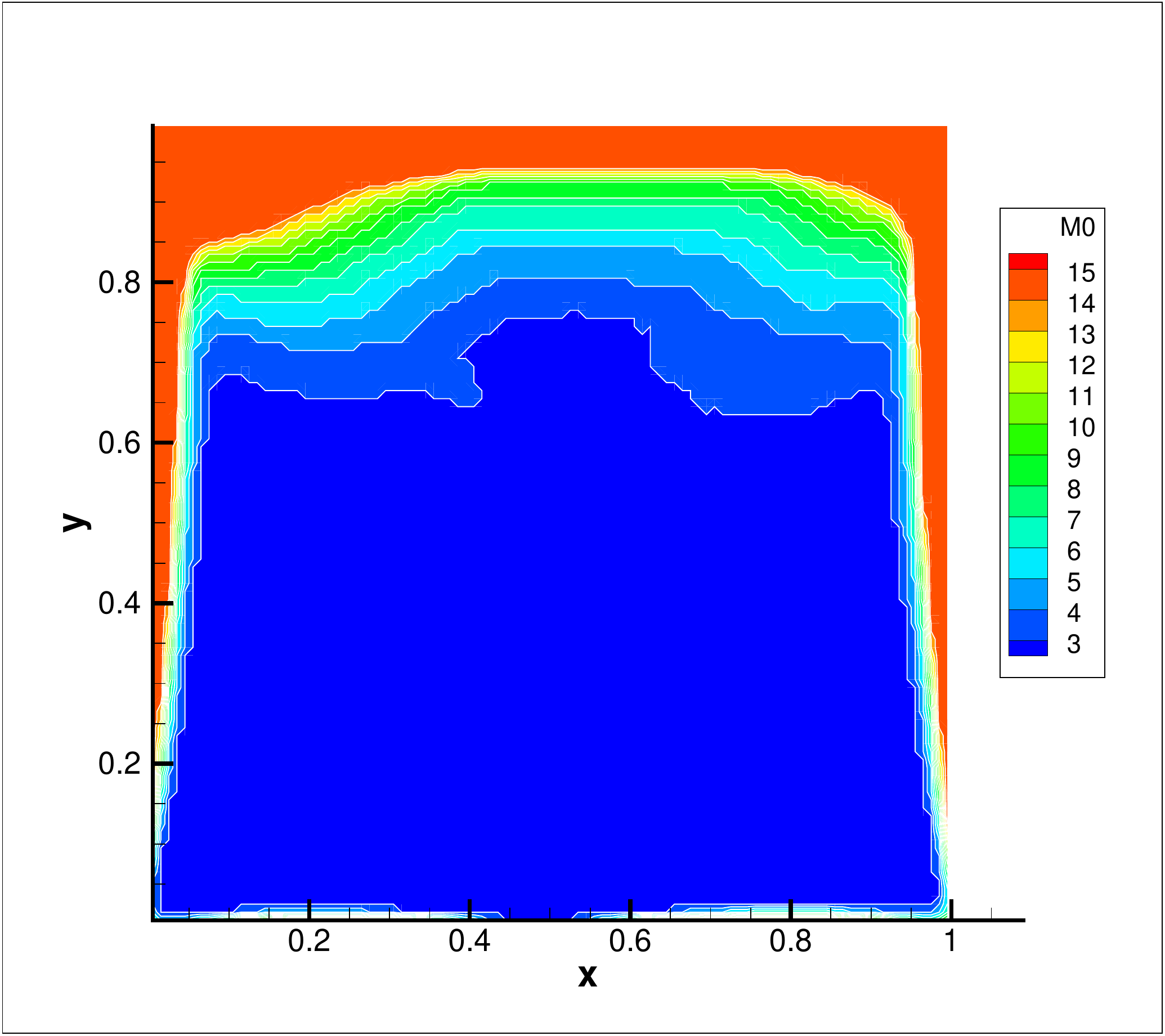}}\hfill
  \subfloat[$(\epsilon_1^2,\epsilon_2^2), t = 1$]{\includegraphics[bb=18 21 584 500,width=0.33\textwidth,clip]{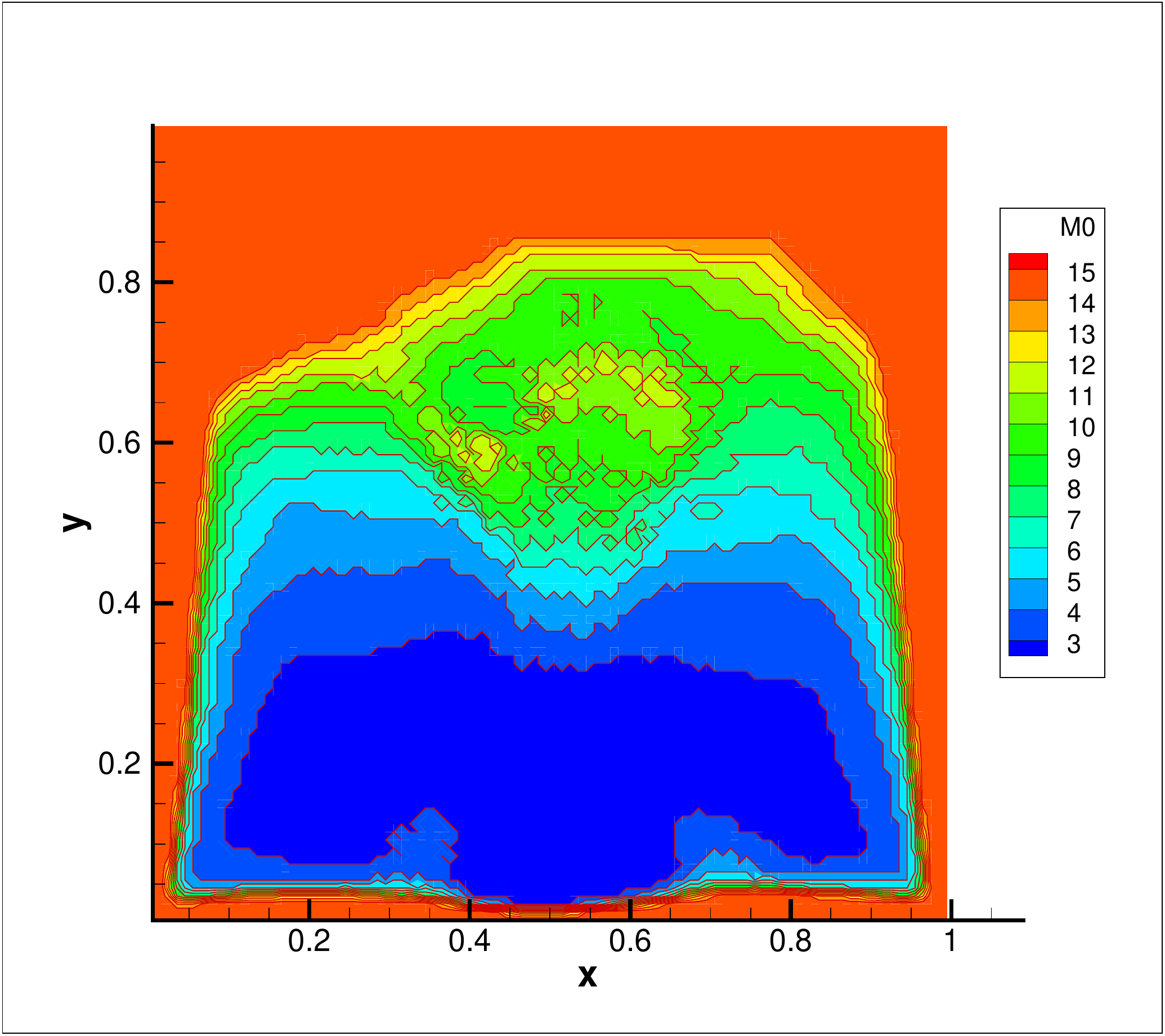}} \hfill \mbox{} \\
  \hfill \subfloat[$(\epsilon_1^1, \epsilon_2^1), t = 5$]{\includegraphics[bb=18 21 584 500,width=0.33\textwidth,clip]{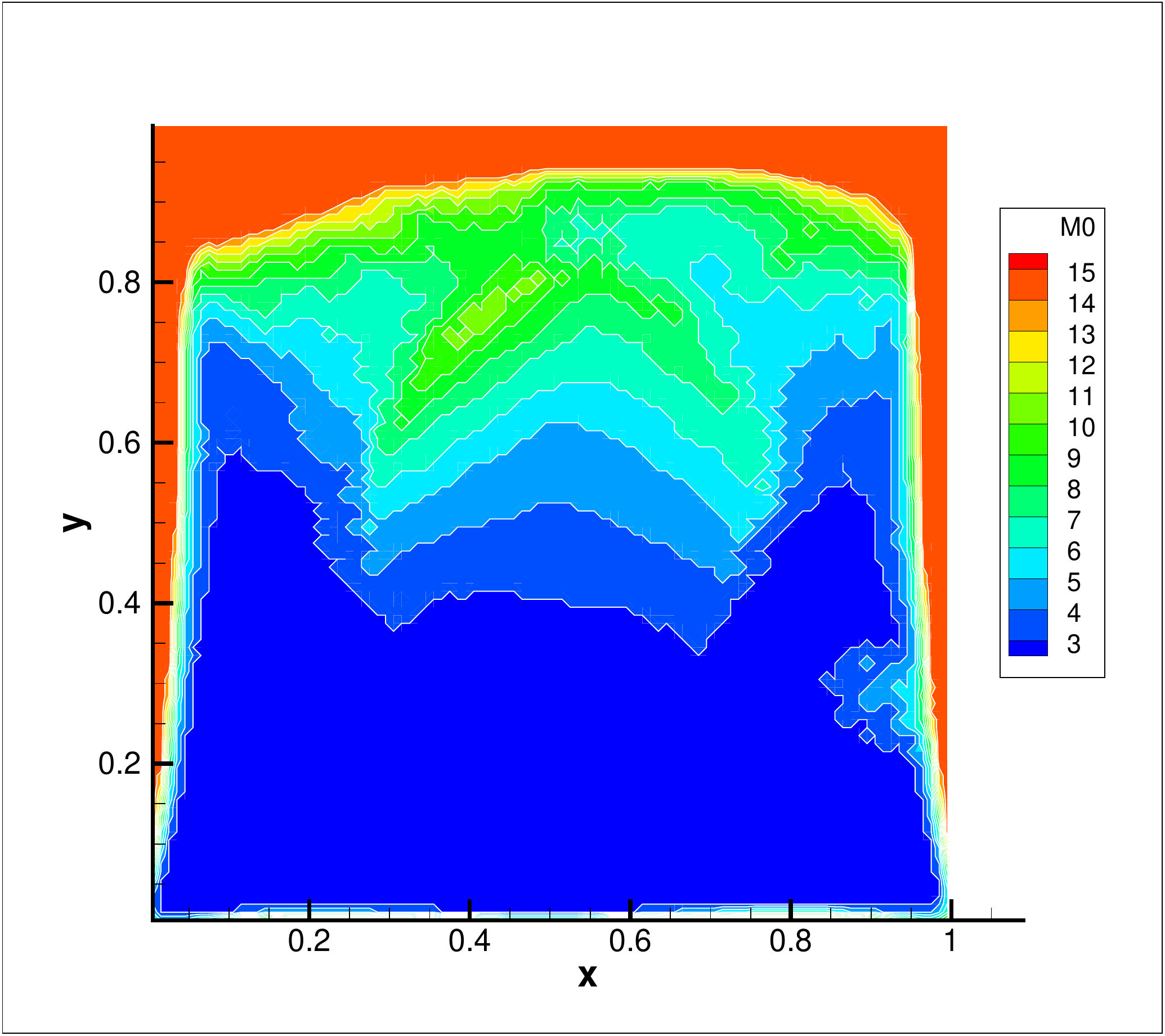}}\hfill
  \subfloat[$(\epsilon_1^2,\epsilon_2^2), t = 5$]{\includegraphics[bb=18 21 584 500,width=0.33\textwidth,clip]{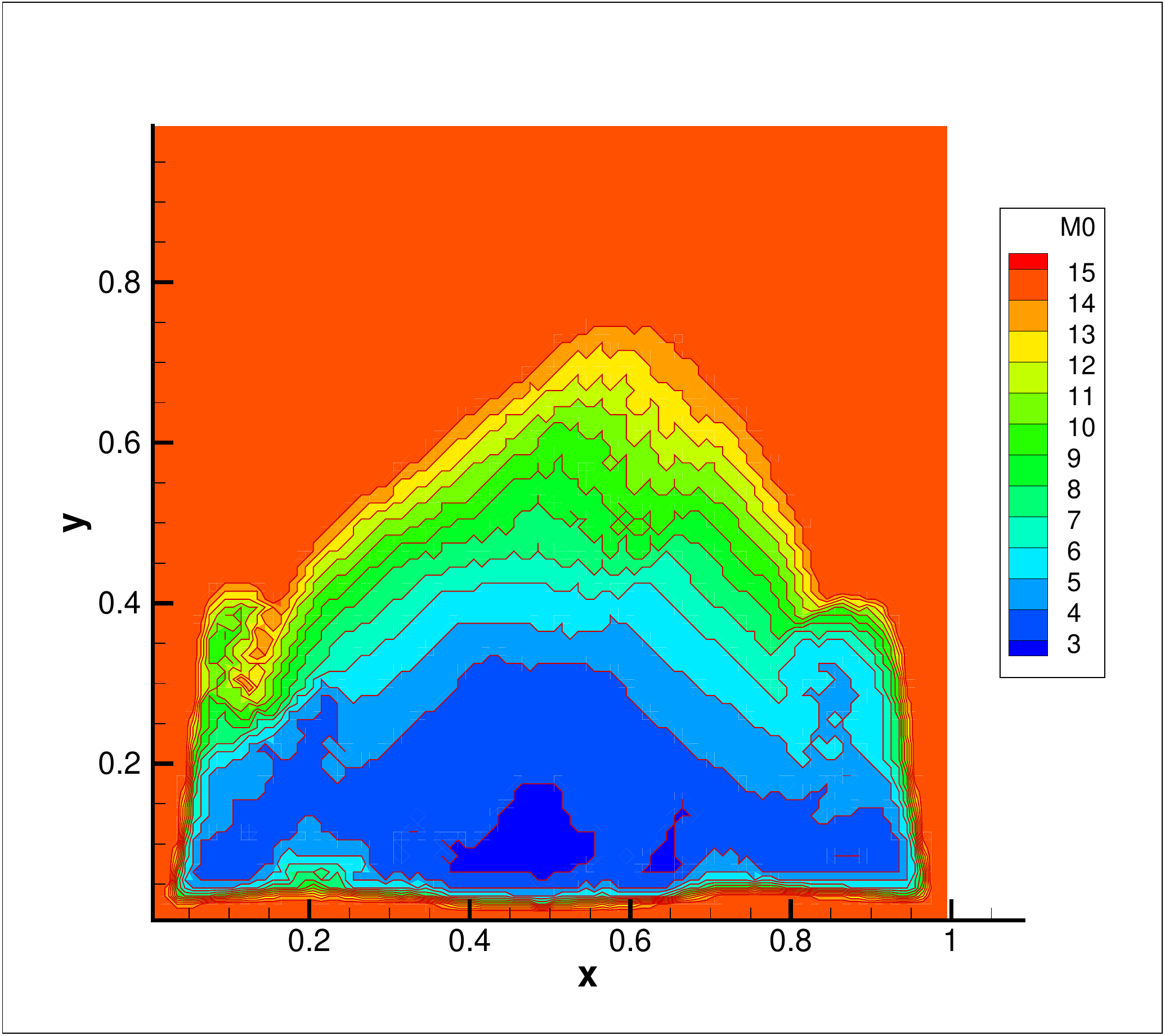}} \hfill \mbox{}\\
  \hfill \subfloat[$(\epsilon_1^1, \epsilon_2^1)$, steady state]{\includegraphics[bb=18 21 584 500,width=0.33\textwidth,clip]{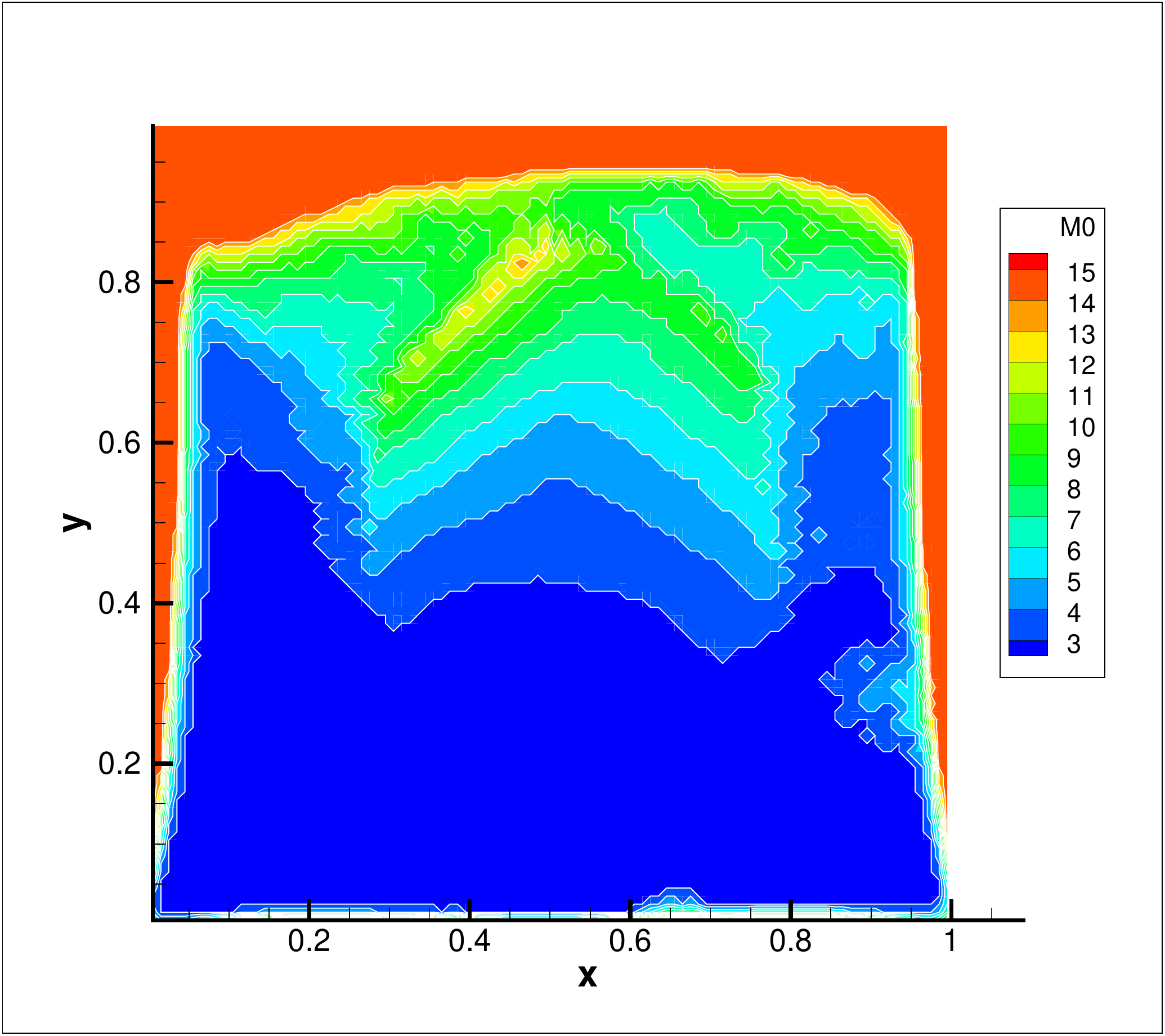}}\hfill
  \subfloat[$(\epsilon_1^2,\epsilon_2^2)$, steady state]{\includegraphics[bb=18 21 584 500,width=0.33\textwidth,clip]{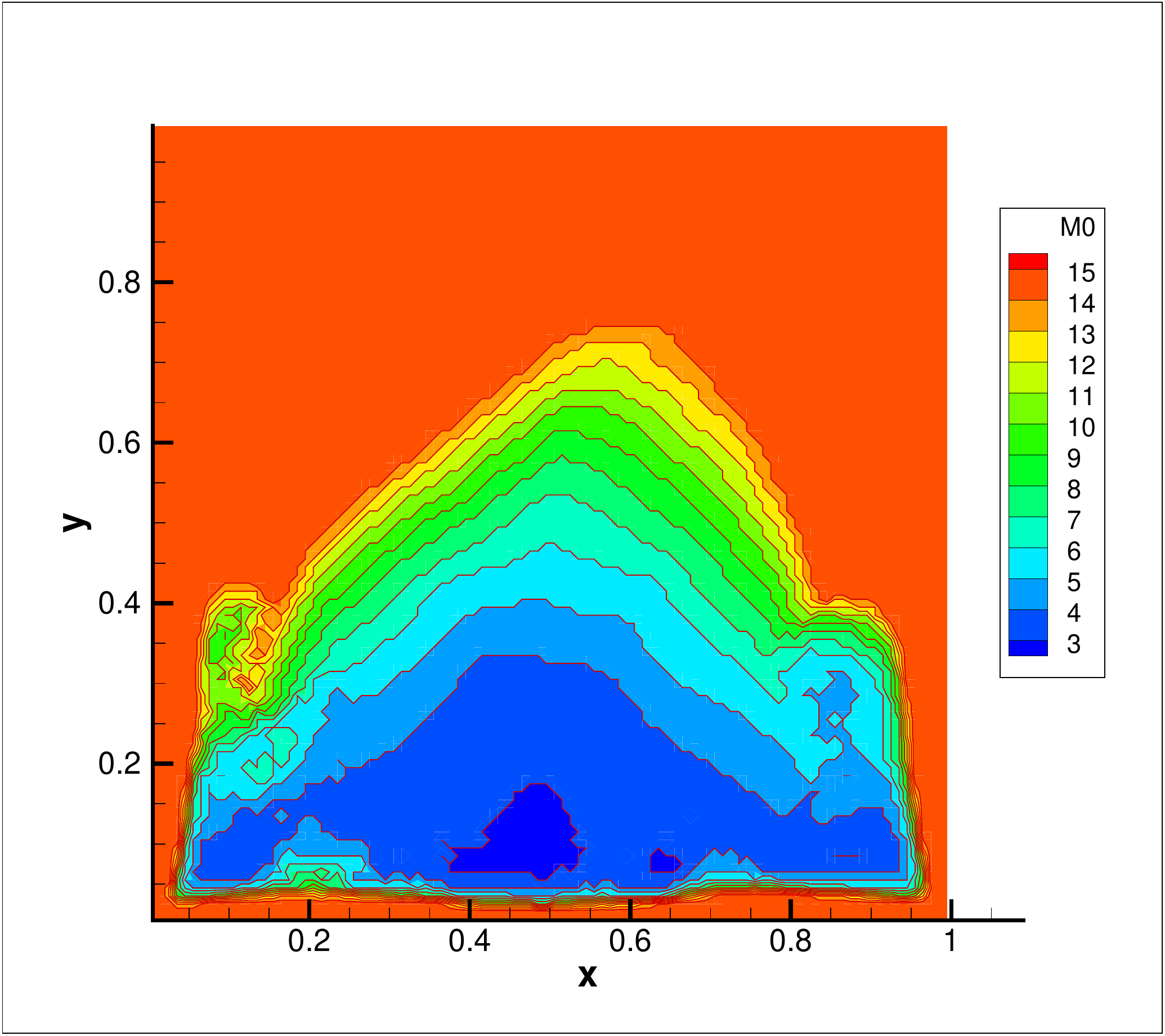}} \hfill \mbox{}
 \caption{The distribution of $M_0$ for the lid-driven cavity flow at different times. } \label{fig:ex3_M0}
\end{figure}

\subsection{Discussion on the choice of parameters}
In the previous numerical examples, one can observe that the choice of the thresholds $\epsilon_1$ and $\epsilon_2$ appears to be quite problem-dependent. This is mainly due to the different requirements of the numerical accuracy in different problems. Generally speaking, for flows with larger fluctuations such as Section \ref{sec:diffusion}, we tend to choose a larger pair of parameters since the small error is less noticeable, while for the lid-driven cavity flow in Section \ref{sec:cavity_flow}, the parameters are chosen smaller since the contour lines are more sensitive to the numerical error. In practice, when the flow structure is complicated, the flows in different areas may have different features, which may require different thresholds to obtain proper relative errors. To achieve this, a straightforward method is to add the spatial variable $x$ to both thresholds. Then the method in Section \ref{sec:adaptive} can still be applied to determine $\epsilon_1(x)$ and $\epsilon_2(x)$. Further study of this approach will be considered in our future work.

%%% Local Variables:
%%% mode: latex
%%% TeX-master: "article"
%%% End:

\section{Summary and outlook} \label{sec:conclusion}

This paper contributes to the efficient simulation of the Boltzmann equation with the quadratic collision operator. Instead of a full discretization of the binary collision term, we choose to replace part of it with the BGK simulation, and the choice of ``BGK part'' changes with the distribution function. To make proper choices adaptively, we construct our error indicator based on a novel idea that uses a cheaper linear operator to control some quadratic parts of the error term, so that even in the case where the full binary collision operator has to be used widely, our adaptive method does not slow down the computation. Our numerical simulation shows the affordability and reliability of our indicators.

The error indicator introduced in this paper is specially designed for the Burnett spectral method, while we expect that the same idea can be applied to other approaches such as the Fourier spectral method, which has lower time complexity. As the Fourier spectral method is also much cheaper for certain particular models \cite{Filbet2006}, we are exploring the possibilities of such extensions.
\appendix

\section{Choice of the parameter $\nu_{M_0}$} \label{sec:nu}
The parameter $\nu_{M_0}$ in the approximate collision term \eqref{eq:approx_coll} is chosen following \cite{Cai2015,Wang2019,Cai2020}. It can be obtained by the following steps:
\begin{itemize}
\item Set $\overline{\bu}$ and $\overline{\theta}$ to be the velocity $\bu$ and temperature $\theta$, respectively, so that
\begin{displaymath}
\mathbf{M} = \begin{pmatrix}
  \mathbf{M}^{(1)} \\ \mathbf{M}^{(2)}
\end{pmatrix} = (\rho, 0, \cdots, 0)^T.
\end{displaymath}
\item Define the linearized collision operator
\begin{displaymath}
\mathbf{L}(\mathbf{f}^{(1)}) = \mathbf{Q}_{M_0} : (\mathbf{f}^{(1)} \otimes \mathbf{M}^{(1)}).
\end{displaymath}
Since $\mathbf{M}^{(1)}$ denotes an isotropic distribution function, the operator can be expressed by
\begin{displaymath}
g_{lmn} = \sum_{n'=0}^{\lfloor (M_0-l)/2 \rfloor} a_{lnn'}^0 f_{lmn'}, \quad l = 0,1,\cdots,M_0, \quad m = -l,\cdots,l, \quad n = 0,\cdots,\lfloor (M_0-l)/2 \rfloor
\end{displaymath}
where $f_{lmn'}$ are the components of $\mathbf{f}^{(1)}$ and $g_{lmn}$ are the components of $\mathbf{L}(\mathbf{f}^{(1)})$.
\item Set $\nu_{M_0}$ to be the spectral radius of $\mathbf{L}$, which can be computed via
\begin{displaymath}
\nu_{M_0} = \max_{l=0,1,\cdots,M_0} \max \{|\lambda|: \lambda \text{ is the eigenvalue of the matrix $\mathbf{A}_l = (a_{lnn'}^0)$}\}.
\end{displaymath}
The coefficients $a_{lnn'}^0$ are given in \eqref{eq:a_lnn1^n'}, and the matrix eigenvalues are numerically computed.
\end{itemize}
When $M_0 = 2$, the matrix $\mathbf{A}_0$ is a $2\times 2$ matrix, and $\mathbf{A}_1$ and $\mathbf{A}_2$ are scalars. Due to the conservation of mass, momentum and energy, we have $\mathbf{A}_0 = 0$, $\mathbf{A}_1 = 0$. Thus, the absolute value of the only coefficient $a_{200}^0$ in $\mathbf{A}_2$ provides the value of $\nu_{M_0}$. This coefficient indicates the decay rate of the stress tensor, which is often used as the collision frequency in the BGK model.

\section{Proof of Theorem \ref{thm:pp}} \label{sec:proof}
\begin{lemma}
Given non-negative indices $l,n,n',l_1,n_1,l_2,n_2$ and integer indices $m_1 \in [-l_1, l_1]$ and $m_2 \in [-l_2, l_2]$, the integral
\begin{displaymath}
\int_{\mathbb{R}^3} \int_{\mathbb{R}^3} p_{l_1 m_1 n_1}^{\dagger}(\sqrt{2}\bh) p_{l_2 m_2 n_2}^{0\dagger} \left( \frac{\bg}{\sqrt{2}} \right) p_{l0n} \left( \bh + \frac{1}{2} \bg \right) p_{00n'} \left( \bh - \frac{1}{2} \bg \right) \omega\left(\bh + \frac{1}{2} \bg\right) \omega \left( \bh - \frac{1}{2} \bg \right) \,\mathrm{d}\bg \,\mathrm{d}\bh
\end{displaymath}
is nonzero only if $l_1 + l_2 + 2(n_1 + n_2) = l + 2(n+n')$ and $m_1 + m_2 = 0$.
\end{lemma}
This conclusion can be found in \cite[eqs. (112)(114)]{kumar1966}.
\begin{proof}[Proof of Theorem \ref{thm:pp}]
Due to the orthogonality of the polynomials $p_{lmn}$, we know that
\begin{equation} \label{eq:A_integral}
\begin{split}
A_{lnn'}^{l_1 l_2 m_2 n_2} = \int_{\mathbb{R}^3} \int_{\mathbb{R}^3} p_{l_1 m_1 n_1}^{\dagger}(\sqrt{2}\bh) p_{l_2 m_2 n_2}^{0\dagger} \left( \frac{\bg}{\sqrt{2}} \right) p_{l0n} \left( \bh + \frac{1}{2} \bg \right) p_{00n'} \left( \bh - \frac{1}{2} \bg \right)  \\
\times \omega\left(\bh + \frac{1}{2} \bg\right) \omega \left( \bh - \frac{1}{2} \bg \right) \,\mathrm{d}\bg \,\mathrm{d}\bh,
\end{split}
\end{equation}
where $m_1 = -m_2$ and $n_1 = l-l_1-l_2+2(n+n'-n_2)$.
For simplicity, we assume that $\overline{\bu} = 0$, and in the case of nonzero $\overline{\bu}$, the result can be obtained by translation. To derive the recurrence relation of $A_{lnn'}^{l_1l_2 m_2 n_2}$, we use the recurrence relation of Laguerre polynomials to get
\begin{displaymath}
p_{00,n'+1}\left(\bh - \frac{1}{2}\bg \right) = -\frac{1}{2\sqrt{(n+1)(n+3/2)}} \overline{\theta}^{-1} \left\| \bh - \frac{1}{2} \bg \right\|^2 p_{00n'} \left( \bh - \frac{1}{2} \bg \right) + (\text{Lower degree polynomials}),
\end{displaymath}
where ``lower degree polynomials'' refers to the polynomials of $\bg$ and $\bh$ of degree less than $2(n'+1)$. Due to the orthogonality of $p_{lmn}$, these terms will vanish when calculating the integral \eqref{eq:A_integral} with $n'$ replaced by $n'+1$:
\begin{equation} \label{eq:Alnnp1}
\begin{split}
A_{ln,n'+1}^{l_1 l_2 m_2 n_2} &= -\frac{1}{2\sqrt{(n+1)(n+3/2)}} \overline{\theta}^{-1} \int_{\mathbb{R}^3} \int_{\mathbb{R}^3} \left( h^2 + \frac{1}{4} g^2 - \bh \cdot \bg \right) p_{l_1 m_1 n_1}^{\dagger}(\sqrt{2}\bh) p_{l_2 m_2 n_2}^{0\dagger} \left( \frac{\bg}{\sqrt{2}} \right) \\
& \qquad \times p_{l0n} \left( \bh + \frac{1}{2} \bg \right) p_{00n'} \left( \bh - \frac{1}{2} \bg \right)
\omega\left(\bh + \frac{1}{2} \bg\right) \omega \left( \bh - \frac{1}{2} \bg \right) \,\mathrm{d}\bg \,\mathrm{d}\bh.
\end{split}
\end{equation}
Using
\begin{equation} \label{eq:h2_g2}
\begin{split}
& \overline{\theta}^{-1} \left( h^2 + \frac{1}{4} g^2 \right) p_{l_1 m_1 n_1}(\sqrt{2} \bh) p_{l_2 m_2 n_2}^0 \left( \frac{\bg}{\sqrt{2}} \right) \\
={} & -\sum_{k=1}^2 \sqrt{n_k(n_k+l_k+1/2)} p_{l_1,m_1,n_1-\delta_{1k}} (\sqrt{2}\bh) p_{l_2,m_2,n_2-\delta_{2k}}^0 \left( \frac{\bg}{\sqrt{2}} \right) + (\text{Higher degree polynomials})
\end{split}
\end{equation}
and
\begin{equation} \label{eq:hg}
\begin{split}
& \frac{1}{2} \overline{\theta}^{-1} (\bh \cdot \bg) p_{l_1 m_1 n_1}(\sqrt{2} \bh) p_{l_2 m_2 n_2}^0 \left( \frac{\bg}{\sqrt{2}} \right) \\
={} & \sum_{\mu=-1}^1 (-1)^{\mu} \left[ \sqrt{l_1 + n_1 + 1/2} \gamma_{l_1,m_1-\mu}^{\mu} p_{l_1-1,m_1-\mu,n_1}(\sqrt{2}\bh) - (-1)^{\mu} \sqrt{n_1} \gamma_{-l_1-1,m_1-\mu}^{\mu} p_{l_1+1,m_1-\mu,n_1-1}(\sqrt{2}\bh) \right] \\
& \qquad \times \left[ \sqrt{l_2 + n_2 + 1/2} \gamma_{l_2,m_2+\mu}^{-\mu} p_{l_2-1,m_2+\mu,n_2}(\sqrt{2}\bh) - (-1)^{\mu} \sqrt{n_2} \gamma_{-l_2-1,m_2+\mu}^{-\mu} p_{l_2+1,m_2+\mu,n_2-1}(\sqrt{2}\bh) \right] \\
& + (\text{Higher degree polynomials}).
\end{split}
\end{equation}
We refer the readers to \cite[Appendix B]{Cai2015} for the derivation of these equations. In the above two equations, ``higher degree polynomials'' refer to the orthogonal polynomials of $\bg$ and $\bh$ whose degrees are higher than $l+2(n+n')-1$, and such terms will vanish after substituting \eqref{eq:h2_g2} and \eqref{eq:hg} into \eqref{eq:Alnnp1}. The recurrence relation \eqref{eq:A_recur} can be obtained by such substitution and using the properties
\begin{displaymath}
m_1 + m_2 = 0, \qquad \gamma_{l,m}^{\mu} = \gamma_{l,-m}^{-\mu}. \qedhere
\end{displaymath}
\end{proof}

\bibliographystyle{amsplain}
\bibliography{article}

\providecommand{\bysame}{\leavevmode\hbox to3em{\hrulefill}\thinspace}
\providecommand{\MR}{\relax\ifhmode\unskip\space\fi MR }
\providecommand{\MRhref}[2]{%
  \href{http://www.ams.org/mathscinet-getitem?mr=#1}{#2}
}
\providecommand{\href}[2]{#2}
\begin{thebibliography}{10}

\bibitem{Abdelmalik2017}
M.R.A. Abdelmalik and E.H. {van Brummelen}, \emph{Error estimation and adaptive
  moment hierarchies for goal-oriented approximations of the {B}oltzmann
  equation}, Comput. Methods Appl. Mech. Engrg. \textbf{325} (2017), 219--239.

\bibitem{Alekseenko2014}
A.~Alekseenko and E.~Josyula, \emph{Deterministic solution of the spatially
  homogeneous {B}oltzmann equation using discontinuous {G}alerkin
  discretizations in the velocity space}, J. Comput. Phys. \textbf{272} (2014),
  170--188.

\bibitem{Baranger2014}
C.~Baranger, J.~Claudel, N.~H{\'e}rouard, and L.~Mieussens, \emph{Locally
  refined discrete velocity grids for stationary rarefied flow simulations}, J.
  Comput. Phys. \textbf{257} (2014), 572--593.

\bibitem{Bennoune2008}
M.~Bennoune, M.~Lemou, and L.~Mieussens, \emph{Uniformly stable numerical
  schemes for the {B}oltzmann equation preserving the compressible
  {N}avier-{S}tokes asymptotics}, J. Comput. Phys. \textbf{227} (2008),
  3781--3803.

\bibitem{Bobylev1997}
A.~V. Bobylev and S.~Rjasanow, \emph{Difference scheme for the {B}oltzmann
  equation based on the fast {F}ourier transform}, Eur. J. Mech. B Fluids
  \textbf{16} (1997), 293--306.

\bibitem{Cai2020}
Z.~Cai, Y.~Fan, and Y.~Wang, \emph{{B}urnett spectral method for the spatially
  homogeneous {B}oltzmann equation}, Comput. Fluids \textbf{200} (2020),
  104456.

\bibitem{Cai2015}
Z.~Cai and M.~Torrilhon, \emph{Approximation of the linearized {B}oltzmann
  collision operator for hard-sphere and inverse-power-law models}, J. Comput.
  Phys. \textbf{295} (2015), 617--643.

\bibitem{Cai2018}
\bysame, \emph{Numerical simulation of microflows using moment method with
  linearized collision operator}, J. Sci. Comput. \textbf{74} (2018), 336--374.

\bibitem{Cai2020r}
Z.~Cai and Y.~Wang, \emph{Regularized 13-moment equations for inverse power law
  models}, J. Fluid Mech. \textbf{894} (2020), A12.

\bibitem{Cercignani2006}
C.~Cercignani, \emph{Slow rarefied flows: Theory and application to
  micro-electro-mechanical systems}, Progress in Mathematical Physics, vol.~41,
  Birkh{\"a}user, 2006.

\bibitem{Chen2012}
S.~Chen, K.~Xu, C.~Lee, and Q.~Cai, \emph{A unified gas kinetic scheme with
  moving mesh and velocity space adaptation}, J. Comput. Phys. \textbf{231}
  (2012), no.~20, 6643--6664.

\bibitem{Degond2012}
P.~Degond and G.~Dimarco, \emph{Fluid simulations with localized boltzmann
  upscaling by direct simulation {M}onte-{C}arlo}, J. Comput. Phys.
  \textbf{231} (2012), 2414--2437.

\bibitem{Degond2005}
P.~Degond, S.~Jin, and L.~Mieussens, \emph{A smooth transition model between
  kinetic and hydrodynamic equations}, J. Comput. Phys. \textbf{209} (2005),
  no.~2, 665--694.

\bibitem{Degond2004}
P.~Degond, L.~Pareschi, and G.~Russo, \emph{Modeling and computational methods
  for kinetic equations}, Modeling and Simulation in Science, Engineering and
  Technology, Birkh\"auser Basel, 2004.

\bibitem{Dimarco2018}
G.~Dimarco, R.~Loub{\`e}re, J.~Narski, and T.~Rey, \emph{An efficient numerical
  method for solving the {B}oltzmann equation in multidimensions}, J. Comput.
  Phys. \textbf{353} (2018), 46--81.

\bibitem{Filbet2006}
F.~Filbet, C.~Mouhot, and L.~Pareschi, \emph{Solving the {B}oltzmann equation
  in {$N \log_2 N$}}, SIAM J. Sci. Compute. \textbf{28} (2006), no.~3,
  1029--1053.

\bibitem{Filbet2015}
F.~Filbet, L.~Pareschi, and T.~Rey, \emph{On steady-state preserving spectral
  methods for homogeneous {B}oltzmann equations}, C. R. Acad. Sci. Paris, Ser.
  I \textbf{353} (2015), 309--314.

\bibitem{Rey2015}
F.~Filbet and T.~Rey, \emph{A hierarchy of hybrid numerical methods for
  multiscale kinetic equations}, SIAM J. Sci. Comput. \textbf{37} (2015),
  no.~3, A1218--A1247.

\bibitem{Filbet2018}
F.~Filbet and T.~Xiong, \emph{A hybrid discontinuous {G}alerkin scheme for
  multi-scale kinetic equations}, J. Comput. Phys. \textbf{372} (2018),
  841--863.

\bibitem{Gamba2017}
I.~M. Gamba, J.~R. Haack, C.~D. Hauck, and J.~Hu, \emph{A fast spectral method
  for the {B}oltzmann collision operator with general collision kernels}, SIAM
  J. Sci. Comput. \textbf{39} (2017), no.~14, B658--B674.

\bibitem{Gamba2018}
I.~M. Gamba and S.~Rjasanow, \emph{{G}alerkin-{P}etrov approach for the
  {B}oltzmann equation}, J. Comput. Phys. \textbf{366} (2018), 341--365.

\bibitem{Grad1949}
H.~Grad, \emph{On the kinetic theory of rarefied gases}, Comm. Pure Appl. Math.
  \textbf{2} (1949), no.~1, 331--407.

\bibitem{HLL}
A.~Harten, P.~D. Lax, and B.~Van Leer, \emph{On upstream differencing and
  {G}odunov-type schemes for hyperbolic conservation laws}, SIAM Review
  \textbf{25} (1983), no.~1, 35--61.

\bibitem{Hu2020b}
Z.~Hu and Z.~Cai, \emph{{B}urnett spectral method for high-speed rarefied gas
  flows}, SIAM J. Sci. Comput. \textbf{42} (2020), no.~5, B1193--B1226.

\bibitem{Hu2020}
Z.~Hu, Z.~Cai, and Y.~Wang, \emph{Numerical simulation of microflows {H}ermite
  spectral methods}, SIAM J. Sci. Comput. \textbf{42} (2020), B105--B134.

\bibitem{Ikenberry1956}
E.~Ikenberry and C.~Truesdell, \emph{On the pressures and the flux of energy in
  a gas according to {M}axwell's kinetic theory, {I}}, J. Rat. Mech. Anal.
  \textbf{5} (1956), no.~1, 1--54.

\bibitem{Jaiswal2019}
S.~Jaiswal, A.~A. Alexeenko, and J.W. Hu, \emph{A discontinuous {G}alerkin fast
  spectral method for the full {B}oltzmann equation with general collision
  kernels}, J. Comput. Phys. \textbf{378} (2019), 178--208.

\bibitem{John2010}
B.~John, X.~J. Gu, and D.~R. Emerson, \emph{Investigation of heat and mass
  transfer in a lid-driven cavity under nonequilibrium flow conditions}, Numer.
  Heat Tr. B-fund. \textbf{58} (2010), 287--303.

\bibitem{Kitzler2019}
G.~Kitzler and J.~Sch{\"o}berl, \emph{A polynomial spectral method for the
  spatially homogeneous boltzmann equation}, SIAM J. Sci. Comput. \textbf{41}
  (2019), no.~1, B27--B49.

\bibitem{Kolobov2007}
V.I. Kolobov, R.R. Arslanbekov, V.V. Aristov, A.A. Frolova, and S.A. Zabelok,
  \emph{Unified solver for rarefied and continuum flows with adaptive mesh and
  algorithm refinement}, J. Comput. Phys \textbf{223} (2007), no.~2, 589--608.

\bibitem{kumar1966}
K.~Kumar, \emph{Polynomial expansions in kinetic theory of gases}, Ann. Phys.
  \textbf{37} (1966), no.~1, 113--141.

\bibitem{Lebedev1975}
V.I. Lebedev, \emph{Values of the nodes and weights of ninth to seventeenth
  order gauss-markov quadrature formulae invariant under the octahedron group
  with inversion}, USSR Comput. Math. Math. Phys. \textbf{15} (1975), no.~1,
  44--51.

\bibitem{Leveque2002}
R.~J. Leveque, \emph{Finite volume methods for hyperbolic problems}, Cambridge,
  2002.

\bibitem{Levermore1998}
C.~D. Levermore, W.~J. Morokoff, and B.~T. Nadiga, \emph{Moment realizability
  and the validity of the {N}avier–{S}tokes equations for rarefied gas
  dynamics}, Phys. Fluids \textbf{10} (1998), no.~12, 3214--3226.

\bibitem{Liu2016}
C.~Liu, K.~Xu, Q.~Sun, and Q.~Cai, \emph{A unified gas-kinetic scheme for
  continuum and rarefied flows {IV}: {F}ull {B}oltzmann and model equations},
  J. Comput. Phys. \textbf{314} (2016), 305--340.

\bibitem{Mouhot2004}
C.~Mouhot and C.~Villani, \emph{Regularity theory for the spatially homogeneous
  {B}oltzmann equation with cut-off}, Arch. Rat. Mech. Anal. \textbf{173}
  (2004), 169--212.

\bibitem{Pareschi1996}
L.~Pareschi and B.~Perthame, \emph{A {F}ourier spectral method for homogeneous
  {B}oltzmann equations}, Transport Theory Statist. Phys. \textbf{25} (1996),
  369--382.

\bibitem{Shakhov1968}
E.~M. Shakhov, \emph{Generalization of the {K}rook kinetic relaxation
  equation}, Fluid Dyn. \textbf{3} (1968), no.~5, 95--96.

\bibitem{Wang2003}
W.-L. Wang and I.~D. Boyd, \emph{Predicting continuum breakdown in hypersonic
  viscous flows}, Phys. Fluids \textbf{15} (2003), no.~1, 91--100.

\bibitem{Wang2019}
Y.~Wang and Z.~Cai, \emph{Approximation of the {B}oltzmann collision operator
  based on hermite spectral method}, J. Comput. Phys. \textbf{397} (2019),
  108815.

\end{thebibliography}
\end{document}